\documentclass[review,onefignum,onetabnum]{siamonline171218}

\usepackage{lipsum}
\usepackage{amsfonts}
\usepackage{graphicx}
\usepackage{diagbox,color,xcolor}
\usepackage{epstopdf}
\usepackage{algorithmic}
\usepackage{multirow,booktabs,tabularx,cases,threeparttable}
\ifpdf
  \DeclareGraphicsExtensions{.eps,.pdf,.png,.jpg}
\else
  \DeclareGraphicsExtensions{.eps}
\fi
\usepackage{hyperref}
\usepackage{enumitem}
\setlist[enumerate]{leftmargin=.5in}
\setlist[itemize]{leftmargin=.5in}

\newcommand{\y}{\color{black}}
\newcommand{\h}{\color{black}}
\newsiamremark{remark}{Remark}
\newsiamremark{hypothesis}{Hypothesis}
\newsiamremark{assumption}{Assumption}
\crefname{hypothesis}{Hypothesis}{Hypotheses}
\newsiamthm{claim}{Claim}
\DeclareMathOperator*{\argmin}{arg\,min}
\usepackage{graphicx}
\usepackage{tikz}
\usepackage{subfigure}
\usetikzlibrary{positioning}
\usetikzlibrary{calc} 
\usepackage{ifthen}
\usepackage{xparse} 

 \ExplSyntaxOn
\NewDocumentCommand{\aspectratio}{smo}
 {
  \hbox_set:Nn \l_tmpa_box {\includegraphics{#2}}
  \IfNoValueTF{#3}
   {
    \__student_aspectratio:nn { \box_wd:N \l_tmpa_box } { \box_ht:N \l_tmpa_box }
   }
   {
    \IfBooleanTF{#1}{ \tl_gset:Nx } { \tl_set:Nx } #3
     {
      \__student_aspectratio:nn { \box_wd:N \l_tmpa_box } { \box_ht:N \l_tmpa_box }
     }
   }
 }

\cs_new:Nn \__student_aspectratio:nn
 {
  \fp_eval:n {round( #1 / #2 , 5)}
 }
\ExplSyntaxOff

\newcommand{\neworrenewcommand}[1]{\providecommand{#1}{}\renewcommand{#1}}

\newcommand{\zoomincludgraphic}[9]{
    \neworrenewcommand{\ffoo}[5]{
\begin{tikzpicture}[x=#1, y=#1, font=\footnotesize]
\aspectratio{#2}[\imsizeratio] 

  \node[anchor = south east, inner sep=0] (image) at (1,0) {\includegraphics[width=#1]{#2}};
	    \coordinate (viewport lower left) at (#3,#4/\imsizeratio);  
	    \coordinate(viewport upper right) at (#5,#6/\imsizeratio);  
        \draw[##5, line width = ##4 pt] (viewport lower left) rectangle (viewport upper right);
 
     \pgfmathsetmacro{\multone}{#5-#3}
     \pgfmathsetmacro{\multtwo}{#6/\imsizeratio-#4/\imsizeratio}
     
     \ifthenelse{\equal{#9}{bottom_left} }{ 
	      \node[anchor=north, draw= ##3, inner sep=0pt, line width = ##2 pt,outer sep=0pt] (zoomPart) at (\multone*#7/2+##2/345*1.333, \multtwo*#7+##2/345*1.333) {
	       \scalebox{#7}{\tikz{
	         \clip (#3,#4/\imsizeratio) rectangle (#5,#6/\imsizeratio);
	           
	         \node[anchor=south east, inner sep=0] at (1,0) {\includegraphics[width=#1]{#2}}; 
	         }}};
	   \ifthenelse{\equal{##1}{line_connection_on} }{ 
		  \draw[red, dashed] (viewport upper right|-viewport lower left) -- (zoomPart.north east); 
		  \draw[red, dashed] (viewport lower left) -- (zoomPart.north west);
		   }{}
	       
	 }{}

     \ifthenelse{\equal{#9}{bottom_right} }{ 
	      \node[anchor=north, draw= ##3, inner sep=0pt, line width = ##2 pt,outer sep=0pt] (zoomPart) at (1-\multone*#7/2-##2/345*1.333, \multtwo*#7 + ##2/345*1.333) {
	       \scalebox{#7}{\tikz{
			 \clip (#3,#4/\imsizeratio) rectangle (#5,#6/\imsizeratio);
	         \node[anchor=south east, inner sep=0] at (1,0) {\includegraphics[width=#1]{#2}}; 
	         }}%
	       };
		\ifthenelse{\equal{##1}{line_connection_on} }{ 
			  \draw[red, dashed] (viewport upper right|-viewport lower left) -- (zoomPart.south west); 
			  \draw[red, dashed] (viewport upper right) -- (zoomPart.north west);
		   }{}
     }{}  
       
    \ifthenelse{\equal{#9}{up_right} }{  
	      \node[anchor=north, draw= ##3, inner sep=0pt, line width = ##2pt, outer sep=0pt] (zoomPart) at (1-\multone*#7/2-##2/345*1.333,1/\imsizeratio-##2/345*1.333) {
	       \scalebox{#7}{\tikz{
	          \clip (#3,#4/\imsizeratio) rectangle (#5,#6/\imsizeratio);
	          \node[anchor=south east, inner sep=0] at (1,0) {\includegraphics[width=#1]{#2}}; 
	         }}%
	         
	       };
	   \ifthenelse{\equal{##1}{line_connection_on} }{ 
		  \draw[red, dashed] (viewport lower left|-viewport upper right) -- (zoomPart.south west);
		  \draw[red, dashed] (viewport upper right) -- (zoomPart.south east);
		   }{}
     }{}

     \ifthenelse{\equal{#9}{up_left} }{ 
	      \node[anchor=north, draw= ##3, inner sep=0pt, line width = ##2pt,outer sep=0pt] (zoomPart) at (\multone*#7/2+##2/345*1.333, 1/\imsizeratio-##2/345*1.333) {
	       \scalebox{#7}{\tikz{
	         \clip (#3,#4/\imsizeratio) rectangle (#5,#6/\imsizeratio);
	           
	          \node[anchor=south east, inner sep=0] at (1,0) {\includegraphics[width=#1]{#2}}; 
	         }}};
		   \ifthenelse{\equal{##1}{line_connection_on} }{ 
			  \draw[red, dashed] (viewport lower left|-viewport upper right) -- (zoomPart.south west);
			  \draw[red, dashed] (viewport upper right) -- (zoomPart.south east);
			   }{}
	     }{}

  	\ifthenelse{\equal{#8}{help_grid_on} }{ 
           \begin{scope}[
                x={(image.south east)},
                y={(image.north west)},
                font=\footnotesize,
                help lines,
                overlay
            ]
            
            \draw[help lines, xstep=.1,ystep=.1,overlay] (0,0) grid (1,1);
            \foreach \x in {0,1,...,9} { 
                \node[anchor=north] at (\x/10,0) {0.\x}; 
            }
            \foreach \y in {0,1,...,9} {
                \node[anchor=east] at (0,\y/10) {0.\y};
            }
        \end{scope}    
	}{}  
   
\end{tikzpicture}

    }
    \ffoo
}
\nolinenumbers

\headers{Extrapolated Plug-and-Play Three-Operator Splitting Methods}{Z. Wu, C. Huang, and T. Zeng}

\title{Extrapolated Plug-and-Play Three-Operator Splitting Methods for Nonconvex Optimization with Applications to Image Restoration\thanks{Submitted to the editors October 22, 2023.
\funding{This work was supported by Grant NSFC/RGC N\_CUHK 415/19, Grant ITF ITS/173/22FP, Grant RGC 14300219, 14302920, 14301121, and CUHK Direct Grant for Research, the National Natural Science Foundation of China Grant 12001286, and the China Postdoctoral Science Foundation Grant 2022M711672.}}}

\author{Zhongming Wu\thanks{Co-first  author. School of Management Science and Engineering, Nanjing University of Information Science and Technology, Nanjing, China (\email{wuzm@nuist.edu.cn}).}
\and Chaoyan Huang\thanks{Co-first  author. Department of  Mathematics,  The  Chinese  University of  Hong  Kong,  Shatin,  Hong  Kong, China  (\email{cyhuang@math.cuhk.edu.hk}).}
\and Tieyong Zeng\thanks{Corresponding author. Department  of  Mathematics,  The  Chinese  University  of  Hong  Kong,  Shatin,  Hong  Kong, China (\email{zeng@math.cuhk.edu.hk}).
}}

\usepackage{amsopn}

\makeatletter
\newcommand*{\addFileDependency}[1]{
  \typeout{(#1)}
  \@addtofilelist{#1}
  \IfFileExists{#1}{}{\typeout{No file #1.}}
}
\makeatother

\begin{document}
\maketitle

\begin{abstract}
This paper investigates the convergence properties and applications of the three-operator splitting method, also known as Davis-Yin splitting (DYS) method, integrated with extrapolation and Plug-and-Play (PnP) denoiser within a nonconvex framework. 
We first propose an extrapolated DYS method to effectively solve a class of structural nonconvex optimization problems that involve minimizing the sum of three possible nonconvex functions. 
Our approach provides an algorithmic framework that encompasses both extrapolated forward-backward splitting and extrapolated Douglas-Rachford splitting methods.
To establish the convergence of the proposed method, we rigorously analyze its behavior based on the Kurdyka-Łojasiewicz property, subject to some tight parameter conditions.
Moreover, we introduce two extrapolated PnP-DYS methods with convergence guarantee, where the traditional regularization prior is replaced by a gradient step-based denoiser. 
This denoiser is designed using a differentiable neural network and can be reformulated as the proximal operator of a specific nonconvex functional.
We conduct extensive experiments on image deblurring and image super-resolution problems, where our results showcase the advantage of the extrapolation strategy and the superior performance of the learning-based model that incorporates the PnP denoiser in terms of achieving high-quality recovery images. 
\end{abstract}

\begin{keywords}
Plug-and-Play, three-operator splitting method, nonconvex optimization, denoising prior, convergence guarantee
\end{keywords}

\begin{AMS}
  90C26, 90C30, 90C90, 65K05
\end{AMS}

\section{Introduction}
In this paper, we consider the following type of structural nonconvex optimization problem:
\begin{equation}\label{general}
    \min_{{\bf x}\in\mathbb{R}^n}  F({\bf x})= f_1({\bf x})+f_2({\bf x})+h({\bf x}),
\end{equation}
where $f_1 $ and $h $ are continuously differentiable and potentially nonconvex, and $f_2 $ is a proper closed (possibly nonconvex) function. The model \cref{general} captures  a rich number of applications in fields of deep learning, signal and image processing, and statistical learning, see e.g., \cite{bian2021three,combettes2021fixed,condat2023proximal,jain2017non,yang2011alternating,yin2015minimization,yurtsever2021three}. In particular, the smooth term includes the least squares or logistic loss functions, and the nonsmooth term can be represented as
 regularizers, e.g., to promote potential behavior such as sparsity and low-rank.

Splitting methods, which fully leverage the inherent separable structure, is a class of popular and state-of-the-art approaches for effectively addressing structural optimization problems. A generic way to solve the type of problem \cref{general} is the three-operator splitting method, also known as Davis-Yin splitting (DYS) { method which was} first studied in \cite{davis2017three} for convex optimization, i.e., all the involved functions in \cref{general} are convex. The concrete iterative scheme of DYS method can be read as  
\begin{equation}\label{DYS}
\left\{\begin{array}{l}
 {\bf y}^{k+1} \in \argmin\limits _{{\bf y}\in\mathbb{R}^n}\left\{f_1({\bf y})+\frac{1}{2 \gamma}\left\|{\bf y}-{\bf x}^k\right\|^2\right\}, \\ [0.3cm]
 {\bf z}^{k+1} \in \argmin\limits _{{\bf z}\in\mathbb{R}^n}\left\{f_2({\bf z})+\frac{1}{2 \gamma}\left\|{\bf z}-\left(2 {\bf y}^{k+1}-\gamma \nabla h\left({\bf y}^{k+1}\right)-{\bf x}^k\right)\right\|^2\right\}, \\ [0.3cm]
 {\bf x}^{k+1}={\bf x}^k+\left({\bf z}^{k+1}-{\bf y}^{k+1}\right) ,
\end{array}\right.      
 \end{equation}
where $\gamma>0$ is a proximal parameter. DYS method \cref{DYS} includes two proximal subproblems with respect to ${\bf y}$ and ${\bf z}$, which extends various previous splitting schemes such as the forward-backward splitting (FBS) method \cite{attouch2013convergence}, Douglas-Rachford splitting (DRS) method \cite{li2016douglas}, alternating direction method of multipliers (ADMM) algorithm \cite{han2022survey} and the generalized forward-backward splitting method \cite{raguet2013generalized}. Later on, some variants and extensions of DYS method are explored for convex optimization \cite{latafat2017asymmetric,ryu2020operator,salim2022dualize,tang2022preconditioned}. However, for the nonconvex setting as that in \cref{general}, convergence properties of the DYS method \cref{DYS} are less understood. 
In contrast, the FBS method and DRS method, two special cases of the DYS method, have been well studied for nonconvex optimization, see e.g., \cite{attouch2013convergence,li2016douglas,themelis2020douglas}. { Indeed, splitting methods are widely employed in image processing because numerous problems in image restoration can be addressed through variational methods. The resulting image is obtained as a minimizer of a suitable energy functional, typically exhibiting a separable structure. For recent applications in this field, we refer to \cite{deng2019new, liu2022operator, setzer2011operator, tang2022preconditioned}.}

Another captivating and intriguing topic within the realm of splitting methods is the incorporation of acceleration techniques. Since the pioneering work of Polyak \cite{polyak1964some} on the heavy-ball method approach to gradient descent, extrapolation, as well as named inertial strategy, has been adapted to various optimization schemes to achieve accelerated convergence. Notable examples include the accelerated proximal point algorithm \cite{chen2015general} for variational inequality problems and the accelerated FBS \cite{attouch2014dynamical,villa2013accelerated,beck2009fast} for convex optimization. Over the past decade, the extrapolation technique has also been extended to various splitting methods for solving nonconvex optimization problems and expediting convergence based on Kurdyka–{\L}ojasiewicz framework (see \Cref{def2.2}), as demonstrated in studies such as \cite{ochs2014ipiano,le2020inertial,wang2023generalized,liang2016multi,pock2016inertial,wu2019general,wu2021inertial,phan2023inertial}. In this paper, our first focus is to investigate the convergence properties of the DYS method \cref{DYS} when combined with extrapolation technique for solving \cref{general}. This endeavor will result in the development of a versatile framework encompassing extrapolated (or named inertial) FBS and extrapolated DRS methods as specialized schemes tailored for nonconvex optimization. 

Recently, Plug-and-Play (PnP) methods combine splitting algorithm with denoising priors are widely used in solving many practical problems \cite{gavaskar2021plug,wei2022tfpnp,wu2023retinex,li2023spherical}. PnP method offers a concise yet adaptable approach for integrating statistical priors into a problem, eliminating the requirement to explicitly construct an objective function.
The first PnP method was the PnP-ADMM developed in \cite{venkatakrishnan2013plug} to address a range of imaging problems, which simply replaces the proximal subproblem with the denoising prior. Since then, many PnP-based methods such as PnP-FBS \cite{sreehari2016plug,tirer2018image}, PnP-DRS \cite{buzzard2018plug,hurault2022proximal} and PnP-primal dual \cite{ono2017primal} approaches, reported empirical success on a large variety of applications, but with scarce theoretical guarantees. 
In several recent studies, the convergence of PnP methods has been achieved through the utilization of contractive fixed-point iterations. For example, the convergence of various proximal algorithms has been established by assuming properties such as denoiser averaging \cite{sun2019online}, firm nonexpansiveness \cite{sun2021scalable}, or simple nonexpansiveness \cite{liu2021recovery,reehorst2018regularization}.
However, it is important to note that off-the-shelf deep denoisers often lack 1-Lipschitz continuity, which is equivalent to nonexpansiveness. The imposition of strict Lipschitz constraints on the network adversely affects its denoising performance \cite{hertrich2021convolutional,hurault2022proximal}.

To address the challenge of nonexpansiveness in deep denoisers, Ryu et al. \cite{ryu2019plug} proposed a method where each layer is individually normalized using its spectral norm. However, this approach imposes limitations on the utilization of residual skip connections, which are widely employed in deep denoisers.
In a recent study, Hurault et al. \cite{hurault2022gradient} tackled this issue by training a deep image denoiser using a gradient-based PnP prior. By replacing the regularization step with the constructed denoiser, they demonstrated that the resulting gradient step PnP prior corresponds to the proximal operator of a specific nonconvex functional \cite{hurault2022proximal}. Under this condition, they successfully established the convergence of PnP-FBS, PnP-ADMM, and PnP-DRS iterates towards stationary points of explicit functions. Inspired by this research direction, it is worth exploring the convergence guarantees and potential applications of combining PnP methods with the DYS algorithm \cref{DYS} in the form of \cref{general}.

\subsection{Our contribution}
This paper provides a generic algorithm framework that combines splitting methods, extrapolation strategy, and deep prior. The main contributions of this paper are threefold:
\begin{itemize}
\item We propose an extrapolated DYS method for solving the type of structural nonconvex optimization problem \cref{general}, which provides a generic algorithm framework including extrapolated FBS and extrapolated DRS methods. Under the tight parameter conditions, the convergence of the generated iterates is established based on Kurdyka–{\L}ojasiewicz framework.
\item By replacing the regularization step with the gradient step-based denoiser, we propose two extrapolated PnP-DYS methods. The denoiser is constructed by a differentiable neural network and can be reformulated as the proximal operator of a specific nonconvex functional. The convergence of both PnP-DYS algorithms is also established.
\item Extensive experiments on image deblurring and image super-resolution problems are conducted to evaluate the performance of the proposed schemes. The numerical results illustrate the advantages and efficiency of the extrapolation strategy. Moreover, the experiments reveal the superiority of the PnP-based model with deep denoiser in terms of the quality of the recovered images. 
\end{itemize}

\subsection{Organization}
The remainder of this paper is organized as follows. Some related methods and preliminaries are reviewed in  \Cref{sec2}. An extrapolated DYS method with convergence analysis is developed in \Cref{sec:alg}. \Cref{sec4} combines PnP approach and produces two extrapolated PnP-DYS methods with convergence guarantee. Some experimental results are reported in \Cref{sec:experiments}, and the conclusions follow in
\Cref{sec:conclusions}.

 \subsection{Notation}
 We use $\mathbb{R}^n$ to denote the $n$-dimensional Euclidean space, $\mathbb{R}_+$ to denote the set of nonnegative real numbers, $\langle\cdot,\cdot\rangle$ to denote the inner product, and $\|\cdot\|$ to denote the norm induced from the inner product. For an extended real-valued function $f$, the domain of $f$ is defined as ${\rm dom} f:=\{{\bf x}\in\mathbb{R}^n\;|\;f({\bf x})<\infty\}$. We say that the function $f$ is proper if ${\rm dom}f\neq\emptyset$ and $f({\bf x}) > -\infty$ for any ${\bf x} \in {\rm dom}f$, and is closed if it is lower semicontinuous.
For any subset $S \subseteq \mathbb{R}^{n}$ and any point ${\bf x}\in \mathbb{R}^{n}$, the distance from ${\bf x}$ to $S$ is defined by
${\rm dist}({\bf x},S):= \inf\left\{\|{\bf y}-{\bf x}\|\; \big| \; {\bf y}\in S\right\},$
and ${\rm dist}({\bf x},S)=\infty$ for all $\bf x$ when $S=\emptyset$.

\section{Preliminaries}\label{sec2}
In this section, we review the definitions of subdifferential and Kurdyka-{\L}ojasiewicz (KL) property for further analysis.

\begin{definition} \label{def2.1}{\rm\cite{attouch2013convergence, bolte2014alternating}} \rm(Subdifferentials)
Let $f:\mathbb {R}^{n}\rightarrow (-\infty,+\infty]$ be a proper and lower semicontinuous function.
\begin{itemize}
\item[(i)] For a given ${\bf x}\in {\rm dom} f$, the Fr\'{e}chet subdifferential of $f$ at ${\bf x}$, written by $\widehat{\partial}f({\bf x})$, is the set of all vectors ${\bf u}\in \mathbb{R}^n$ satisfying\vspace{-0.05in}
$$\liminf_{{\bf y}\neq {\bf x}, {\bf y}\rightarrow {\bf x}}\frac{f({\bf y})-f({\bf x})-\langle {\bf u},{\bf y}-{\bf x}\rangle}{\|{\bf y}-{\bf x}\|}\geq0,\vspace{-0.05in}$$
and we set $\widehat{\partial}f({\bf x}) = \emptyset$ when ${\bf x}\notin {\rm dom}f$.
\vspace{0.2cm}
\item[(ii)] The limiting-subdifferential, or simply the subdifferential, of $f$ at ${\bf x}$, written by $\partial f({\bf x})$, is defined by 
\begin{equation}\label{pf}
\partial f({\bf x}):=\{{\bf u}\in\mathbb{R}^n\; | \; \exists ~ {\bf x}^k\rightarrow {\bf x}, ~{\rm s.t.}~f({\bf x}^k)\rightarrow f({\bf x})
 ~{\rm and}~ \widehat{\partial}f({\bf x}^k) \ni {\bf u}^k \rightarrow {\bf u} \}.  \end{equation}

\item[(iii)] A point ${\bf x}^*$ is called (limiting-)critical point or stationary point of $f$ if it satisfies $0\in\partial f({\bf x}^*)$, and the set of critical points of $f$ is denoted by ${\rm crit} f$.
\end{itemize}
\end{definition}

\cref{def2.1} implies that the property $\widehat{\partial}f({\bf x})\subseteq \partial f({\bf x})$ holds immediately, and $\widehat{\partial}f({\bf x})$ is closed and convex while $\partial f({\bf x})$ is closed. Indeed, the subdifferential \cref{pf} reduces to the gradient of $f$ denoted by $\nabla f$ if $f$ is continuously differentiable. Furthermore, as described in \cite{rockafellar2009variational}, if $g$ is a continuously differentiable function, it holds that $\partial(f+g)=\partial f+\nabla g$.
%

Next, we recall the KL property \cite{attouch2010proximal, bolte2014alternating}, which is important in the convergence analysis.

\begin{definition}\label{def2.2}(KL property and KL function)
Let $f:\mathbb{R}^{n}\rightarrow (-\infty,+\infty]$ be a proper and lower semicontinuous function.
\begin{itemize}
\item[$(a)$] The function $f$ is said to have KL property at ${\bf x}^*\in{\rm dom}(\partial f)$ if there exist $\eta\in(0,+\infty]$, a neighborhood $U$ of ${\bf x}^*$ and a continuous and concave function $\varphi:[0,\eta)\rightarrow \mathbb{R}_+$ such that
\begin{itemize}
\item[\rm(i)] $\varphi(0)=0$ and $\varphi$ is continuously differentiable on $(0,\eta)$ with $\varphi'>0$;

\item[\rm(ii)] for all ${\bf x}\in U\cap\{{\bf z} \in \mathbb{R}^n\; | \; f({\bf x}^*) < f({\bf z}) < f({\bf x}^*)+\eta\}$, the following KL inequality holds: 
\begin{equation} \label{kl}
\varphi'(f({\bf x})-f({\bf x}^*)){\rm dist}(0,\partial f({\bf x}))\geq1. 
\end{equation}
\end{itemize}

\item[$(b)$] If $f$ satisfies the KL property at each point of dom$(\partial f)$, then $f$ is called a KL function.
\end{itemize}
\end{definition}

Denote $\Phi_{\eta}$ as the set of functions $\varphi$ which satisfy the involved conditions in \cref{def2.2}(a). Then, we give an uniformized KL property which was established in \cite{bolte2014alternating} in the following, it will be useful for further convergence analysis.

\begin{lemma}\label{Lem2.1}{\rm\cite{bolte2014alternating}} {\rm (Uniformized KL property)} Let  $f:\mathbb{R}^{n}\rightarrow(-\infty,+\infty]$ be a proper and lower semicontinuous function and $\Omega$ be a compact set. Assume that $f$ is a constant on $\Omega$ and satisfies the KL property at each point of $\Omega$. Then, there exist $\varsigma>0,~\eta>0$ and $\varphi\in \Phi_{\eta}$ such that  \begin{equation}
\varphi'(f({\bf x})-f(\bar {\bf x})){\rm dist}(0,\partial f({\bf x}))\geq 1, 
\end{equation}
 for all $\bar {\bf x}\in \Omega$ and each ${\bf x}$ satisfying
${\rm dist}({\bf x},\Omega)<\varsigma$ and $f(\bar {\bf x}) < f({\bf x}) < f(\bar {\bf x})+\eta.$
\end{lemma}\par
Below we give a well-known descent lemma for smooth functions in the literature and the detailed proof can be found in \cite[Lemma 1.2.3]{nesterov2004}.
 \begin{lemma} \rm\cite{nesterov2004} \label{Lem2.2}
 Let $h:~\mathbb{R}^{n}\rightarrow \mathbb{R}$ be a continuously differentiable function with gradient $\nabla h$ assumed $L_h$-Lipschitz continuous. Then, we have  
 \begin{equation}
 \Big|h({\bf u})-h({\bf v})-\langle {\bf u}-{\bf v},\nabla h({\bf v})\rangle\Big|\leq \frac{L_h}{2}\|{\bf u}-{\bf v}\|^2, \qquad \forall ~ {\bf u},{\bf v}\in \mathbb{R}^{n}. \end{equation}
 \end{lemma}
\begin{lemma}\label{bot} {\rm\cite{boct2016inertial}}
Let $\{a_n\}$ and $\{b_n\}$ be two nonnegative sequences satisfying $\sum_{n\in\mathbb{N}}b_n<\infty$ and $a_{n+1}\leq a\cdot a_n+b\cdot a_{n-1}+b_n$ for all $n\geq 1$, where $a\in\mathbb{R}$, $b\geq 0$ and $a+b<1$. Then, we have $\sum_{n\in\mathbb{N}}a_n<\infty$.
\end{lemma}

\section{Extrapolated DYS method with convergence analysis}\label{sec:alg}

\begin{algorithm}[b!]
\caption{An extrapolated DYS method}\label{alg3.1}
\begin{algorithmic}
\STATE{Choose the parameters $\alpha\geq 0$ and $\gamma>0$. Given ${\bf x}^0$ and ${\bf x}^{-1}={\bf x}^0$, set $k=0$.}
\WHILE{the stopping criteria is not satisfied,}
\STATE{\vspace{-0.5cm}
 \begin{equation}\label{alg:2}
\left\{\begin{aligned}
 {\bf w}^k &= {\bf x}^k+\alpha({\bf x}^k-{\bf x}^{k-1}),\\
 {\bf y}^{k+1} &= {\rm Prox}_{\gamma f_1}\left({\bf w}^k\right), \\
 {\bf z}^{k+1} &= {\rm Prox}_{\gamma f_2}\left(2 {\bf y}^{k+1}-\gamma \nabla h({\bf y}^{k+1})-{\bf w}^k\right), \\
 {\bf x}^{k+1} &= {\bf w}^k+\left({\bf z}^{k+1}-{\bf y}^{k+1}\right) .
\end{aligned}\right.     
 \end{equation}
} \vspace{-0.2cm}
\ENDWHILE
\end{algorithmic}
\end{algorithm}

In this section, we propose a general extrapolated DYS method and conduct the convergence analysis.
\subsection{The extrapolated DYS method}
We propose an extrapolated DYS algorithm to solve the general nonconvex optimization problem \cref{general}, where an extrapolation step is incorporated to accelerate the convergence speed. Note that for any $\gamma>0$, the proximal operator of the function $f$ is defined by 
$$
{\rm Prox}_{\gamma f}({\bf x})= \argmin _{{\bf y} \in \mathbb{R}^n}\left\{f({\bf y})+\frac{1}{2 \gamma}\|{\bf y}-{\bf x}\|^2\right\}. 
$$
 We say that  $f$ is prox-bounded if $f+\frac{1}{2 \gamma}\|\cdot\|^2$ is lower bounded for some $\gamma>0$. The supremum of all such $\gamma$ is the threshold of prox-boundedness of $f$, denoted as $\gamma_f$. If $f$ is lower semicontinuous, then ${\rm Prox}_{\gamma f}$ is nonempty  and compact for all $\gamma \in\left(0, \gamma_f\right)$ \cite[Theorem 1.25]{rockafellar2009variational}.

The concrete iterative scheme is summarized in \cref{alg3.1}, which provides a versatile algorithmic framework
that encompasses both (extrapolated) forward-backward splitting and (extrapolated) Douglas-Rachford 
splitting methods. 
In particular, when the extrapolation step vanishes, i.e., $\alpha=0$, \cref{alg3.1} simplifies to the classical three-operator splitting method studied in \cite{bian2021three, davis2017three}. When $f_1=0$ in \cref{general}, \cref{alg3.1} reduces to the extrapolated (or named inertial) forward-backward splitting method, also known as inertial proximal gradient method, studied in \cite{attouch2014dynamical,lorenz2015inertial,liang2017activity,wu2019general}. \cref{alg3.1} also recovers extrapolated Douglas-Rachford splitting method when $h=0$.

Besides, when $\alpha=0$ and the function $h$ vanishes, \cref{alg3.1} reduces
to the classical DRS algorithm. The convergence of DRS method for nonconvex optimization was first discussed in \cite{li2016douglas}, and then refined in \cite{themelis2020douglas}. Some other variants and extensions of DRS method for nonconvex optimization can refer to \cite{guo2018note,guo2017convergence,li2019convergence,lindstrom2021survey,themelis2022douglas}. 
 When $\alpha=0$ and $f_1$ vanishes, the DYS algorithm becomes another very popular approach, namely, the forward-backward splitting (FBS) or proximal gradient method.
 We refer to \cite{ahookhosh2021bregman,attouch2013convergence,boct2016inertial,themelis2018forward,wu2019general}
 for the extension studies of FBS method in the nonconvex setting.

Next we present some assumptions for problem \cref{general} to facilitate convergence analysis.

\begin{assumption}\label{ass1}
The functions $f_1, f_2$, and $g$ in \cref{general} satisfy the following conditions:
\begin{itemize}
\item[(i)]  $f_1$ has a Lipschitz continuous gradient, i.e., there exists a constant $L_{f_1}>0$ such that 
$$
\left\|\nabla f_1\left({\bf y}_1\right)-\nabla f_1\left({\bf y}_2\right)\right\| \leq L_{f_1}\left\|{\bf y}_1-{\bf y}_2\right\|, \quad \forall ~ {\bf y}_1, {\bf y}_2 \in \mathbb{R}^n.
$$ 
\item[(ii)]  $h$ has a Lipschitz continuous gradient, i.e., there exists a constant $L_h>0$ such that
$$
\left\|\nabla h\left({\bf y}_1\right)-\nabla h\left({\bf y}_2\right)\right\| \leq L_h\left\|{\bf y}_1-{\bf y}_2\right\|, \quad \forall ~ {\bf y}_1, {\bf y}_2 \in \mathbb{R}^n.
$$
\item[(iii)] $f_2: \mathbb{R}^n\rightarrow \mathbb{R}\cup\{\infty\}$ is a proper closed function, and the objective function $F$ is bounded from below.  
\end{itemize}
\end{assumption}

Let $l\in\mathbb{R}$ be a constant such that $f_1 + \frac{l}{2}\|\cdot\|^2$ is convex. It should be noted that the existence of such an $l$ can be guaranteed by the Lipschitz continuity of $\nabla f_1$. Specifically, one can always choose $l = L_{f_1}$. In addition, it follows from the convexity of $f_1 + \frac{l}{2}\|\cdot\|^2$ that
$$f_1({\bf y}_1)-f_1({\bf y}_2)-\langle\nabla f_1({\bf y}_2),{\bf y}_1-{\bf y}_2\rangle\geq-\frac{l}{2}\|{\bf y}_1-{\bf y}_2\|^2,\quad \forall ~ {\bf y}_1, {\bf y}_2 \in \mathbb{R}^n.$$
Then, according to the Lipschitz continuity of $\nabla f_1$ and Lemma \ref{Lem2.2}, it must holds that $l\geq -L_{f_1}$. Hence, there must exist a constant
 $l\in[-L_{f_1},L_{f_1}]$ such that $f_1 + \frac{l}{2}\|\cdot\|^2$ is convex.
Note that $l<0$ implies that $f_1$ is strongly convex. 
Define
\begin{equation}\label{Lambda}
\Lambda(\gamma):=\frac{1-\gamma { l}-2\gamma L_h}{2 +\gamma L_h}- \gamma^2L_{f_1}^2.
\end{equation}
Now we give the parameter conditions for \cref{alg3.1} in the following assumption.
\begin{assumption}\label{ass2}
{ The parameters $\alpha$ and $\gamma$ should be chosen such that  $0<\gamma<\frac{1}{L_{f_1}+L_h}$ and  
$
0\leq\alpha<\Lambda(\gamma).
$}
\end{assumption} 

\begin{remark}
Note that  for given $ L_{f_1}>0$ and $L_h \geq 0, ~\Lambda(\gamma)>0$ always holds if $\gamma>0$ is sufficiently small. { Moreover, for the case of $L_h=0$, i.e., when $h=0$, it is easy to determine that $\Lambda(\gamma)>0$ if the following threshold for $\gamma$ is satisfied:
\begin{equation}\label{gamma_range}
0<\gamma<{ \frac{-l+\sqrt{l^2+8L_{f_1}^2}}{4 L_{f_1}^2}}.
\end{equation}
The above relation implies that $\gamma<\frac{1}{L_{f_1}}$, since the maximum value of the upper bound can be attained when $l = -L_{f_1}$ for every fixed value of $L_{f_1}$.
 Indeed, when $h = 0$ and $\alpha=0$, the extrapolated DYS algorithm \cref{alg:2} reduces to the classical DRS algorithm in \cite{li2016douglas,themelis2020douglas}. In this case, the range of $\gamma$ specified in \cref{gamma_range} is tighter compared to that in \cite{li2016douglas}, particularly in terms of the larger upper bound.
 For $L_h>0$, we can also provide a computable threshold for $\gamma$ to ensure that \cref{ass2} holds, i.e., $0<\gamma<\frac{1}{L_{f_1}+L_h}$ and $\Lambda(\gamma)>0$, as follows:
\begin{equation} \label{gamma_range1}
0<\gamma< \min\left\{\frac{1}{L_{f_1}+L_h},\gamma_0 \right\}, 
\end{equation}}
where $\gamma_0:=\frac{-(2L_h+ L_{f_1}+l)+\sqrt{( 2L_h+ L_{f_1}+l)^2+ 4(L_h L_{f_1}+L_{f_1}^2)}}{2(L_h L_{f_1}+L_{f_1}^2)}.    $
 \end{remark}
 \begin{remark}
When $\alpha=0$, the extrapolated DYS algorithm \cref{alg:2} reduces to the method studied in \cite{bian2021three,liu2019envelope}. {\h However, in this case, the range of $\gamma$ based on \cref{ass2} is different from the result in \cite{bian2021three} for the fixed $L_{f_1}$, $L_h$ and $l$. }
Especially the upper bound of $\gamma$ is different due to the distinct construction of $\Lambda(\gamma)$ in \cref{Lambda}. 
In other words, as a byproduct, this paper provides an improved parameter condition for $\gamma$ to ensure the convergence of the DYS method in the nonconvex setting. In addition, the lower boundedness of the energy function for the DYS method, and a certain sublinear convergence rate are established under some common conditions, which will be detailed later.
\end{remark}

\subsection{Convergence analysis}
In this subsection, we prove the convergence of \cref{alg3.1}, i.e., the extrapolated DYS algorithm, for the general nonconvex optimization problem \cref{general} under \cref{ass1} and \cref{ass2}. 

For convenience, we first present the corresponding first-order optimality conditions for the $\bf y$- and $\bf z$-subproblems in \cref{alg:2}, which will be frequently utilized in the subsequent convergence analysis. Specifically, the optimality condition for $\bf y$-subproblem in \cref{alg:2} is  
\begin{equation}\label{optimality_y}   
  0 = \nabla f_1({\bf y}^{k+1})+\frac{1}{\gamma}\left({\bf y}^{k+1}-{\bf w}^k\right),  
\end{equation}
and that for  $\bf z$-subproblem in \cref{alg:2} is  
\begin{equation}\label{optimality_z}
    0 \in \partial f_2({\bf z}^{k+1})+\frac{1}{\gamma}\left({\bf z}^{k+1}+\gamma \nabla h({\bf y}^{k+1})-2 {\bf y}^{k+1}+{\bf w}^k\right). 
\end{equation}

To simplify the notations in our analysis, we denote  
\begin{equation} \label{uv}
{\bf v}^k = ({\bf y}^k,{\bf z}^k,{\bf x}^k)^\top,\quad {\bf u}^k = ({\bf v}^k,{\bf x}^{k-1},{\bf x}^{k-2})^\top,\quad\forall ~ k\geq 1,
\end{equation}
and
\begin{equation} \label{Delta}
\Delta_{\bf x}^k = {\bf x}^k-{\bf x}^{k-1}, \quad  \Delta_{\bf y}^k = {\bf y}^k-{\bf y}^{k-1},\quad\forall ~ k\geq 1.  
 \end{equation}

Next, for  $\gamma>0$, we  define an auxiliary function $\mathcal{H}_\gamma$ as follows:
\begin{equation}\label{merit1}
\begin{aligned}
\mathcal{H}_\gamma\left({\bf y}, {\bf z},{\bf x}\right)&=  f_1({\bf y})+f_2({\bf z})+h({\bf y})+\frac{1}{2 \gamma}\|2 {\bf y}-{\bf z}-{\bf x}-\gamma \nabla h({\bf y})\|^2 \\
&\quad -\frac{1}{2 \gamma}\left\|{\bf x}-{\bf y}+\gamma \nabla h({\bf y})\right\|^2-\frac{1}{\gamma}\left\|{\bf y}-{\bf z}\right\|^2\\
&=  f_1({\bf y})+f_2({\bf z})+h({\bf y})+\frac{1}{2 \gamma}\left\|{\bf y}-{\bf x}-\gamma \nabla h({\bf y})\right\|^2
-\frac{1}{2 \gamma}\left\|{\bf z}-{\bf x}-\gamma \nabla h({\bf y})\right\|^2,
\end{aligned} 
\end{equation}
{  which is motivated by the DYS envelope studied in \cite{liu2019envelope} and also utilized in \cite{bian2021three}.} 
Based on the definition of $\mathcal{H}_\gamma$, we define the  energy function associated with extrapolated DYS method \cref{alg:2} as follows:  
\begin{equation}\label{merit2}
\begin{aligned}
\Theta_{\alpha,\gamma}\left({\bf y},{\bf z},{\bf x},{\bf x}_1,{\bf x}_2\right)= \mathcal{H}_\gamma({\bf y},{\bf z},{\bf x})+\frac{\alpha^2}{2\gamma}\left\|{\bf x}_1-{\bf x}_2\right\|^2,
\end{aligned} 
\end{equation}
where $\alpha \geq 0$ is a constant parameter that remains consistent with that in \cref{alg3.1}.

We first show that the sequence  $\{\Theta_{\alpha,\gamma} \left({\bf u}^k\right)\}_{k\geq 1}$ is monotonically nonincreasing.

\begin{lemma}\label{lemma_descent}
Suppose that \cref{ass1} and \cref{ass2} hold. Let the sequence $\{({\bf y}^k,{\bf z}^k,{\bf x}^k)\}_{k\geq 1}$ be  generated by \cref{alg:2}, and $\{{\bf u}^k\},
~\{{\bf v}^k\}$ and $\{\Delta_{\bf x}^k\},~\{\Delta_{\bf y}^k\}$ are defined in \cref{uv} and \cref{Delta}, respectively. Then, for a given $\tau\in\left(\alpha,\Lambda(\gamma)\right)$, the sequence  $\{\Theta_{\alpha,\gamma}\left({\bf u}^k\right)\}_{k\geq 1}$ is monotonically nonincreasing. In particular, for any $k\geq 1$, we have 
\begin{equation} \label{descent} 
\begin{aligned}
&\Theta_{\alpha,\gamma}\left({\bf u}^k\right)-\Theta_{\alpha,\gamma}\left({\bf u}^{k+1}\right) \geq  \left(\Lambda(\gamma)-\tau\right)\left(\frac{1}{\gamma}+\frac{L_h}{2}\right)\left\|\Delta_{\bf y}^{k+1}\right\|^2 +\xi(\alpha,\gamma)\left\|\Delta_{\bf x}^k\right\|^2,
\end{aligned}    
\end{equation}
where $\Lambda(\gamma)$ is defined in \cref{Lambda}, and $\xi(\alpha,\gamma):= \frac{\alpha}{\gamma} +\alpha L_h-\frac{\alpha^2L_h}{2}-\frac{\gamma\alpha^2 L_h^2}{2}-\frac{\alpha^2 L_h}{2\tau}-\frac{\alpha^2}{\tau\gamma}>0$.
\end{lemma}
\begin{proof}
It follows from \cref{merit1} that  
\begin{equation}\label{proof_eq1}
\begin{aligned}
&\mathcal{H}_\gamma\left({\bf y}^{k+1}, {\bf z}^{k+1}, {\bf x}^k\right)-\mathcal{H}_\gamma\left({\bf y}^{k+1}, {\bf z}^{k+1}, {\bf x}^{k+1}\right) \\
& =\frac{1}{\gamma}\left\langle -\Delta_{\bf x}^{k+1}, {\bf z}^{k+1}-{\bf y}^{k+1}\right\rangle \\
& =-\frac{1}{\gamma}\left\|{\bf y}^{k+1}-{\bf z}^{k+1}\right\|^2-\frac{\alpha}{\gamma}\left\langle {\bf z}^{k+1}-{\bf y}^{k+1},\Delta_{\bf x}^k\right\rangle,
\end{aligned}  
\end{equation}
where the last equality follows from the first and last relations in \cref{alg:2}. Since ${\bf z}^{k+1}$ is a minimizer of the $\bf z$-subproblem according to the third equality in \cref{alg:2}, we have  
\begin{equation*} 
\begin{aligned}
& f_2({\bf z}^k)+\frac{1}{2 \gamma}\left\|2 {\bf y}^{k+1}-{\bf z}^k-{\bf w}^k-\gamma \nabla h({\bf y}^{k+1})\right\|^2  \\
&\geq f_2({\bf z}^{k+1})+\frac{1}{2 \gamma}\left\|2 {\bf y}^{k+1}-{\bf z}^{k+1}-{\bf w}^k-\gamma \nabla h({\bf y}^{k+1})\right\|^2.
\end{aligned} 
\end{equation*}
This together with \cref{merit1}, we have 
\begin{equation}\label{proof_eq2}
\begin{aligned}
& \mathcal{H}_\gamma\left({\bf y}^{k+1}, {\bf z}^k, {\bf x}^k\right)-\mathcal{H}_\gamma\left({\bf y}^{k+1}, {\bf z}^{k+1}, {\bf x}^k\right) \\
&=  f_2({\bf z}^k)+\frac{1}{2 \gamma}\left\|2 {\bf y}^{k+1}-{\bf z}^k-{\bf x}^k-\gamma \nabla h({\bf y}^{k+1})\right\|^2 -\frac{1}{\gamma}\left\|{\bf y}^{k+1}-{\bf z}^k\right\|^2 \\
&\quad-f_2({\bf z}^{k+1})-\frac{1}{2 \gamma}\left\|2 {\bf y}^{k+1}-{\bf z}^{k+1}-{\bf x}^k-\gamma \nabla h({\bf y}^{k+1})\right\|^2 +\frac{1}{\gamma}\left\|{\bf y}^{k+1}-{\bf z}^{k+1}\right\|^2 \\
& \geq \frac{1}{\gamma}\left\|{\bf y}^{k+1}-{\bf z}^{k+1}\right\|^2-\frac{1}{\gamma}\left\|{\bf y}^{k+1}-{\bf z}^k\right\|^2 
+\frac{1}{\gamma}\left\langle {\bf z}^{k+1}-{\bf z}^k,{\bf w}^k-{\bf x}^k\right\rangle\\
&= \frac{1}{\gamma}\left\|{\bf y}^{k+1}-{\bf z}^{k+1}\right\|^2-\frac{1}{\gamma}\left\|{\bf y}^{k+1}-{\bf z}^k\right\|^2 +\frac{\alpha}{\gamma}\left\langle {\bf z}^{k+1}-{\bf z}^k,\Delta_{\bf x}^k\right\rangle,
\end{aligned} 
\end{equation}
where  the last equality follows from the relation {  ${\bf w}^k={\bf x}^k+\alpha \Delta_{\bf x}^k$} in \cref{alg:2}. 
Since $f_1+\frac{1}{2 \gamma}\left\|{\bf w}^k-\cdot\right\|^2$ is a strongly convex function with modulus { $\frac{1}{\gamma}-l$}, { and recall the optimality condition $0 = \nabla f_1({\bf y}^{k+1})+\frac{1}{\gamma}\left({\bf y}^{k+1}-{\bf w}^k\right)$ for the $\bf y$-subproblem in \cref{optimality_y},} we obtain  
$$f_1({\bf y}^k)+\frac{1}{2 \gamma}\left\|{\bf y}^k-{\bf w}^k\right\|^2 \geq 
f_1({\bf y}^{k+1})+\frac{1}{2 \gamma}\left\|{\bf y}^{k+1}-{\bf w}^k\right\|^2+\frac{1}{2}\left(\frac{1}{\gamma}-{ l}\right)\left\|\Delta_{\bf y}^{k+1}\right\|^2. $$
This implies that 
$$\begin{aligned}
&f_1 ({\bf y}^k )+\frac{1}{2 \gamma}\left\|{\bf y}^k-{\bf x}^k\right\|^2 \\
&\geq f_1 ({\bf y}^{k+1} )+\frac{1}{2 \gamma}\left\|{\bf y}^{k+1}-{\bf x}^k\right\|^2 -\frac{\alpha}{\gamma}\left\langle \Delta_{\bf y}^{k+1},\Delta_{\bf x}^k\right\rangle+\frac{1}{2}\left(\frac{1}{\gamma}-{ l}\right)\left\|\Delta_{\bf y}^{k+1}\right\|^2.
\end{aligned} $$
Therefore, it follows from \cref{merit1} that
\begin{equation*} 
\begin{aligned}
& \mathcal{H}_\gamma\left({\bf y}^k, {\bf z}^k, {\bf x}^k\right)-\mathcal{H}_\gamma\left({\bf y}^{k+1}, {\bf z}^k, {\bf x}^k\right) \\
& = f_1({\bf y}^k)+h({\bf y}^k)+\frac{1}{2 \gamma}\left\|{\bf y}^k-{\bf x}^k-\gamma \nabla h({\bf y}^k)\right\|^2-\frac{1}{2 \gamma}\left\|{\bf z}^k-{\bf x}^k-\gamma \nabla h({\bf y}^k)\right\|^2 \\
&\quad -f_1({\bf y}^{k+1})-h({\bf y}^{k+1})-\frac{1}{2 \gamma}\left\|{\bf y}^{k+1}-{\bf x}^k-\gamma \nabla h({\bf y}^{k+1})\right\|^2+\frac{1}{2 \gamma}\left\|{\bf z}^k-{\bf x}^k-\gamma \nabla h({\bf y}^{k+1})\right\|^2 \\
& \geq h({\bf y}^k)-\left\langle {\bf y}^k-{\bf x}^k,\nabla h({\bf y}^k)\right\rangle+\frac{\gamma}{2}\left\|\nabla h({\bf y}^k)\right\|^2-\frac{1}{2 \gamma}\left\|{\bf z}^k-{\bf x}^k-\gamma \nabla h({\bf y}^k)\right\|^2\\
&\quad -h({\bf y}^{k+1})+\left\langle {\bf y}^{k+1}-{\bf x}^k, \nabla h({\bf y}^{k+1})\right\rangle-\frac{\gamma}{2}\left\|\nabla h({\bf y}^{k+1})\right\|^2+\frac{1}{2 \gamma}\left\|{\bf z}^k-{\bf x}^k-\gamma \nabla h({\bf y}^{k+1})\right\|^2 \\
&\quad +\frac{1}{2}\left(\frac{1}{\gamma}-{ l}\right)\left\|\Delta_{\bf y}^{k+1}\right\|^2-\frac{\alpha}{\gamma}\left\langle \Delta_{\bf y}^{k+1},\Delta_{\bf x}^k\right\rangle.
\end{aligned}
\end{equation*}
Then, expanding the squares and combining the terms in the right-hand side of the above inequality, we have
\begin{equation}\label{proof_eq3}
\begin{aligned}
& \mathcal{H}_\gamma\left({\bf y}^k, {\bf z}^k, {\bf x}^k\right)-\mathcal{H}_\gamma\left({\bf y}^{k+1}, {\bf z}^k, {\bf x}^k\right) \\
 &\geq h({\bf y}^k)+\left\langle {\bf z}^k-{\bf y}^k,\nabla h({\bf y}^k)\right\rangle-\frac{1}{2 \gamma}\left\|{\bf z}^k-{\bf x}^k\right\|^2
  -h({\bf y}^{k+1})+\left\langle {\bf y}^{k+1}-{\bf z}^k, \nabla h({\bf y}^{k+1})\right\rangle\\
 &\quad+\frac{1}{2 \gamma}\left\|{\bf z}^k-{\bf x}^k\right\|^2+\frac{1}{2}\left(\frac{1}{\gamma}-{ l}\right)\left\|\Delta_{\bf y}^{k+1}\right\|^2-\frac{\alpha}{\gamma}\left\langle \Delta_{\bf y}^{k+1},\Delta_{\bf x}^k\right\rangle\\
 &=  h({\bf y}^k)+\left\langle {\bf z}^k-{\bf y}^k,\nabla h({\bf y}^k)\right\rangle  -h({\bf y}^{k+1})-\left\langle {\bf z}^k-{\bf y}^{k+1}, \nabla h({\bf y}^{k+1})\right\rangle\\
 &\quad + \frac{1}{2}\left(\frac{1}{\gamma}-{ l}\right)\left\|\Delta_{\bf y}^{k+1}\right\|^2-\frac{\alpha}{\gamma}\left\langle \Delta_{\bf y}^{k+1},\Delta_{\bf x}^k\right\rangle,
\end{aligned}
\end{equation}
Next, according to \cref{Lem2.2} and the $L_h$-Lipschitz continuity of $\nabla h$, we have 
\begin{equation}\label{proof_eq4}
\begin{aligned}
&  h ({\bf y}^k)-h({\bf y}^{k+1})+\left\langle {\bf z}^k-{\bf y}^k,\nabla h({\bf y}^k)\right\rangle-\left\langle {\bf z}^k-{\bf y}^{k+1},\nabla h({\bf y}^{k+1})\right\rangle \\
& =h({\bf y}^k)-h({\bf y}^{k+1})+\left\langle \Delta_{\bf y}^{k+1},\nabla h({\bf y}^k)\right\rangle-\left\langle {\bf z}^k-{\bf y}^{k+1},\nabla h({\bf y}^{k+1})-\nabla h({\bf y}^k)\right\rangle \\
& \geq-\frac{L_h}{2}\left\|\Delta_{\bf y}^{k+1}\right\|^2-\left\langle {\bf z}^k-{\bf y}^{k+1},\nabla h({\bf y}^{k+1})-\nabla h({\bf y}^k)\right\rangle \\
& \geq -\frac{L_h}{2}\left\|\Delta_{\bf y}^{k+1}\right\|^2-\frac{L_h}{2}\left\|{\bf y}^{k+1}-{\bf z}^k\right\|^2-\frac{L_h}{2}\left\|\Delta_{\bf y}^{k+1}\right\|^2 .
\end{aligned} 
\end{equation}
Substituting \cref{proof_eq4} into \cref{proof_eq3}, we obtain
\begin{equation}\label{proof_eq5}
\begin{aligned}
& \mathcal{H}_\gamma\left({\bf y}^k, {\bf z}^k, {\bf x}^k\right)-\mathcal{H}_\gamma\left({\bf y}^{k+1}, {\bf z}^k, {\bf x}^k\right) \\
&\geq \frac{1}{2}\left(\frac{1}{\gamma}-{ l}\right)\left\|\Delta_{\bf y}^{k+1}\right\|^2-\frac{\alpha}{\gamma}\left\langle \Delta_{\bf y}^{k+1},\Delta_{\bf x}^k\right\rangle-L_h\left\|\Delta_{\bf y}^{k+1}\right\|^2-\frac{L_h}{2}\left\|{\bf y}^{k+1}-{\bf z}^k\right\|^2.
\end{aligned}
\end{equation}
 Summing \cref{proof_eq1}, \cref{proof_eq2} and \cref{proof_eq5} yields
\begin{equation}\label{proof_eq6}
\begin{aligned}
&\mathcal{H}_\gamma\left({\bf y}^k, {\bf z}^k, {\bf x}^k\right)-\mathcal{H}_\gamma\left({\bf y}^{k+1}, {\bf z}^{k+1}, {\bf x}^{k+1}\right) \\
&\geq 
\frac{1-\gamma { l}-2\gamma L_h}{2 \gamma}\left\|\Delta_{\bf y}^{k+1}\right\|^2-\left(\frac{1}{\gamma}+\frac{L_h}{2}\right)\left\|{\bf y}^{k+1}-{\bf z}^k\right\|^2+\frac{\alpha}{\gamma}\left\langle {\bf y}^k-{\bf z}^k,\Delta_{\bf x}^k\right\rangle \\
&=\frac{1-\gamma { l}-2\gamma L_h}{2 \gamma}\left\|\Delta_{\bf y}^{k+1}\right\|^2-\left(\frac{1}{\gamma}+\frac{L_h}{2}\right)\left\|{\bf y}^{k+1}-{\bf z}^k\right\|^2\\
&\quad-\frac{\alpha}{\gamma}\left\|\Delta_{\bf x}^k\right\|^2+\frac{\alpha^2}{\gamma}\left\langle \Delta_{\bf x}^{k-1},\Delta_{\bf x}^k\right\rangle,
\end{aligned}   
\end{equation}
where the last equality holds due to the fact ${\bf y}^k-{\bf z}^k={\bf w}^{k-1}-{\bf x}^k={\bf x}^{k-1}-{\bf x}^k+\alpha({\bf x}^{k-1}-{\bf x}^{k-2})$ by \cref{alg:2}. Our further aim is to analyze the negative term $\|{\bf y}^{k+1}-{\bf z}^k\|^2$. It follows from the second equality in \cref{alg:2} that
$ 0=\nabla f_1 ({\bf y}^{k+1} )+\frac{1}{\gamma} ({\bf y}^{k+1}-{\bf w}^k)$.
Further, we obtain
\begin{equation}\label{proof_eq7}
\begin{aligned}
&\left\|{\bf y}^{k+1}-{\bf z}^k\right\|^2 \\
& =\left\|{\bf y}^{k+1}-\left({\bf x}^k-{\bf w}^{k-1}+{\bf y}^k\right)\right\|^2\\
&=\left\|({\bf y}^{k+1}-{\bf w}^k)-({\bf y}^k-{\bf w}^{k-1})+({\bf w}^k-{\bf x}^k)\right\|^2 \\
& =\gamma^2\left\|\nabla f_1({\bf y}^k)-\nabla f_1({\bf y}^{k+1})\right\|^2+\|{\bf w}^k-{\bf x}^k\|^2 + 2\langle ({\bf y}^{k+1}-{\bf w}^k)-({\bf y}^k-{\bf w}^{k-1}),{\bf w}^k-{\bf x}^k \rangle\\
&\leq \gamma^2 L_{f_1}^2\left\|\Delta_{\bf y}^{k+1}\right\|^2+\alpha^2\left\|\Delta_{\bf x}^k\right\|^2 +2\alpha\left\langle ({\bf y}^{k+1}-{\bf w}^k)-({\bf y}^k-{\bf w}^{k-1}),\Delta_{\bf x}^k \right\rangle\\
&=\gamma^2 L_{f_1}^2\left\|\Delta_{\bf y}^{k+1}\right\|^2+2\alpha\left\langle \Delta_{\bf y}^{k+1},\Delta_{\bf x}^k \right\rangle
+2\alpha^2\left\langle \Delta_{\bf x}^{k-1},\Delta_{\bf x}^k\right\rangle-\alpha(2+\alpha)\left\|\Delta_{\bf x}^k\right\|^2,
\end{aligned}
\end{equation}
where the last equality follows from the relation ${\bf w}^{k-1}-{\bf w}^k=\alpha({\bf x}^{k-1}-{\bf x}^{k-2})+(1+\alpha)({\bf x}^{k-1}-{\bf x}^k)$ by \cref{alg:2}.
Substituting \cref{proof_eq7} into \cref{proof_eq6} yields  
\begin{equation} \label{proof_eq8}
\begin{aligned}
&\mathcal{H}_\gamma\left({\bf y}^k, {\bf z}^k, {\bf x}^k\right)-\mathcal{H}_\gamma\left({\bf y}^{k+1}, {\bf z}^{k+1}, {\bf x}^{k+1}\right) \\
&\geq  \left(\frac{1-\gamma { l}-2\gamma L_h}{2 \gamma}-\left(\frac{1}{\gamma}+\frac{L_h}{2}\right)\gamma^2L_{f_1}^2\right)\left\|\Delta_{\bf y}^{k+1}\right\|^2 \\
&\quad +\left(\frac{\alpha+\alpha^2}{\gamma}+\alpha L_h+\frac{\alpha^2L_h}{2}\right) \left\|\Delta_{\bf x}^k\right\|^2\\
&\quad-\left(\frac{2}{\gamma}+ L_h\right)\alpha\left\langle \Delta_{\bf y}^{k+1},\Delta_{\bf x}^k\right\rangle -\left(\frac{1}{\gamma}+L_h\right)\alpha^2\left\langle \Delta_{\bf x}^{k-1},\Delta_{\bf x}^k \right\rangle.
\end{aligned}   
\end{equation}
Note that for any { $\tau>0$}, it holds that
\begin{equation*}
 \alpha\left\langle \Delta_{\bf y}^{k+1},\Delta_{\bf x}^k\right\rangle\leq \frac{\tau}{2}\left\|\Delta_{\bf y}^{k+1}\right\|^2+\frac{\alpha^2}{2\tau}\left\|\Delta_{\bf x}^k\right\|^2,    
\end{equation*}
and
\begin{equation*}\begin{aligned}
 &\left(\frac{1}{\gamma}+L_h\right)\alpha^2\left\langle \Delta_{\bf x}^{k-1},\Delta_{\bf x}^k\right\rangle \leq \frac{\alpha^2}{2\gamma}\left\|\Delta_{\bf x}^{k-1}\right\|^2+\frac{\gamma\alpha^2}{2}\left(\frac{1}{\gamma}+L_h\right)^2\left\|\Delta_{\bf x}^k\right\|^2.   
\end{aligned}\end{equation*}
Substituting the above inequalities into \cref{proof_eq8}, we get
\begin{equation}  \label{proof_eq9}
\begin{aligned}
&\mathcal{H}_\gamma\left({\bf y}^k, {\bf z}^k, {\bf x}^k\right)-\mathcal{H}_\gamma\left({\bf y}^{k+1}, {\bf z}^{k+1}, {\bf x}^{k+1}\right) \\
&\geq  \left(\Lambda(\gamma)-\tau\right)\left(\frac{1}{\gamma}+\frac{L_h}{2}\right)\left\|\Delta_{\bf y}^{k+1}\right\|^2-\frac{\alpha^2}{2\gamma}\left\| \Delta_{\bf x}^{k-1}\right\|^2 \\
&\quad +\left(\frac{\alpha}{\gamma}+\frac{\alpha^2}{2\gamma}+\alpha L_h-\frac{\alpha^2L_h}{2}-\frac{\gamma\alpha^2 L_h^2}{2}-\frac{\alpha^2 L_h}{2\tau}-\frac{\alpha^2}{\tau\gamma}\right) \left\|\Delta_{\bf x}^k\right\|^2,
\end{aligned}   
\end{equation}
where $\Lambda(\gamma)$ is defined in \cref{Lambda} and $\tau$ is an auxiliary parameter assumed to satisfy $\alpha<\tau<\Lambda(\gamma)$, which must exist according to  \cref{ass2}. Then, according to the definition of $\Theta_{\alpha,\gamma}$ in \cref{merit2} and \cref{proof_eq9}, the conclusion \cref{descent} can be obtained directly.

Now we show that $\xi(\alpha,\gamma):= \frac{\alpha}{\gamma} +\alpha L_h-\frac{\alpha^2L_h}{2}-\frac{\gamma\alpha^2 L_h^2}{2}-\frac{\alpha^2 L_h}{2\tau}-\frac{\alpha^2}{\tau\gamma}>0$. Since $\tau>\alpha$, we can easily obtain that $\frac{\alpha}{\gamma}-\frac{\alpha^2}{\tau\gamma}>0$. It leaves to show 
$\alpha L_h-\frac{\alpha^2L_h}{2}-\frac{\gamma\alpha^2 L_h^2}{2}-\frac{\alpha^2 L_h}{2\tau}>0$. According to \cref{ass2}, we know that
\begin{equation} 
\alpha<\Lambda(\gamma)<\frac{1-\gamma { l}-2\gamma L_h}{2 +\gamma L_h}<\frac{1}{1+\gamma L_h} .
\end{equation}
This together with $\tau>\alpha$, we have
$\alpha L_h-\frac{\alpha^2L_h}{2}-\frac{\gamma\alpha^2 L_h^2}{2}-\frac{\alpha^2 L_h}{2\tau}>\alpha L_h-\frac{\alpha^2L_h}{2}-\frac{\gamma\alpha^2 L_h^2}{2}-\frac{\alpha^2 L_h}{2\alpha}=\frac{\alpha L_h}{2}(1-\alpha-\alpha\gamma L_h)>0$.
This completes the proof.
\end{proof}

{\h The following lemma presents that the sequences $\{\Delta_{\bf x}^k\}$, $\{\Delta_{\bf y}^k\}$ and $\{{\bf y}^k-{\bf z}^k\}$  vanish with certain sublinear convergence rate.
\begin{lemma}\label{lemma_descent_add}
Suppose that \cref{ass1} and \cref{ass2} hold. Let  the sequence  $\{({\bf y}^k,{\bf z}^k,{\bf x}^k)\}_{k\geq 1}$ be generated by \cref{alg:2} which is assumed to be bounded, and the sequences $\{{\bf u}^k\},
~\{{\bf v}^k\}$ are defined in \cref{uv}, respectively. Then, 
\begin{itemize}
\item[(i)] it holds that $\sum_{k=0}^\infty\|\Delta_{\bf x}^k\|^2<+\infty$ and $\sum_{k=0}^\infty\|\Delta_{\bf y}^k\|^2<+\infty$. Furthermore, we have $\lim_{k\rightarrow \infty}\|\Delta_{\bf x}^k\|=0$, $\lim_{k\rightarrow \infty}\|\Delta_{\bf y}^k\|=0$, and $\lim_{k\rightarrow \infty}\|{\bf y}^k-{\bf z}^k\|=0$. 
\item[(ii)] { it holds that $\min_{k\leq K}\|\Delta_{\bf x}^k\|=\mathcal{O}(\frac{1}{\sqrt{K}})$, $\min_{k\leq K}\|\Delta_{\bf y}^k\|=\mathcal{O}(\frac{1}{\sqrt{K}})$, and $\min_{k\leq K}\|{\bf y}^k-{\bf z}^k\|=\mathcal{O}(\frac{1}{\sqrt{K}})$.}
\end{itemize}
\end{lemma}}
\begin{proof}
We now prove (i). 
We first show that $\Theta_{\alpha,\gamma}\left({\bf u}^k\right)$ is lower bounded for all $k$. It follows from the definition of $\Theta_{\alpha,\gamma}$ in \cref{merit2} that
\begin{equation}  \label{rev1}
\begin{aligned}
\Theta_{\alpha,\gamma}\left({\bf u}^k\right)
&= \mathcal{H}_\gamma\left({\bf y}^k, {\bf z}^k, {\bf x}^k\right) +\frac{\alpha}{2\gamma}\|\Delta_{\bf x}^{k-1}\|^2\\
&= f_1({\bf y}^k)+f_2({\bf z}^k)+h({\bf y}^k)+\frac{1}{2 \gamma}\|2 {\bf y}^k-{\bf z}^k-{\bf x}^k-\gamma \nabla h({\bf y}^k)\|^2 \\
&\quad -\frac{1}{2 \gamma}\left\|{\bf x}^k-{\bf y}^k+\gamma \nabla h({\bf y}^k)\right\|^2-\frac{1}{\gamma}\left\|{\bf y}^k-{\bf z}^k\right\|^2+\frac{\alpha}{2\gamma}\|\Delta_{\bf x}^{k-1}\|^2.
\end{aligned}   
\end{equation}
Since $\nabla f_1$ and $\nabla h$ are both Lipschitz continuous with moduli $L_{f_1}$ and $L_h$, then
\begin{equation*}
f_1({\bf y}^k)\geq f_1({\bf z}^k)-\langle\nabla f_1({\bf y}^k),{\bf z}^k-{\bf y}^k\rangle-\frac{L_{f_1}}{2}\|{\bf y}^k-{\bf z}^k\|^2,
\end{equation*}
and 
\begin{equation*}
h({\bf y}^k)\geq h({\bf z}^k)-\langle\nabla h({\bf y}^k),{\bf z}^k-{\bf y}^k\rangle-\frac{L_h}{2}\|{\bf y}^k-{\bf z}^k\|^2.
\end{equation*}
Substituting them into \cref{rev1}, and togethering with $\nabla f_1({\bf y}^{k})=-\frac{1}{\gamma}\left({\bf y}^{k}-{\bf w}^{k-1}\right) $ from \cref{optimality_y}, we have 
\begin{equation}  \label{rev2}
\begin{aligned}
\Theta_{\alpha,\gamma}\left({\bf u}^k\right)
&\geq f_1({\bf z}^k)+f_2({\bf z}^k)+h({\bf z}^k)+\left(\frac{1}{2\gamma}-\frac{L_{f_1}+L_h}{2}\right)\left\|{\bf y}^k-{\bf z}^k\right\|^2 \\
&\quad -\frac{1}{\gamma}\left\langle {\bf x}^k-{\bf w}^{k-1},{\bf y}^k-{\bf z}^k\right\rangle-\frac{1}{\gamma}\left\|{\bf y}^k-{\bf z}^k\right\|^2+\frac{\alpha}{2\gamma}\|\Delta_{\bf x}^{k-1}\|^2\\
&\geq F({\bf z}^k)+\left(\frac{1}{2\gamma}-\frac{L_{f_1}+L_h}{2}\right)\left\|{\bf y}^k-{\bf z}^k\right\|^2,
\end{aligned}   
\end{equation}
where the first inequality follows from $\frac{1}{2 \gamma}\|2 {\bf y}^k-{\bf z}^k-{\bf x}^k-\gamma \nabla h({\bf y}^k)\|^2 =\frac{1}{2\gamma}\left\|{\bf y}^k-{\bf z}^k\right\|^2+\frac{1}{\gamma}\left\langle {\bf y}^k-{\bf x}^k,{\bf y}^k-{\bf z}^k\right\rangle-\langle{\bf y}^k-{\bf z}^k,\nabla h({\bf y}^k)\rangle
+\frac{1}{2 \gamma}\left\|{\bf x}^k-{\bf y}^k+\gamma \nabla h({\bf y}^k)\right\|^2$, and the second one follows from $\frac{\alpha}{2\gamma}\geq 0$ and $ {\bf x}^{k} = {\bf w}^{k-1}+\left({\bf z}^{k}-{\bf y}^{k}\right)$ by \cref{alg:2}.
{\h This implies that $\Theta_{\alpha,\gamma}\left({\bf u}^k\right)$ for all $k\geq 1$ is bounded from below due to the fact that $0<\gamma<\frac{1}{L_{f_1}+L_h}$ and the boundedness of $F$ and $\left\{{\bf u}^k\right\}_{k\geq 1}$.}
Summing \cref{descent} from $k = 1$ to $N-1\geq 0$, we get
\begin{equation}  \label{proof_eq10}
\begin{aligned}
&\Theta_{\alpha,\gamma}\left({\bf u}^1\right)-\Theta_{\alpha,\gamma}\left({\bf u}^{N}\right) \geq  \left(\Lambda(\gamma)-\tau\right)\left(\frac{1}{\gamma}+\frac{L_h}{2}\right)\sum_{k=2}^N\left\|\Delta_{\bf y}^{k}\right\|^2 +\xi(\alpha,\gamma)\sum_{k=1}^{N-1}\left\|\Delta_{\bf x}^k\right\|^2.
\end{aligned}   
\end{equation}
 Therefore, letting $N\rightarrow +\infty$ and following the lower boundedness of  $\{\Theta_{\alpha,\gamma}\left({\bf u}^k\right)\}_{k\geq 1}$, we have
\begin{equation}  
\begin{aligned}
&  \left(\Lambda(\gamma)-\tau\right)\left(\frac{1}{\gamma}+\frac{L_h}{2}\right)\sum_{k=2}^\infty\left\|\Delta_{\bf y}^{k}\right\|^2 +\xi(\alpha,\gamma)\sum_{k=1}^\infty\left\|\Delta_{\bf x}^k\right\|^2
<+\infty.
\end{aligned}   
\end{equation}
This implies that
$\sum_{k=0}^\infty\|\Delta_{\bf x}^k\|^2<+\infty$ and $\sum_{k=0}^\infty\|\Delta_{\bf y}^k\|^2<+\infty$. Therefore, it holds that $\lim_{k\rightarrow \infty}\|\Delta_{\bf x}^k\|=0$ and $\lim_{k\rightarrow \infty}\|\Delta_{\bf y}^k\|=0$. Since  ${\bf y}^k-{\bf z}^k={\bf w}^{k-1}-{\bf x}^k={\bf x}^{k-1}-{\bf x}^k+\alpha({\bf x}^{k-1}-{\bf x}^{k-2})$ from \cref{alg:2}, we further have $\lim_{k\rightarrow \infty}\|{\bf y}^k-{\bf z}^k\|=0$. 

We turn to prove (ii). According to \cref{proof_eq10} and recalling $\xi(\alpha,\gamma)>0$ and the lower boundedness of $\{\Theta_{\alpha,\gamma}\left({\bf u}^k\right)\}_{k\geq 1}$, we know that there exists a constant $C_0$ such that
\begin{equation}  
\begin{aligned}
&K \cdot \min_{1\leq k\leq K}\left\|\Delta_{\bf x}^k\right\|^2\leq \sum_{k=1}^{K}\left\|\Delta_{\bf x}^k\right\|^2\leq \frac{1}{\xi(\alpha,\gamma)}\left(\Theta_{\alpha,\gamma}\left({\bf u}^1\right)-\Theta_{\alpha,\gamma}\left({\bf u}^{K+1}\right) \right)\leq C_0.
\end{aligned}   
\end{equation}
This implies that $\min_{k\leq K}\|\Delta_{\bf x}^k\|=\mathcal{O}(\frac{1}{\sqrt{K}})$. Similarly, we can obtain  $\min_{k\leq K}\|\Delta_{\bf y}^k\|=\mathcal{O}(\frac{1}{\sqrt{K}})$ and $\min_{k\leq K}\|{\bf y}^k-{\bf z}^k\|=\mathcal{O}(\frac{1}{\sqrt{K}})$. This completes the proof.
\end{proof}

Note that in \cref{lemma_descent_add}, we show the lower boundedness of $\Theta_{\alpha,\gamma}$, as well as  $\mathcal{H}_\gamma$, for the generated sequences, relying on \cref{ass1}(iii) and \cref{ass2}. 
The lower boundedness plays a crucial role in establishing both the sublinear convergence rate and the convergence of the generated sequence.
Some similar results have also been discussed in \cite{liu2019envelope}, which demonstrates the consistency between the lower bound and the minimizer of $\mathcal{H}_\gamma$ and $F$.

In the following, we give the subsequential convergence result for \cref{alg3.1}. 
\begin{theorem}\label{theo1}
Suppose that \cref{ass1} and \cref{ass2} hold. Let  the sequence  $\{({\bf y}^k,{\bf z}^k,{\bf x}^k)\}_{k\geq 1}$ be generated by \cref{alg:2} which is assumed to be bounded, and the sequences $\{{\bf u}^k\},
~\{{\bf v}^k\}$ are defined in \cref{uv}. Then, 
\begin{itemize}
\item[(i)] any cluster point ${\bf u}^*:=({\bf y}^*,{\bf z}^*,{\bf x}^*,{\bf x}^*,{\bf x}^*)$  of the sequence $\left\{{\bf u}^k\right\}_{k\geq 1}$ is a critical point of the problem \cref{general}, i.e., it holds that $0\in \partial F({\bf y}^*)$.
\item[(ii)] The limit $\lim_{k\rightarrow \infty}  \Theta_{\alpha,\gamma} ({\bf u}^k)$ exists and for any cluster point ${\bf u}^*$ of
the sequence $\{{\bf u}^k\}_{k\geq 1}$, we have
\begin{equation}
\Theta^*:=\lim_{k\rightarrow \infty}  \Theta_{\alpha,\gamma} ({\bf u}^k)=\Theta_{\alpha,\gamma} ({\bf u}^*).
\end{equation}
\end{itemize}
\end{theorem}
\begin{proof}
We first prove (i). It follows from \cref{alg:2} and \cref{lemma_descent_add}(i) that 
$$\lim _{k \rightarrow \infty}\left\|{\bf z}^{k+1}-{\bf z}^k\right\|=0.$$
Let ${\bf u}^*$ be a cluster point of $\left\{{\bf u}^k\right\}_{k \geq 1}$, and assume that $\left\{{\bf u}^{k_j}\right\}$ is a convergent subsequence such that
$
\lim _{k \rightarrow \infty} {\bf u}^{k_j} ={\bf u}^* .
$
Then
\begin{equation}\label{limit}
  \lim _{j \rightarrow \infty} {\bf u}^{k_j} =\lim _{j \rightarrow \infty} {\bf u}^{k_j-1}={\bf u}^*.
\end{equation} 
Summing \cref{optimality_y} and \cref{optimality_z} and taking the limit along the convergent subsequence $\{{\bf u}^{k_j}\}$, and applying \cref{pf} and \cref{limit}, we have
$$
0 \in \nabla f_1\left({\bf y}^*\right)+\partial f_2\left({\bf y}^*\right)+\nabla h\left({\bf y}^*\right).
$$

Now we prove (ii).
Suppose that $\left\{{\bf u}^{k_j}\right\}$ is a subsequence which converges to ${\bf u}^*$ as $j \rightarrow \infty$.
 It follows from \cref{lemma_descent} and \cref{lemma_descent_add} that $\Theta_{\alpha,\gamma}$ is nonincreasing and bounded from below by \cref{ass2}. Therefore, 
$\Theta^*:=\lim_{k\rightarrow \infty}  \Theta_{\alpha,\gamma} ({\bf u}^k)$ exists. 
It follows from \cref{alg:2} that ${\bf z}^k$ is the minimizer of ${\bf z}$-subproblem, we have 
$$\begin{aligned}
&f_2({\bf z}^k)+\frac{1}{2 \gamma}\left\|{\bf z}^k-\left(2 {\bf y}^k-\gamma \nabla h({\bf y}^k)-{\bf x}^{k-1}\right)\right\|^2 \\
&\leq f_2({\bf z}^*)+\frac{1}{2 \gamma}\left\|{\bf z}^*-\left(2 {\bf y}^k-\gamma \nabla h({\bf y}^k)-{\bf x}^{k-1}\right)\right\|^2.   
\end{aligned}$$
Replacing $k$ by $k_j$ in the above inequality and taking the limit on both sides, it follows from \cref{limit}  yields
$
\lim _{j \rightarrow \infty} f_2 ({\bf z}^{k_j} ) \leq f_2\left({\bf z}^*\right).
$
On the other hand, since $f_2$ is proper and closed, we have $\lim\inf _{j \rightarrow \infty} f_2\left({\bf z}^{k_j}\right) \geq$ $f_2\left({\bf z}^*\right)$. Hence
$$
\lim _{j \rightarrow \infty} f_2({\bf z}^{k_j})=f_2({\bf z}^*).
$$
This together with the properties of $f_1$ and $g$ in \cref{ass1} and \cref{limit}, { and the boundedness of the sequence $\{{\bf u}^k\}$}, we claim that
$$\Theta^*:=\lim_{k\rightarrow \infty}  \Theta_{\alpha,\gamma} ({\bf u}^k)=\Theta_{\alpha,\gamma} ({\bf u}^*).$$
This completes the proof.
\end{proof}

\begin{remark}
Note that the boundedness of the sequence $\{{\bf x}^k\}$ is a standard assumption for the nonconvex optimization algorithms. It is documented in \cite[Remark 3.3]{attouch2010proximal} that the boundedness assumption on the sequence $\{{\bf x}^k\}$ automatically holds when the corresponding lower level set
$\{{\bf x}~|~F({\bf x}) \leq F_0\}$ is compact for some $F_0 \in\mathbb{R}$.
\end{remark}

We present an inequality characterizing the upper bound of the subdifferential of $\Theta_{\alpha,\gamma}$, which plays a key role in further convergence analysis.
\begin{lemma}\label{lemma4}
 Suppose that \cref{ass1} and \cref{ass2} hold. Let $h$ be a twice continuously differentiable
function with a bounded Hessian, i.e., there exists a constant $M > 0$ such that $\| \nabla^2 h({\bf y})\| \leq  M, \forall ~ {\bf y}\in\mathbb{R}^n$.
Let  $\{({\bf y}^k,{\bf z}^k,{\bf x}^k)\}_{k\geq 1}$ be the sequence generated by \cref{alg:2}, and $\{{\bf u}^k\},
~\{{\bf v}^k\}$ are defined in \cref{uv}.  Then,  for any $k \geq  1$,
 there exists a constant $b>0$ such that 
 \begin{equation}\label{lemma3_eq1}
 {\rm dist}\left(0,\partial \Theta_{\alpha,\gamma}({\bf u}^{k+1})\right)\leq b\left(\|\Delta_{\bf x}^{k+1}\|+\|\Delta_{\bf x}^k\|\right).
 \end{equation}
\end{lemma}
\begin{proof}
 Firstly, from the definition of $\Theta_{\alpha,\gamma}$ in \cref{merit2}, we have
 \begin{equation}\label{lemma3_eq2}
\begin{aligned}
 \nabla_{\bf y} \Theta_{\alpha,\gamma}({\bf u}^{k+1}) 
& =\nabla f_1({\bf y}^{k+1})+\frac{1}{\gamma}\left({\bf y}^{k+1}-{\bf x}^{k+1}\right)+\nabla^2 h ({\bf y}^{k+1} )^\top\left({\bf z}^{k+1}-{\bf y}^{k+1}\right) \\
& =\frac{1}{\gamma}\left({\bf x}^k-{\bf x}^{k+1}\right)+\frac{\alpha}{\gamma}\left({\bf x}^k-{\bf x}^{k-1}\right)+\nabla^2 h ({\bf y}^{k+1})^\top\left({\bf z}^{k+1}-{\bf y}^{k+1}\right),
\end{aligned}
\end{equation}
where the last equality follows from \cref{optimality_y}.
Secondly, we compute the subgradient of $\Theta_{\alpha,\gamma}$ with respect to $\bf z$ as follows: 
 \begin{equation}\label{lemma3_eq3}
\begin{aligned}
& \partial_{\bf z} \Theta_{\alpha,\gamma}({\bf u}^{k+1}) \\
& =\partial f_2({\bf z}^{k+1})+\frac{1}{\gamma}\left({\bf z}^{k+1}-2 {\bf y}^{k+1}+\gamma \nabla h({\bf y}^{k+1})+{\bf x}^{k+1}\right)-\frac{2}{\gamma}\left({\bf z}^{k+1}-{\bf y}^{k+1}\right) \\
& \ni-\frac{1}{\gamma}\left({\bf x}^{k+1}-{\bf x}^k\right)-\frac{\alpha}{\gamma}\left({\bf x}^k-{\bf x}^{k-1}\right),
\end{aligned}
\end{equation}
where the inclusion follows from \cref{alg:2} and \cref{optimality_z}.
Thirdly, from the definition of $\Theta_{\alpha,\gamma}$ in \cref{merit2}, it is easy to obtain
 \begin{equation}\label{lemma3_eq4}
\begin{aligned}
\nabla_{\bf x} \Theta_{\alpha,\gamma}({\bf u}^{k+1}) & =\frac{1}{\gamma}\left({\bf z}^{k+1}-{\bf y}^{k+1}\right) =\frac{1}{\gamma}\left({\bf x}^{k+1}-{\bf x}^k\right)+\frac{\alpha}{\gamma}\left({\bf x}^k-{\bf x}^{k-1}\right),
\end{aligned}
\end{equation}
where the last equality follows from \cref{alg:2}. 
Finally, it follows from \cref{merit2} that 
 \begin{equation}\label{lemma3_eq5}
 \nabla_{{\bf x}_1} \Theta_{\alpha,\gamma}({\bf u}^{k+1})=\frac{\alpha^2}{\gamma}\left( {\bf x}^k-{\bf x}^{k-1}\right)~~ {\rm and}~~ 
\nabla_{{\bf x}_2} \Theta_{\alpha,\gamma}({\bf u}^{k+1})=\frac{\alpha^2}{\gamma}\left( {\bf x}^{k-1}-{\bf x}^k\right).
\end{equation}
Besides, by the boundedness of $\|\nabla^2 h(\cdot)\|$, we get
 \begin{equation}\label{lemma3_eq6}
\begin{aligned}
& \left\|\nabla_{\bf y} \Theta_{\alpha,\gamma}({\bf u}^{k+1})\right\| \\
& \leq \frac{1}{\gamma}\left\|{\bf x}^k-{\bf x}^{k+1}\right\|+\frac{\alpha}{\gamma}\left\|{\bf x}^k-{\bf x}^{k-1}\right\|+M\left\|{\bf z}^{k+1}-{\bf y}^{k+1}\right\| \\
& \leq\left(\frac{1}{\gamma}+M\right)\left\|{\bf x}^k-{\bf x}^{k+1}\right\| +\left(\frac{\alpha}{\gamma}+\alpha M\right)\left\|{\bf x}^k-{\bf x}^{k-1}\right\|,
\end{aligned}
\end{equation}
where the last inequality follows from the first and last relations in \cref{alg:2}. Combining \cref{lemma3_eq2}, \cref{lemma3_eq3}, \cref{lemma3_eq4}, \cref{lemma3_eq5}, and \cref{lemma3_eq6}, we can obtain the conclusion \cref{lemma3_eq1} immediately. This completes the proof.
\end{proof}

Now we establish the global convergence for \cref{alg3.1} based on the uniformized KL property. { We will show that the sequence $\left\{{\bf u}^k\right\}_{k \geq 1}$ has finite length and thus is convergent. Especially, the sequence $\left\{{\bf y}^k\right\}_{k \geq 1}$ converges to a stationary point in ${\rm crit}F$. }
\begin{theorem}\label{theo2}
 Suppose that \cref{ass1} and \cref{ass2} hold. Let $h$ be a twice continuously differentiable
function with a bounded Hessian, i.e., there exists a constant $M > 0$ such that $\| \nabla^2 h({\bf y})\| \leq  M, \forall ~ {\bf y}\in\mathbb{R}^n$.
Let  $\{({\bf y}^k,{\bf z}^k,{\bf x}^k)\}_{k\geq 1}$ be the sequence generated by \cref{alg:2} which is assumed to be bounded, and $\{{\bf u}^k\},
~\{{\bf v}^k\}$ are defined in \cref{uv}. If $F$ in \cref{general} is  a KL function, then
the sequence $\left\{{\bf u}^k\right\}_{k \geq 1}$ has finite length, that is,  
$$
 \sum_{k = 1}^\infty\left\|{\bf y}^{k+1}-{\bf y}^k\right\|<+\infty, \quad \sum_{k = 1}^\infty\left\|{\bf z}^{k+1}-{\bf z}^k\right\|<+\infty,\quad \sum_{k = 1}^\infty\left\|{\bf x}^{k+1}-{\bf x}^k\right\|<+\infty .
$$
Hence, the whole sequence $\{{\bf u}^k\}_{k\geq 1} $ is convergent.
\end{theorem}
\begin{proof}
   We use $\theta({\bf u}^\infty)$ to denote the cluster point set of the sequence $\{{\bf u}^k\}$. Since $\{{\bf u}^k\}$ is bounded, $\theta({\bf u}^\infty)$ is a nonempty compact set, and it holds that
$$\lim_{k \rightarrow \infty} {\rm dist}\left({\bf u}^k,\theta({\bf u}^\infty)\right) = 0.$$
From \cref{lemma_descent_add}(i), \cref{theo1}(i) and \cref{alg:2}, we know that $\theta({\bf u}^\infty)\subseteq{\rm crit} F\times {\rm crit} F\times {\rm crit} F\times {\rm crit} F\times {\rm crit} F$.
Hence, for any ${\bf u}^*:=({\bf y}^*,{\bf z}^*,{\bf x}^*,{\bf x}^*,{\bf x}^*)\in \theta({\bf u}^\infty)$, there exists a subsequence $\{{\bf u}^{k_i}\}$ of $\{{\bf u}^k\}$
converging to ${\bf u}^*$. 

It follows from \cref{theo1}(ii) that $\lim_{k\rightarrow \infty}  \Theta_{\alpha,\gamma} ({\bf u}^k)=\Theta_{\alpha,\gamma} ({\bf u}^*)$.
If there exists an integer $\bar k$ such that $\Theta_{\alpha,\gamma} ({\bf u}^k)=\Theta_{\alpha,\gamma} ({\bf u}^*)$, then from \cref{lemma_descent}, we have
$$\begin{aligned}
&\left(\Lambda(\gamma)-\tau\right)\left(\frac{1}{\gamma}+\frac{L_h}{2}\right)\left\|\Delta_{\bf y}^{k+1}\right\|^2 +\xi(\alpha,\gamma)\left\|\Delta_{\bf x}^k\right\|^2\\
&\leq \Theta_{\alpha,\gamma} ({\bf u}^k)-\Theta_{\alpha,\gamma} ({\bf u}^{k+1})\\
&\leq \Theta_{\alpha,\gamma} ({\bf u}^{\bar k})-\Theta_{\alpha,\gamma} ({\bf u}^*)\\
&=0 \quad \forall ~ k>\bar k. 
\end{aligned}
$$
Thus, we have ${\bf y}^{k+1}={\bf y}^k$ and ${\bf x}^{k+1}={\bf x}^k$ for any $k>\bar k$. Together with \cref{alg:2}, we also have ${\bf z}^{k+1}={\bf z}^k$, and thus the assertion $\sum_{k = 1}^\infty\left\|{\bf x}^{k+1}-{\bf x}^k\right\|<+\infty,~ \sum_{k = 1}^\infty\left\|{\bf y}^{k+1}-{\bf y}^k\right\|<+\infty,$ and $\sum_{k = 1}^\infty\left\|{\bf z}^{k+1}-{\bf z}^k\right\|<+\infty$ hold trivially.
Otherwise, since $\Theta_{\alpha,\gamma} ({\bf u}^k)$ is nonincreasing from \cref{lemma_descent}, we have $\Theta_{\alpha,\gamma} ({\bf u}^k)>\Theta_{\alpha,\gamma} ({\bf u}^*)$ for all $k$. Again from $\lim_{k\rightarrow \infty}  \Theta_{\alpha,\gamma} ({\bf u}^k)=\Theta_{\alpha,\gamma} ({\bf u}^*)$, we know that for any $\eta>0$, there exists a nonnegative integer $k_0$ such that $\Theta_{\alpha,\gamma} ({\bf u}^k)<\Theta_{\alpha,\gamma}({\bf u}^*)+\eta$ for any $k>k_0$. In addition, for any $\varsigma>0$ there exists a positive integer $k_1$ such that ${\rm dist}\left({\bf u}^{k},\theta({\bf u}^\infty)\right)<\varsigma$ for all $k>k_1$. Consequently, for any $\eta,~\varsigma>0$, when $k>k_2:=\max\{k_0,k_1\}$, we have
$${\rm dist}\left({\bf u}^k,\theta({\bf u}^\infty)\right)<\varsigma \qquad \hbox{and} \qquad \Theta_{\alpha,\gamma} ({\bf u}^k)<\Theta_{\alpha,\gamma}({\bf u}^*)+\eta.$$
Since $\theta({\bf u}^\infty)$ is a nonempty and compact set, and $\Theta_{\alpha,\gamma}$ is a constant on $\theta({\bf u}^\infty)$, we can apply \cref{Lem2.1} with $\Omega:=\theta({\bf u}^\infty)$.  Therefore, for any $k>k_2$, we have
\begin{equation}\label{var}
\varphi'(\Theta_{\alpha,\gamma}({\bf u}^k)-\Theta_{\alpha,\gamma}({\bf u}^*)){\rm dist}(0,\partial \Theta_{\alpha,\gamma}({\bf u}^k))\geq 1.
\end{equation}
From the concavity of $\varphi$, we have
\begin{eqnarray*}
&&\varphi (\Theta_{\alpha,\gamma}({\bf u}^k)-\Theta_{\alpha,\gamma}({\bf u}^*))-\varphi (\Theta_{\alpha,\gamma}({\bf u}^{k+1})-
\Theta_{\alpha,\gamma}({\bf u}^*))\\&&\quad\geq\varphi'(\Theta_{\alpha,\gamma}({\bf u}^k)-\Theta_{\alpha,\gamma}({\bf u}^*))( \Theta_{\alpha,\gamma}({\bf u}^k)-\Theta_{\alpha,\gamma}({\bf u}^{k+1})).
\end{eqnarray*}
Then, associated with ${\rm dist}\left(0,\partial \Theta_{\alpha,\gamma}({\bf u}^k)\right)\leq b(\|{\bf x}^k-{\bf x}^{k-1}\|+\|{\bf x}^{k-1}-{\bf x}^{k-2}\|)$ in \cref{lemma4}, \cref{var}, and $\varphi'(\Theta_{\alpha,\gamma}({\bf u}^k)-\Theta_{\alpha,\gamma}({\bf u}^*))>0$, we get
\begin{equation*}
\begin{split}
\Theta_{\alpha,\gamma}({\bf u}^k)-\Theta_{\alpha,\gamma}({\bf u}^{k+1})&\leq\frac{\varphi (\Theta_{\alpha,\gamma}({\bf u}^k)-\Theta_{\alpha,\gamma}({\bf u}^*))-\varphi (\Theta_{\alpha,\gamma}({\bf u}^{k+1})-
\Theta_{\alpha,\gamma}({\bf u}^*))}{\varphi'(\Theta_{\alpha,\gamma}({\bf u}^k)-\Theta_{\alpha,\gamma}({\bf u}^*))}\\
&\leq b(\|{\bf x}^k-{\bf x}^{k-1}\|+\|{\bf x}^{k-1}-{\bf x}^{k-2}\|)\\
&\quad\times[\varphi (\Theta_{\alpha,\gamma}({\bf u}^k)-\Theta_{\alpha,\gamma}({\bf u}^*))-\varphi (\Theta_{\alpha,\gamma}({\bf u}^{k+1})-\Theta_{\alpha,\gamma}({\bf u}^*))].
\end{split}
\end{equation*}
For convenience, for all $p,q\in\mathbb{ N}$, we define
$$\zeta_{p,q}:=\varphi\big(\Theta_{\alpha,\gamma}({\bf u}^p)-\Theta_{\alpha,\gamma}({\bf u}^*)\big)-\varphi\big(\Theta_{\alpha,\gamma}({\bf u}^q)-\Theta_{\alpha,\gamma}({\bf u}^*)\big).$$
Combining \cref{descent} and the above relation, it yields that for any $k>k_2$,
$$\begin{aligned} 
&\left(\Lambda(\gamma)-\tau\right)\left(\frac{1}{\gamma}+\frac{L_h}{2}\right)\left\|\Delta_{\bf y}^{k+1}\right\|^2 +\xi(\alpha,\gamma)\left\|\Delta_{\bf x}^k\right\|^2\\
&\leq \Theta_{\alpha,\gamma}({\bf u}^k)- \Theta_{\alpha,\gamma}({\bf u}^{k+1}) \\
&\leq b\left(\|\Delta_{\bf x}^k\|+\|\Delta_{\bf x}^{k-1}\|\right)\zeta_{k,k+1}.
\end{aligned}$$
This implies that
$$ \|\Delta_{\bf y}^{k+1}\| \leq \sqrt{\frac{1}{2}(\|\Delta_{\bf x}^k\|+\|\Delta_{\bf x}^{k-1}\|)}\sqrt{\frac{2b}{\rho_1}\zeta_{k,k+1}},$$
and 
$$ \|\Delta_{\bf x}^{k}\| \leq \sqrt{\frac{1}{2}(\|\Delta_{\bf x}^k\|+\|\Delta_{\bf x}^{k-1}\|)}\sqrt{\frac{2b}{\rho_2}\zeta_{k,k+1}},$$
where $\rho_1:=\left(\Lambda(\gamma)-\tau\right)\left(\frac{1}{\gamma}+\frac{L_h}{2}\right)$ and $\rho_2:= \xi(\alpha,\gamma)$.
Further, using the fact that $\sqrt{\mu_1\mu_2}\leq \mu_1/2 + \mu_2/2$ with $\mu_1=(\|\Delta_{\bf x}^k\|+\|\Delta_{\bf x}^{k-1}\|)/2$
and $\mu_2= 2b \zeta_{k,k+1}/\rho_1$ or $\mu_2= 2b \zeta_{k,k+1}/\rho_2$, we get
\begin{equation}\label{the_aa}
 \|\Delta_{\bf y}^{k+1}\|\leq\frac{1}{4}\left(\|\Delta_{\bf x}^k\|+\|\Delta_{\bf x}^{k-1}\|\right)+\frac{b}{\rho_1}\zeta_{k,k+1},
\end{equation}
and 
\begin{equation}\label{the11}
 \|\Delta_{\bf x}^{k}\|\leq\frac{1}{4}\left(\|\Delta_{\bf x}^k\|+\|\Delta_{\bf x}^{k-1}\|\right)+\frac{b}{\rho_2}\zeta_{k,k+1}.
\end{equation}
 Then, it follows from \cref{bot} and \cref{the11} that
$\sum_{k=1}^{\infty}\|\Delta_{\bf x}^{k+1}\| <+\infty,$ and we further have $\sum_{k=1}^{\infty}\|\Delta_{\bf y}^{k+1}\| <+\infty$ due to  \cref{the_aa}. Again from \cref{alg:2}, we know that $\sum_{k=1}^{\infty}\|\Delta_{\bf z}^{k+1}\| <+\infty$. Thus, $\{{\bf u}^k\}_{k\geq 1}$ is a Cauchy sequence and hence it is convergent. Applying \cref{theo1}(i), there exists a ${\bf y}^*\in {\rm crit}F$ such that $\lim_{k\rightarrow\infty}{\bf y}^k={\bf y}^*$. This completes the proof.
\end{proof}

\begin{remark}
KL functions exhibit remarkable versatility and are extensively applied in various domains, including semi-algebraic analysis, subanalytic analysis, and log-exp functions.
Concrete examples of KL functions can be found in \cite{attouch2010proximal, attouch2013convergence, bolte2014alternating}. These examples encompass many common instances such as $\ell_p$-norm (where $p\geq 0$), indicator functions of semi-algebraic sets, and a majority of convex functions.
\end{remark}

\section{Extrapolated PnP-DYS methods} \label{sec4}
In this section, we focus on the development of a  class of Plug-and-Play Davis-Yin splitting (PnP-DYS) algorithms with convergence guarantee. 
The PnP approach is a versatile methodology primarily utilized for addressing inverse problems involving large-scale measurements through the integration of statistical priors defined as denoisers. This approach draws inspiration from well-established proximal algorithms commonly employed in nonsmooth composite optimization, such as FBS, DRS, and ADMM.
The rise in the popularity of deep learning has resulted in the widespread adoption of PnP for effectively utilizing learned priors defined through pre-trained deep neural networks. This adoption has propelled PnP to achieve state-of-the-art performance across a range of applications.
For instance, by replacing the proximal operator of $f_2$ with a learned denoiser $\mathcal{D}_\sigma$ in \cref{DYS}, we can obtain a PnP-DYS method as follows:
\begin{equation*}
\left\{\begin{aligned}
 {\bf y}^{k+1} &= {\rm Prox}_{\gamma f_1}\left({\bf x}^k\right), \\ 
 {\bf z}^{k+1} &= \mathcal{D}_\sigma\left(2 {\bf y}^{k+1}-\gamma \nabla h({\bf y}^{k+1})-{\bf x}^k\right), \\
 {\bf x}^{k+1} &= {\bf x}^k+\left({\bf z}^{k+1}-{\bf y}^{k+1}\right).
\end{aligned}\right.     
 \end{equation*}
 
To guarantee the theoretical convergence, we  consider the Gradient Step (GS) Denoiser developed in \cite{cohen2021has,hurault2022gradient} as follows:
\begin{equation}\label{denoiser}
\mathcal{D}_\sigma= I-\nabla g_\sigma,
\end{equation}
which is obtained from a scalar function:
$$
g_\sigma=\frac{1}{2}\left\|{\bf x}-N_\sigma({\bf x})\right\|^2,
$$
where the mapping $N_\sigma({\bf x})$ is realized as a differentiable neural network, enabling the explicit computation of $g_\sigma$ and ensuring that $g_\sigma$ has a Lipschitz gradient with a constant $L$ ($L<1$). Originally, the denoiser $\mathcal{D}_\sigma$ in \cref{denoiser} is trained to denoise images degraded with Gaussian noise of level $\sigma$. In \cite{hurault2022gradient}, it is shown that, although constrained to be an exact conservative field, it can realize state-of-the-art denoising.
Remarkably, the denoiser $\mathcal{D}_\sigma$ in \cref{denoiser} takes the form of a proximal mapping of a weakly convex function, as stated in the next proposition.

\begin{proposition}\label{pro3.1} \cite[Propostion 3.1]{hurault2022proximal} $\mathcal{D}_\sigma({\bf x})=\operatorname{prox}_{\phi_\sigma}({\bf x})$, where $\phi_\sigma$ is defined by
\begin{equation}\label{phi_sigma}
\phi_\sigma({\bf x})=g_\sigma\left(\mathcal{D}_\sigma^{-1}({\bf x})\right)-\frac{1}{2}\left\|\mathcal{D}_\sigma^{-1}({\bf x})-{\bf x}\right\|^2
\end{equation}
if ${\bf x} \in \operatorname{Im}\left(\mathcal{D}_\sigma\right)$, and $\phi_\sigma({\bf x})=+\infty$ otherwise. Moreover, $\phi_\sigma$ is $\frac{L}{L+1}$-weakly convex and $\nabla \phi_\sigma \text { is } \frac{L}{1-L} \text {-Lipschitz on} \operatorname{Im}\left(\mathcal{D}_\sigma\right)$, { and  $\phi_\sigma({\bf x})\geq g_{\sigma}({\bf x})$, $\forall  {\bf x}\in\mathbb{R}^n$}.
\end{proposition}

Drawing upon the \cref{pro3.1}, we are interested in developing the extrapolated PnP-DYS algorithm, with a plugged denoiser $\mathcal{D}_\sigma$ in \cref{denoiser} that corresponds to the proximal operator of a nonconvex functional $\phi_\sigma$ in \cref{phi_sigma}. To do so, we turn to target the optimization problems as follows:
\begin{equation}\label{pnp_problem}
  \min F_{\gamma, \sigma}({\bf x})=f({\bf x})+\frac{1}{\gamma}\phi_\sigma({\bf x})+ h({\bf x}),  
\end{equation}
where $f$ is a (possibly nonconvex) data-fidelity term, $h$ is differential with Lipschitz continus gradient, $\gamma$ is a regularization parameter and $\phi_\sigma$ is defined as in \cref{pro3.1} from the function $g_\sigma$ satisfying $\mathcal{D}_\sigma=I-\nabla g_\sigma$. In our analysis, to use \cref{pro3.1}, $g_\sigma$ is assumed $\mathcal{C}^2$ with $L$-Lipschitz continuous gradient $(L<1)$. We also assume $f$ and $g_\sigma$ bounded from below. From \cref{pro3.1}, we get that $\phi_\sigma$ and thus $F_{\lambda, \sigma}$ are also bounded from below. In the following, we develop two extrapolated PnP-DYS methods depending on whether $f$ in \cref{pnp_problem} exhibits smoothness and discuss their theoretical convergence.

{\h According to \cite[Lemma 1]{hurault2023convergent}, $\phi_\sigma({\bf x})$ in \cref{phi_sigma} satisfies the Kurdyka-{\L}ojasiewicz (KL) property if $g_\sigma$ is real analytic \cite{krantz2002primer} in a neighborhood of $\bf x\in\mathbb{R}^n$ and its Jacobian matrix $Jg_\sigma ({\bf x})$ is nonsingular. Note that the real analytic property of $g_\sigma$ can be ensured for a broader range of deep neural networks. Meanwhile, the nonsingularity of $Jg_\sigma ({\bf x})$ can be guaranteed by assuming $L<1$ as discussed in \cite{hurault2023convergent}. For more discussions on general conditions under which the KL property holds for deep neural networks, we refer to \cite{barakat2020convergence,castera2021inertial, zeng2019global}. Therefore, selecting a neural network for $g_\sigma$ that guarantees the KL property of $\phi_\sigma({\bf x})$ during implementation is not a difficult task.
}

\subsection{When $f$ is smooth with Lipschitz continuous gradient}\label{subsec_4.1}
In this subsection, we consider the case that $f$ in \cref{pnp_problem} is differentiable with Lipschitz continuous gradient. 
In this case, we replace the second proximal subproblem with a learned denoiser $\mathcal{D}_\sigma$ in \cref{denoiser}, and produce a smooth extrapolated PnP-DYS method detailed in \cref{alg:buildtree}. Actually, \cref{alg:buildtree} reduces to the extrapolated versions, i.e., the accelerated versions, of PnP-DRS and PnP-FBS methods when $h=0$ and $f=0$, respectively. Notably, these specific cases have not been explored in previous literature to the best of our knowledge.

\vspace{-0.05in}
\begin{algorithm}
\caption{A smooth extrapolated PnP-DYS method}
\label{alg:buildtree}
\begin{algorithmic}
\STATE{Choose the parameters $\alpha\geq 0$ and $\gamma>0$. Given ${\bf x}^0$ and ${\bf x}^{-1}={\bf x}^0$, set $k=0$.}
\WHILE{the stopping criteria is not satisfied,}
\STATE{\vspace{-0.5cm}
\begin{equation*}
\left\{\begin{aligned}
{\bf w}^k &= {\bf x}^k+\alpha({\bf x}^k-{\bf x}^{k-1}),\\ 
 {\bf y}^{k+1} &= {\rm Prox}_{\gamma f}\left({\bf w}^k\right), \\ 
 {\bf z}^{k+1} &= \mathcal{D}_\sigma\left(2 {\bf y}^{k+1}-\gamma \nabla h({\bf y}^{k+1})-{\bf w}^k\right), \\ 
 {\bf x}^{k+1} &= {\bf w}^k+\left({\bf z}^{k+1}-{\bf y}^{k+1}\right) .
\end{aligned}\right.     
 \end{equation*}
}\vspace{-0.3cm}
\ENDWHILE
\end{algorithmic}
\end{algorithm}

Next, we discuss the convergence property of \cref{alg:buildtree} for the explicit optimization problem \cref{pnp_problem}. Before the analysis, we define
\begin{equation*}
\begin{aligned}
\mathcal{\widetilde H}_\gamma\left({\bf y}, {\bf z},{\bf x}\right)=   f({\bf y})+\frac{1}{\gamma}\phi_\sigma({\bf z})+h({\bf y})+\frac{1}{2 \gamma}\|{\bf y}-{\bf x}-\gamma \nabla h({\bf y})\|^2-\frac{1}{2 \gamma}\|{\bf z}-{\bf x}-\gamma \nabla h({\bf y})\|^2,
\end{aligned}
\end{equation*} 
and
\begin{equation*}
\begin{aligned}
\widetilde\Theta_{\alpha,\gamma}\left({\bf y},{\bf z},{\bf x},{\bf x}_1,{\bf x}_2\right)= \mathcal{\widetilde H}_\gamma({\bf y}, {\bf z},{\bf x})+\frac{\alpha^2}{2\gamma}\|{\bf x}_1-{\bf x}_2\|^2.
\end{aligned}
\end{equation*}

In the following, we present the convergence results of \cref{alg:buildtree}. 
\begin{theorem}\label{rev_theo4.1}
  Let $g_\sigma: \mathbb{R}^n \rightarrow \mathbb{R} \cup\{+\infty\}$ of class $\mathcal{C}^2$ with $L$-Lipschitz continuous gradient with $L<1$, and $\mathcal{D}_\sigma=I-\nabla g_\sigma$. Let $f: \mathbb{R}^n \rightarrow \mathbb{R} \cup\{+\infty\}$ and $h$ be  differentiable with $L_f$- and $L_h$-Lipschitz continuous gradient, {  and let $l_f$ be a constant such that $f+\frac{l_f}{2}\|\cdot\|$ is convex}. Suppose that $f$, $g_\sigma$ and $h$ are bounded from below, Then, for $\alpha$ and $\gamma$ satisfying  \cref{ass2} with $L_{f_1}:=L_f$ { and $l:=l_f$}, the sequence $\left\{({\bf y}^k, {\bf z}^k, {\bf x}^k)\right\}_{k\geq 1}$ generated by \cref{alg:buildtree} which is assumed to be bounded verify that
\begin{itemize}
    \item[(i)] $\left\{\widetilde\Theta_{\alpha,\gamma}\left({\bf y}^k, {\bf z}^k,{\bf x}^k,{\bf x}^{k-1},{\bf x}^{k-2}\right)\right\}_{k\geq 1}$ is nonincreasing and converges.
    \item[(ii)] { the sequences $\{\Delta_{\bf x}^k\}$, $\{\Delta_{\bf y}^k\}$ and $\{{\bf y}^k-{\bf z}^k\}$ vanish with  rate $\min_{k\leq K}\|\Delta_{\bf x}^k\|=\mathcal{O}(\frac{1}{\sqrt{K}})$, $\min_{k\leq K}\|\Delta_{\bf y}^k\|=\mathcal{O}(\frac{1}{\sqrt{K}})$, and $\min_{k\leq K}\|{\bf y}^k-{\bf z}^k\|=\mathcal{O}(\frac{1}{\sqrt{K}})$, respectively.} 
    \item[(iii)]  any cluster point ${\bf u}^*:=({\bf y}^*,{\bf z}^*,{\bf x}^*,{\bf x}^*,{\bf x}^*)$  of sequence $\left\{{\bf u}^k\right\}_{k\geq 1}$ is a critical point of the problem \cref{pnp_problem}, i.e., it holds that $0\in \partial F_{\gamma,\sigma}({\bf y}^*)$.
    \item[(iv)] if $h$ is a twice continuously differentiable
function with a bounded Hessian, i.e., there exists a constant $M > 0$ such that $\| \nabla^2 h({\bf y})\| \leq  M,~ \forall ~ {\bf y}\in\mathbb{R}^n$, and $F_{\gamma,\sigma}$ in \cref{pnp_problem} is a KL function. Then, the whole sequence $\{{\bf u}^k\}_{k\geq 1} $ is convergent.
\end{itemize}
\end{theorem}
\begin{proof}
Since $f$ and $h$ are differentiable with $L_f$- and $L_h$-Lipschitz continuous gradient, the problem \cref{pnp_problem} is a special form of \cref{general} with $f_1:=f$ and $f_2:=\frac{1}{\gamma}\phi_\sigma$. Therefore, it follows from \cref{lemma_descent} and \cref{lemma_descent_add} that (i) and (ii) hold. The assertion (iii) can be obtained according to \cref{theo1}, and the conclusion (iv) can be derived from \cref{theo2}. This completes the proof.
\end{proof}

\subsection{When $f$ is nonsmooth}\label{subsec_4.2}
To cope with the problem \cref{pnp_problem} with a possibly nondifferentiable function $f$,
 we propose a nonsmooth extrapolated PnP-DYS method in \cref{alg:3}. In this case, we replace the first proximal subproblem in \cref{alg:3} by a learned denoiser $\mathcal{D}_\sigma$ defined in \cref{denoiser} to guarantee the theoretical convergence.
\begin{algorithm}[t!]
\caption{A nonsmooth extrapolated PnP-DYS method}
\label{alg:3}
\begin{algorithmic}
\STATE{Choose the parameters $\alpha\geq 0$ and $\gamma>0$. Given ${\bf x}^0$ and ${\bf x}^{-1}={\bf x}^0$, set $k=0$.}
\WHILE{the stopping criteria is not satisfied,}
\STATE{\vspace{-0.5cm}
\begin{equation*}
\left\{\begin{aligned}
{\bf w}^k &= {\bf x}^k+\alpha({\bf x}^k-{\bf x}^{k-1}),\\
 {\bf y}^{k+1} &= \mathcal{D}_\sigma\left({\bf w}^k\right), \\
 {\bf z}^{k+1} &= {\rm Prox}_{\gamma f}\left(2 {\bf y}^{k+1}-\gamma \nabla h({\bf y}^{k+1})-{\bf w}^k\right), \\
 {\bf x}^{k+1} &= {\bf w}^k+\left({\bf z}^{k+1}-{\bf y}^{k+1}\right) .
\end{aligned}\right.     
 \end{equation*}
}\vspace{-0.3cm}
\ENDWHILE
\end{algorithmic}
\end{algorithm}

In order to analyze the convergence of \cref{alg:3}, we define
\begin{equation*}
\begin{aligned}
\mathcal{\widehat H}_\gamma\left({\bf y},{\bf z},{\bf x}\right)=   \frac{1}{\gamma}\phi_\sigma({\bf y})+f({\bf z})+h({\bf y})+\frac{1}{2 \gamma}\|{\bf y}-{\bf x}-\gamma \nabla h({\bf y})\|^2-\frac{1}{2 \gamma}\|{\bf z}-{\bf x}-\gamma \nabla h({\bf y})\|^2,
\end{aligned}
\end{equation*} 
and
\begin{equation*}
\begin{aligned}
\widehat\Theta_{\alpha,\gamma}\left({\bf y},{\bf z},{\bf x},{\bf x}_1,{\bf x}_2\right)= \mathcal{\widehat H}_\gamma({\bf y},{\bf z},{\bf x})+\frac{\alpha^2}{2\gamma}\|{\bf x}_1-{\bf x}_2\|^2.
\end{aligned}
\end{equation*}

Now we give the convergence results of \cref{alg:3} based on the conclusions in \Cref{sec:alg} and the discussions in \cite{hurault2022proximal}.
\begin{theorem}
Let $g_\sigma: \mathbb{R}^n \rightarrow \mathbb{R} \cup\{+\infty\}$ of class $\mathcal{C}^2$ with $L$-Lipschitz continuous gradient with $L<1$, and $\mathcal{D}_\sigma=I-\nabla g_\sigma$  with ${\rm Im}(\mathcal{D}_\sigma)$ being convex. Let $f: \mathbb{R}^n \rightarrow \mathbb{R} \cup\{+\infty\}$ is a proper closed function and $h$ is differentiable $L_h$-Lipschitz continuous gradient. Suppose that $f$, $g_\sigma$ and $h$ are bounded from below, Then, for $\alpha$ and $\gamma$ satisfying \cref{ass2} with $L_{f_1}:=\frac{L}{\gamma(1-L)}$ and $l:=\frac{L}{\gamma(L+1)}$, the sequence $\left\{({\bf y}^k, {\bf z}^k, {\bf x}^k)\right\}_{k\geq 1}$ generated by \cref{alg:3} which is assumed to be bounded verify that
\begin{itemize}
    \item[(i)] $\left\{\widehat\Theta_{\alpha,\gamma}\left({\bf y}^k, {\bf z}^k,{\bf x}^k,{\bf x}^{k-1},{\bf x}^{k-2}\right)\right\}_{k\geq 1}$ is nonincreasing and converges.
    \item[(ii)] { the sequences $\{\Delta_{\bf x}^k\}$, $\{\Delta_{\bf y}^k\}$ and $\{{\bf y}^k-{\bf z}^k\}$  vanish with  rate $\min_{k\leq K}\|\Delta_{\bf x}^k\|=\mathcal{O}(\frac{1}{\sqrt{K}})$, $\min_{k\leq K}\|\Delta_{\bf y}^k\|=\mathcal{O}(\frac{1}{\sqrt{K}})$, and $\min_{k\leq K}\|{\bf y}^k-{\bf z}^k\|=\mathcal{O}(\frac{1}{\sqrt{K}})$, respectively.} 
    \item[(iii)]  any cluster point ${\bf u}^*:=({\bf y}^*,{\bf z}^*,{\bf x}^*,{\bf x}^*,{\bf x}^*)$  of sequence $\left\{{\bf u}^k\right\}_{k\geq 1}$ is a critical point of the problem \cref{pnp_problem}, i.e., it holds that $0\in \partial F_{\gamma,\sigma}({\bf y}^*)$.
    \item[(iv)] if $h$ is a twice continuously differentiable
function with a bounded Hessian, i.e., there exists a constant $M > 0$ such that $\| \nabla^2 h({\bf y})\| \leq  M, \forall {\bf y}\in\mathbb{R}^n$, and $F_{\gamma,\sigma}$ in \cref{pnp_problem} is a KL function. Then, the whole sequence $\{{\bf u}^k\}_{k\geq 1} $ is convergent.
\end{itemize}
\end{theorem}
\begin{proof}
It follows from \cref{pro3.1} that $\phi_\sigma$ is $\frac{L}{L+1}$-weakly convex and $\nabla \phi_\sigma$ is $ \frac{L}{1-L}$-Lipschitz on $\operatorname{Im} (\mathcal{D}_\sigma)$. Thus, the problem \cref{pnp_problem} can be seen as a special form of \cref{general} with $f_1=\frac{1}{\gamma}\phi_\sigma$ and $f_2=f$.  Since ${\rm Im}(\mathcal{D}_\sigma)$ is convex, it follows from \cite[Appendix C.2]{hurault2022proximal} and \cref{lemma_descent} that (i) and (ii) hold. 
According to the assumptions on $g_\sigma$, we know that $\mathcal{D}_\sigma$ is continuous on ${\rm Im}(\mathcal{D}_\sigma)$, and then the assertion (iii) can be obtained according to \cref{theo1}. Moreover, the conclusion (iv) can be derived from \cref{theo2}. This completes the proof.
\end{proof}
{ 
\begin{remark}
 As discussed in \cite{hurault2022proximal,pesquet2021learning}, one can ensure that the Lipschitz constant $L<1$ for $\nabla g_\sigma$ is to softly constrain it by penalizing the spectral norm of the Hessian of $g_\sigma$ in the denoiser training loss. This approach will be further explained in the experiments. 
In addition, if $L>1$, one can relax the deep prior with a parameter $\eta\in[0,1]$, given by $\mathcal{D}_\sigma^\eta=\eta \mathcal{D}_\sigma+(1-\eta)I$. It is important to note that the relaxed deep prior $\mathcal{D}_\sigma^\eta$ exhibits the same property as stated in \cref{pro3.1}. More specifically, $\mathcal{D}_\sigma^\eta$ continues to be the proximal operator of a certain {\h weakly convex} functional. As a result, the condition becomes $\eta L<1$, which can be easily guaranteed since $\eta\in[0,1]$. We refer to \cite[Subsection 3.4]{hurault2023convergent} for more discussions.
\end{remark}
}

\section{Numerical experiments}\label{sec:experiments}
In this section, we implement the extrapolated DYS algorithm with or without PnP denoiser on image deblurring and image super-resolution tasks, and compare numerical results with other advanced models and methods. All experiments are implemented with PyTorch on an NVIDIA RTX A6000 GPU. 

We consider the image restoration problem with both sparse-induced regularization and Tikhonov regularization, whose mathematical model can be read as 
\begin{equation}\label{image_deblurring}
    \min_{{\bf x}\in\mathbb{R}^n}    \frac{1}{2\nu^2}\Vert A{\bf x} - {\bf b}\Vert^2+r({\bf x})+ \frac{\beta}{2}\Vert {\bf x}\Vert^2,
\end{equation}
where $r(\cdot)$ is the sparse-induced regularizer which maybe nonconvex, $\bf b$ is the observation, $\nu$ is the Gaussian noise level and $A$ is the linear operator. 
When $A$ denotes the blur kernel, the model \cref{image_deblurring} corresponds to image deblurring problem, which aims to restore a clean image ${\bf x}^*$ from the observed image $\bf b$.
Additionally, if $A = SB$, where $B$ denotes the blur operator and $S$ is the standard $s$-fold downsampler (i.e., selecting the upper-left pixel for each distinct $s \times s$ patch), the model \cref{image_deblurring} reduces to the image super-resolution problem. This problem involves enhancing the resolution and quality of a low-resolution image to generate a high-resolution version of the same image.
We can see  that the model \cref{image_deblurring} falls into the form of \cref{general} with $f_1({\bf x})=\frac{1}{2\nu^2}\Vert A{\bf x} - {\bf b}\Vert^2$, $f_2({\bf x})=r({\bf x})$ and $h({\bf x})=\frac{\beta}{2}\Vert {\bf x}\Vert^2$. 
Additionally, the following model with sparse-induced regularization and box constraint is also widely used in solving image deblurring and image super-resolution problems: 
\begin{equation}\label{image_deblurring2}
    \min_{{\bf x}\in\mathcal{B}}    \frac{1}{2\nu^2}\Vert A{\bf x} - {\bf b}\Vert^2+r({\bf x}),
\end{equation}
where $\mathcal{B}$ is a convex box. 
Model \cref{image_deblurring2} is  a special form of \cref{general} if $r(\cdot)$ is smooth with $f_1({\bf x})=r({\bf x})$, $f_2({\bf x})=\delta_{\mathcal{B}}({\bf x})$ and $h({\bf x})=\frac{1}{2\nu^2}\Vert A{\bf x} - {\bf b}\Vert^2$, where $\delta_{\mathcal{B}}({\cdot})$ denotes the { indicator} function. 

In the experiments, we will consider two cases of $r(\cdot)$ for \cref{image_deblurring} and \cref{image_deblurring2} as follows:
\begin{itemize}
\item[1.]  $r({\bf x})=\|{\bf x}\|_{\rm TV}$, the { isotropic} total-variational (TV)  regularizer \cite{goldstein2009split,rudin1992nonlinear};
\item[2.] $r({\bf x})=\frac{1}{\gamma}\phi_\sigma({\bf x})$, the nonconvex regularizer in \cref{phi_sigma} induced by Gradient Step (GS) denoiser $\mathcal{D}_\sigma$.
\end{itemize}
We refer to the model \cref{image_deblurring} with the above two regularizers as TVTik and DeTik. Similarly, the model \cref{image_deblurring2} with both regularizers is denoted as TVBox and DeBox, respectively.
As discussed in \Cref{sec4}, DeTik can be solved by \cref{alg:buildtree}, while DeBox should be solved by \cref{alg:3} due to the nonsmoothness of $\delta_{\mathcal{B}}$.  For the classical TV-based models, i.e., TVTik and TVBox, the split Bregman algorithm is applicable. 
{\h More specifically, we import the image processing package `scikit-image' in Python with `skimage.restoration.denoise\_tv\_bregman' for solving the isotropic TV-subproblem with a maximum iteration of 100.} 
Certainly, \cref{alg3.1} also can be used to solve TVTik, as there are two smooth terms with Lipschitz continuous gradient involved in \cref{image_deblurring}. 
We initialize all the tested algorithms with ${\bf x}^{-1} = {\bf x}^0 = {\bf b}$. The algorithms are terminated when the relative difference between consecutive values of the objective function is less than $\varepsilon = 10^{-8}$ or the number of iterations exceeds $k_{\max} = 1000$.
 
As aforementioned, we utilize the deep GS denoiser to replace the traditional regularizer. Specifically, in the experiments, we employ the classical DRUNet \cite{zhang2021plug} as our denoiser $\mathcal{D}_\sigma$. DRUNet incorporates both U-Net and ResNet architectures and takes an additional noise level map as input, achieving state-of-the-art performance in Gaussian noise removal. 
To ensure $L<1$ of the Lipschitz constant of $\nabla g_\sigma$ in \cref{denoiser}, following the approach in \cite{pesquet2021learning}, we regularize the training loss of $\mathcal{D}_\sigma$ using the spectral norm of the Hessian of $g_\sigma$ as follows: 
\begin{equation}\label{loss1}
\mathcal{L}_S(\sigma)=\mathbb{E}_{{\bf x} \sim p, \xi_\sigma \sim \mathcal{N}(0, \sigma^2)}\left[\|\mathcal{D}_\sigma({\bf x}+\xi_\sigma)-{\bf x}\|^2+\mu \max (\|\nabla^2g_\sigma({\bf x}+\xi_\sigma)\|_S, 1-\epsilon)\right],
\end{equation}
where $p$ is the distribution of a dataset of clean images and $\Vert\cdot\Vert_S$ is the spectral norm. We set $\epsilon=0.1$ and $\mu=0.01$ according to \cite{hurault2022proximal}. Following the setting of \cite{hurault2022gradient}, we have retrained the DRUNet \cite{zhang2021plug} with loss function \cref{loss1} on the Berkeley segmentation dataset, Waterloo Exploration Database, DIV2K dataset, and Flick2K dataset.
For the image deblurring problem, ten different blur kernels\footnote{\url{https://github.com/Huang-chao-yan/convergent_pnp/tree/main/kernels}} 
(from Ker1 to Ker10) and three noise levels: $\nu=\{2.55, 7.65, 12.75\}$ will be used to simulate the degraded image.

\subsection{Effect of extrapolation}
We first test the effectiveness of extrapolation parameter $\alpha$ by applying \cref{alg:buildtree} to solve the DeTik model. For the DeTik model, we know that $L_{f_1}=\frac{1}{\nu^2}\lambda_{\max}(A^\top A)$, $l=-L_{f_1}$ and $L_h=\beta$, where $\lambda_{\max}$ denotes the maximal eigenvalue of a given matrix. In the experiment,  we set the model parameter $\beta\in [0.0005, 0.001]$ for different noise levels in \cref{image_deblurring}. {\y It follows from \cref{ass2} that $0\leq\alpha<\Lambda(\gamma)$. Therefore, for a given and fixed $\gamma$ that satisfies \cref{gamma_range1}, we test the values of $\alpha=\{0,0.25,0.50,0.75,0.99\}*\Lambda(\gamma)$  by performing a numerical comparison of the computational cost and the quality of recovery for the image deblurring problem.}

\begin{figure}[h!]
	\begin{minipage}{0.3\linewidth}
		\centering
	\centerline{\includegraphics[width=2.in]{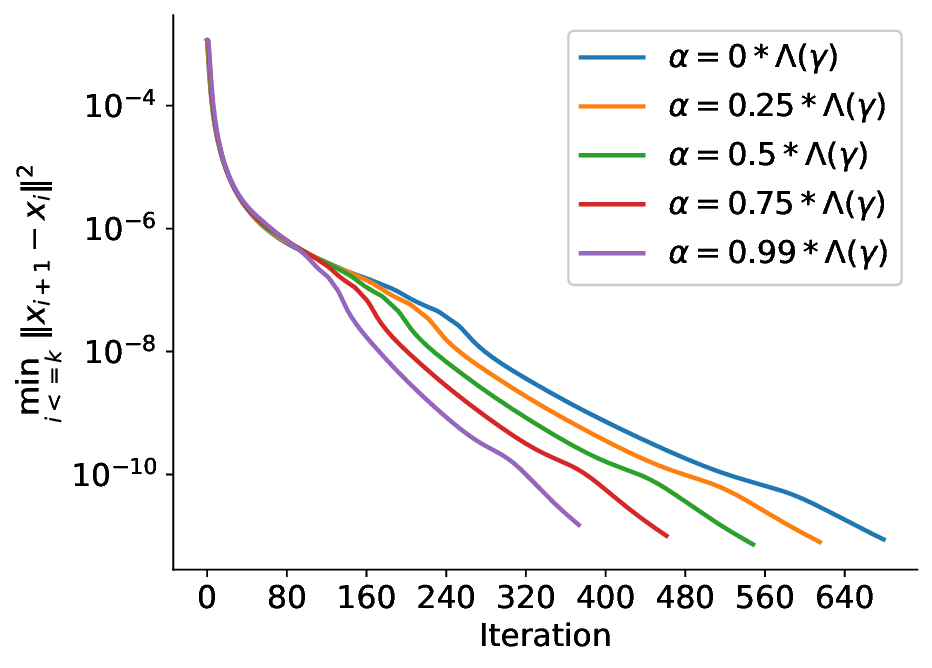}
		}\vspace{-0.03in}
		\centerline{}
	\end{minipage}  \hspace{0.15in}
 \begin{minipage}{0.3\linewidth}
		\centering
		\centerline{\includegraphics[width=2.in]{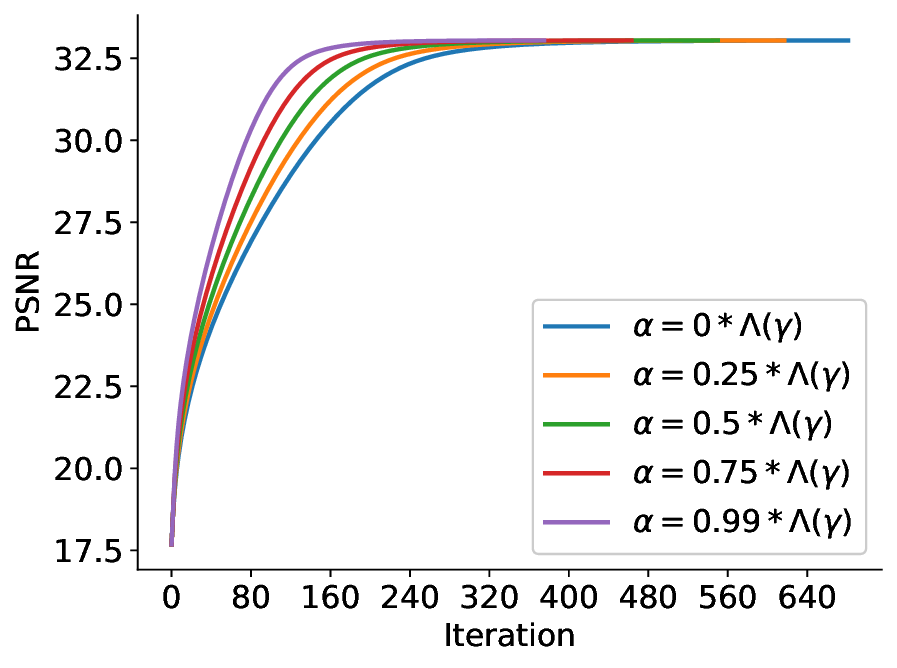}
		}\vspace{-0.03in}
		\centerline{} 
	\end{minipage} \hspace{0.15in}
  \begin{minipage}{0.3\linewidth}
		\centering
		\centerline{\includegraphics[width=2.in]{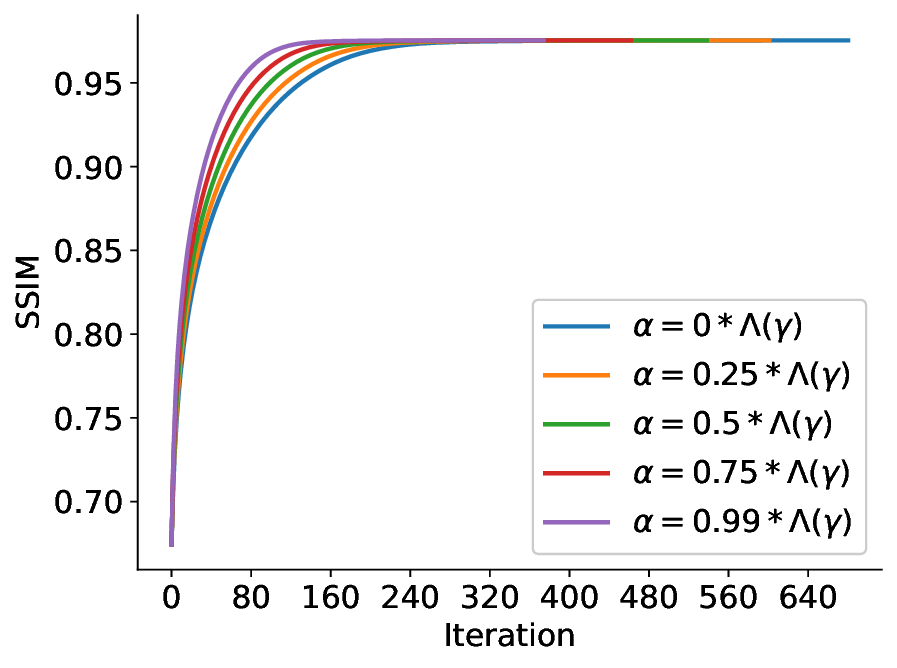}
		}\vspace{-0.03in}
		\centerline{} 
	\end{minipage} 
\vspace{-0.2cm} \caption{\h Effect of $\alpha$ in \cref{alg:buildtree} for solving DeTik model on `butterfly' with Ker1 and noise level $2.55$. Increasing the extrapolation parameter $\alpha$ speeds-up the convergence of the algorithm. This increased convergence speed does not alter the quality of the proposed restoration.
 }\label{fig:alpha}
 \end{figure}

In \cref{fig:alpha}, we report the effect of $\alpha$ on `butterfly' with Ker1 and noise level $2.55$. More specifically, the evolution curves of the convergence of residual $\left\|{\bf x}^{k+1}-{\bf x}^k\right\|$ at rate $\min_{j \leq k}\left\|{\bf x}^{j+1}-{\bf x}^j\right\|^2$, PSNR and SSIM values with respect to the number of iterations are presented, which showcases the advantage of the proposed extrapolation step. 
{\h Furthermore, the detailed results include iteration number (Iter.), computational time in seconds (Time(s)), recovered PSNR (dB), and SSIM for three tested images (butterfly, leaves, and starfish) in Sect3C with different levels of noise are reported in \cref{app_D}. 
From the presented results, we can see that \cref{alg:buildtree} exhibits improved performance as the extrapolation stepsize $\alpha$ increases, particularly in terms of computational cost. }
In our subsequent experiments, we set $\alpha=0.99*\Lambda(\gamma)$ for a given $\gamma$ to obtain results more efficiently.


\begin{figure}[b!]
\hspace{0.1in}
\begin{minipage}{0.23\linewidth}
		\centering
		\centerline{\includegraphics[width=1.9in]{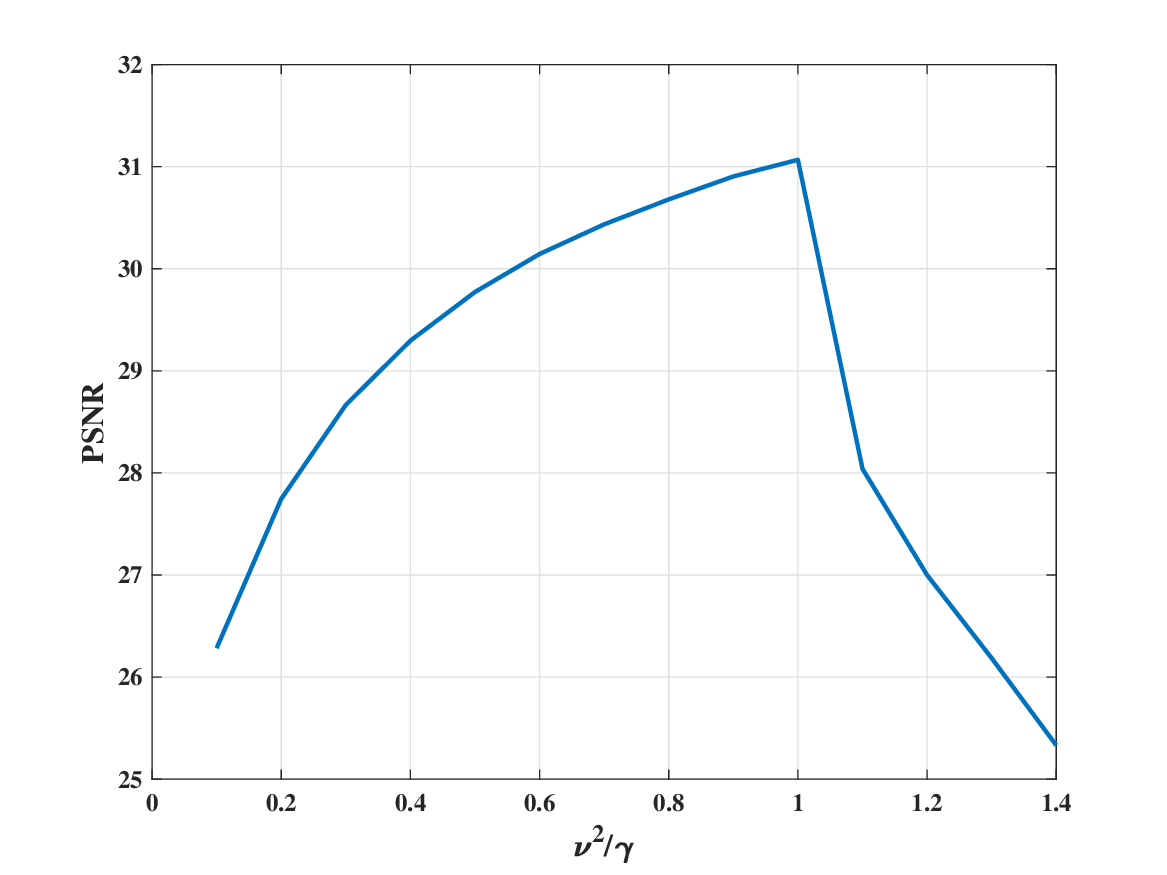}
		}\vspace{-0.03in}
		\centerline{\footnotesize{\h PSNR value along with $\gamma_\nu$}}
	\end{minipage}  \hspace{0.05in}
 \begin{minipage}{0.23\linewidth}
		\centering
		\centerline{\includegraphics[width=1.4in]{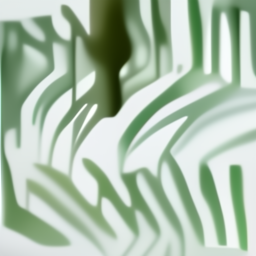}
		}\vspace{-0.03in}
		\centerline{\footnotesize{\y $\gamma_\nu = 0.01$}} 
	\end{minipage}
  \begin{minipage}{0.23\linewidth}
		\centering
		\centerline{\includegraphics[width=1.4in]{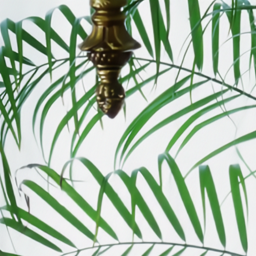}
		}\vspace{-0.03in}
		\centerline{\footnotesize{\y $\gamma_\nu = 1$}} 
	\end{minipage}
  \begin{minipage}{0.23\linewidth}
		\centering
		\centerline{\includegraphics[width=1.4in]{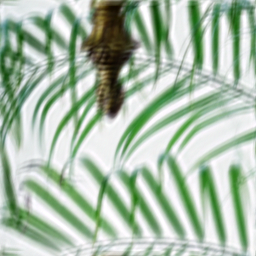}
		}\vspace{-0.03in}
		\centerline{\footnotesize{\y $\gamma_\nu = 1.3$}} 
	\end{minipage}
		\label{fig:pargamma}
  \caption{\y Influence of the parameter $\gamma_\nu=\frac{\nu^2}{\gamma}$ for deblurring with DeTik model. First column: average PSNR along with the $\gamma_\nu$. The other parameters are fixed. Remaining columns: visual results for deblurring `leaves' with various $\gamma_\nu$.}
\end{figure}
\begin{figure}[t!]
\hspace{0.1in}
	\begin{minipage}{0.23\linewidth}
		\centering
		\centerline{\includegraphics[width=1.9in]{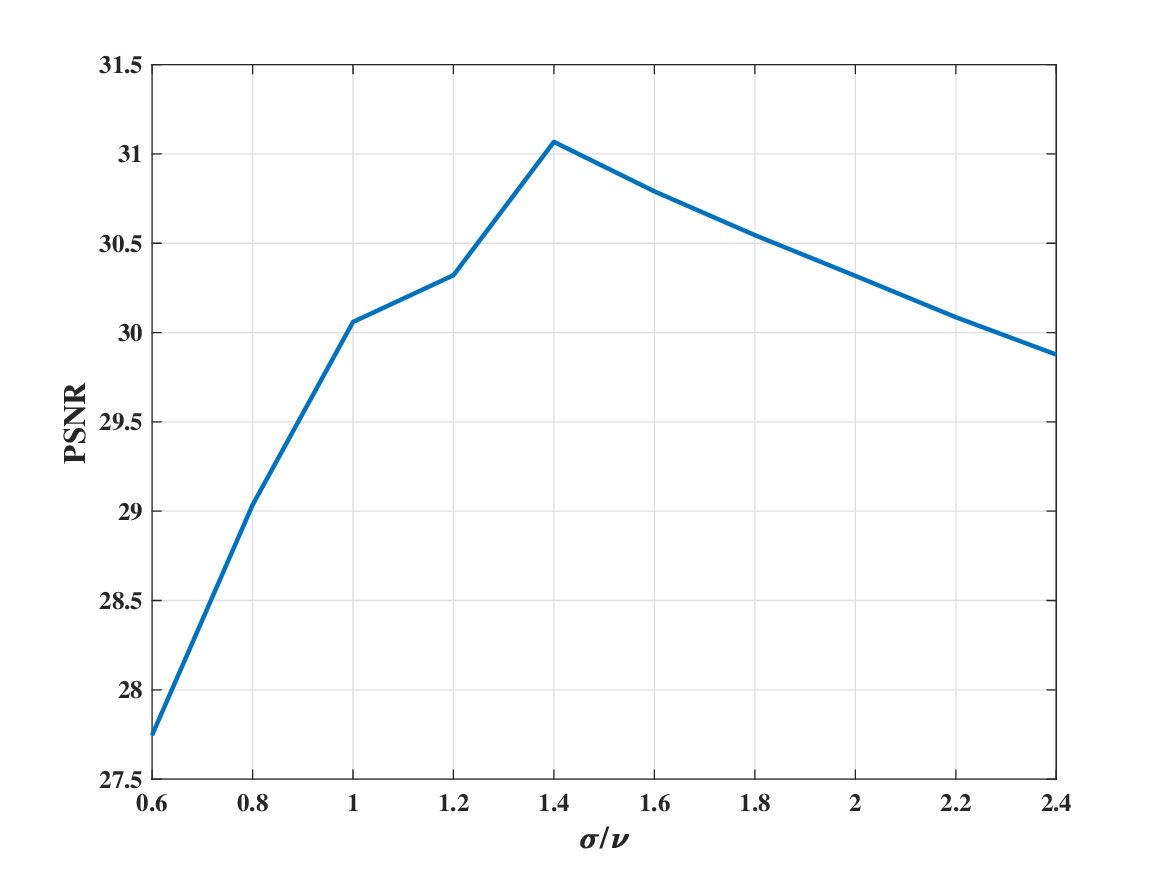}
		}\vspace{-0.03in}
		\centerline{\footnotesize{\h PSNR value along with $\sigma_\nu$}}
	\end{minipage}  \hspace{0.05in}
 \begin{minipage}{0.23\linewidth}
		\centering
		\centerline{\includegraphics[width=1.4in]{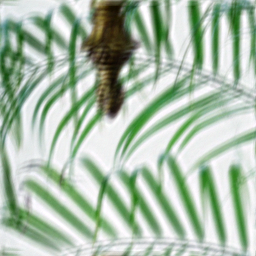}
		}\vspace{-0.03in}
		\centerline{\footnotesize{\y $\sigma_\nu = 0.6$}} 
	\end{minipage}
 \begin{minipage}{0.23\linewidth}
		\centering
		\centerline{\includegraphics[width=1.4in]{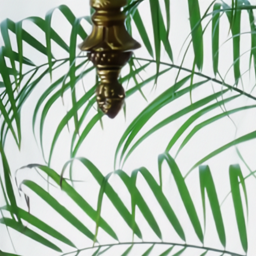}
		}\vspace{-0.03in}
		\centerline{\footnotesize{\y $\sigma_\nu = 1.4$}} 
	\end{minipage}
 \begin{minipage}{0.23\linewidth}
		\centering
		\centerline{\includegraphics[width=1.4in]{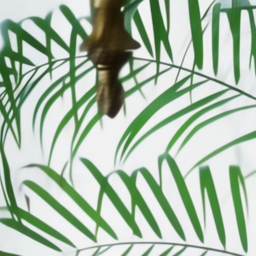}
		}\vspace{-0.03in}
		\centerline{\footnotesize{\y $\sigma_\nu = 10$}} 
	\end{minipage}
 \label{fig:parsigma}
\caption{\y Influence of  the parameter $\sigma_\nu$ for deblurring with DeTik model. First column: average PSNR along with the $\sigma_\nu$. The other parameters are fixed. Remaining columns: visual results for deblurring `leaves' with various $\sigma_\nu$.}
\end{figure}

{\y\subsection{Parameter analysis}
In this subsection, we study the influence of the parameters and initialization of \cref{alg:buildtree} for solving the DeTik model. Recall that DeTik  can be read as
\vspace{-0.3cm}\begin{equation}\nonumber
    \min_{{\bf x}\in\mathbb{R}^n}    \frac{1}{2\nu^2}\Vert A{\bf x} - {\bf b}\Vert^2+\frac{1}{\gamma}\phi_{\sigma}({\bf x})+ \frac{\beta}{2}\Vert {\bf x}\Vert^2,
\end{equation}
where $\nu$  and $\sigma$ are the noise levels of the synth input image and the denoiser $\phi_\sigma$, respectively. We fix model parameter $\beta $ for different noise levels as that in the last subsection, and roughly estimate $\sigma$ proportionally to the input noise level $\nu$ as $\sigma=\sigma_\nu*\nu$, where $\sigma_\nu$ is a positive constant. Consequently, the parameters we will be testing are $\gamma_\nu=\frac{\nu^2}{\gamma}$ and $\sigma_\nu=\sigma/\nu$.

In \cref{fig:pargamma}, we display the average PSNR value of Set3C using 10 tested blur kernels under a noise level of $2.55$, where $\gamma_\nu$ ranges from $0.1$ to $1.4$ with a step size of $0.1$. 
From the results, we can see that the instances with $\gamma_\nu$ values around 1 exhibit superior performance compared to other cases. This observation is further supported by the restored images on the right-hand side, which demonstrate that the quality of that corresponding to $\gamma_\nu=1$ is better than those for $\gamma_\nu=0.01$ and $\gamma_\nu=1.3$.
When $\gamma_\nu=0.01$, the noise is removed, but the blur remains. for a larger value of $\gamma_\nu=1.3$, both the noise and blur remain. Hence, in our experiments, we chose $\gamma_\nu=1$ to address the noise level of $2.55$.  
Next, we test the effect of the parameter $\sigma_\nu$ and present the average PSNR value of Set3C with 10 tested blur kernels under noise level $2.55$ for $\sigma_\nu\in [0.6, 2.4]$ with a step size $0.2$ in \cref{fig:parsigma}. 
The results indicate that almost no deblurring occurs when the value of $\sigma_\nu$ is small. Conversely, as $\sigma_\nu$ increases, excessive smoothing takes place, resulting in the loss of image details. Based on both the curve analysis and the visual outcomes, we select $\sigma_\nu\in[1,2]$.

We further investigate the impact of the initialization of \cref{alg:buildtree}.
In \cref{fig:ini}, we plot the average PSNR value of Set3C obtained from 10 tested blur kernels under a noise level of $2.55$. 
{\h
Due to the nonconvex regularizer, the proposed scheme is sensitive to initial value. Following the setting of \cite{hurault2022proximal}, the initial ${\bf x}^0$ is varied with different noise levels: $\{0.01, 2.55, 5, 7.5, 10\}$.}
Based on the PSNR curve and visual quality in \cref{fig:ini}, we can see that a suitable initial input is crucial for the image deblurring task. When an initial input closely resembles the ground truth image, certain images may not undergo further iterations and terminate prematurely, particularly when the stopping criteria remain unchanged.
On the other hand, if a heavily noisy image serves as the initial input, the iteration process progresses smoothly. However, the resulting image retains the heavy noise due to the low-level denoiser's inability to effectively handle such noise levels. 
In our experiments, we adopt the observation as the initial input to ensure the validity of the obtained results.
}

\begin{figure}[h!]
\hspace{0.1in}
	\begin{minipage}{0.23\linewidth}
		\centering
		\centerline{\includegraphics[width=1.9in]{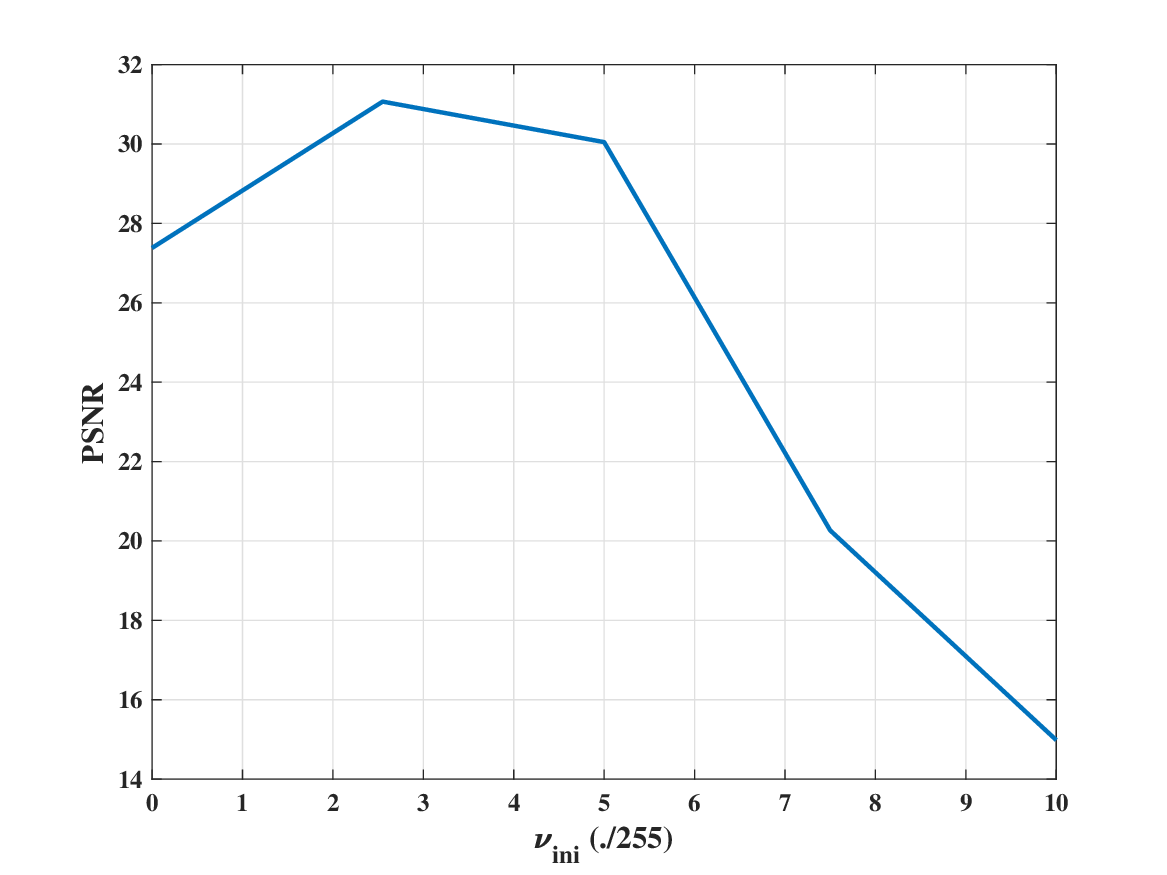}
		}\vspace{-0.03in}
		\centerline{\footnotesize{\h PSNR value along with $\nu_{init}$}}
	\end{minipage}  \hspace{0.05in}
 \begin{minipage}{0.23\linewidth}
		\centering
		\centerline{\includegraphics[width=1.4in]{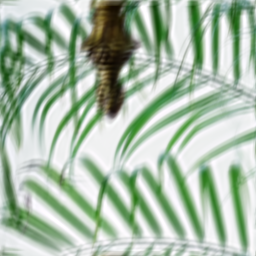}
		}\vspace{-0.03in}
		\centerline{\footnotesize{\y $\nu_{init} = 0.01/255$}} 
	\end{minipage}
 \begin{minipage}{0.23\linewidth}
		\centering
		\centerline{\includegraphics[width=1.4in]{forMR/img_1_33.98dB_Ite172_deblur_1.4.png}
		}\vspace{-0.03in}
		\centerline{\footnotesize{\y $\nu_{init} = 2.55/255$}} 
	\end{minipage}
 \begin{minipage}{0.23\linewidth}
		\centering
		\centerline{\includegraphics[width=1.4in]{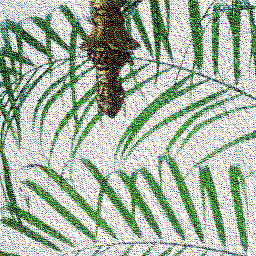}
		}\vspace{-0.03in}
		\centerline{\footnotesize{\y $\nu_{init} = 10/255$}} 
	\end{minipage}
 \label{fig:ini}
\caption{\y Influence of the initialitation ${\bf x}^0$ for deblurring with DeTik model. First column: average PSNR along with different ${\bf x}^0$. The other parameters are fixed. Remaining columns: visual results for deblurring `leaves' with various ${\bf x}^0$.}
\end{figure}

\subsection{Image deblurring and super-resolution}
In this subsection, we are devoted to demonstrating the effectiveness and robustness of the proposed \cref{alg:buildtree} and \cref{alg:3} by solving image deblurring and super-resolution problems. 

As discussed in \cref{subsec_4.1}, \cref{alg:buildtree} can be utilized to solve DeTik model due to the smoothness of $f_1({\bf x})=\frac{1}{2\nu^2}\Vert A{\bf x} - {\bf b}\Vert^2$, where $\nu$ is the noise level;
\cref{alg:3} can be used to solve DeBox model mentioned in \cref{image_deblurring2}. 
We first determine an appropriate $\gamma$ satisfying \cref{ass2} and set $\alpha=0.99*\Lambda(\gamma)$. 
We consider Gaussian noise with 3 noise levels $\nu \in \{2.55, 7.65, 12.75\} / 255$, i.e., $\nu \in \{0.01, 0.03, 0.05\}$, and 2 scale factors $\times 2, \times 3$. 
For the tested noise levels, we set $\sigma = \{1.4 \nu, 0.7\nu, 0.6\nu\}$, $\nu^2/\gamma = \{1, 0.9, 0.6\}$ in \cref{alg:buildtree} for both image deblurring and super-resolution.
For all noise levels, we set $\sigma = \{2\nu, 1\nu, 0.75\nu\}$ and $\nu^2/\gamma = \{5, 1.5, 1\}$ in \cref{alg:3} for both tasks. 
We test the proposed algorithms for different tasks and compare the numerical results recovered by DeTik and DeBox. 

\begin{table}[t!]
\caption{\h Numerical results (PSNR(dB)) of our DeTik and DeBox for image deblurring with Ker1 and 3 noise levels on Dataset Set3C.}\vspace{-0.1in}\label{table: deblur1}\centering
	\resizebox{0.87\hsize}{!}{
		\begin{tabular}{c|ccc|ccc|ccc}
			\hline
                Noise Level&  \multicolumn{3}{c|}{2.55} & \multicolumn{3}{c|}{7.65} &\multicolumn{3}{c}{12.75}   \\
			\hline
Images&Butterfly&Leaves&Starfish&Butterfly&Leaves&Starfish&Butterfly&Leaves&Starfish\\ \hline
Degraded&17.68&16.50&21.56
        &17.48&16.34&21.09
        &17.10&16.06&20.28\\
DeTik   &33.18&34.02&33.14
        &29.91&30.34&29.78
        &27.94&28.06&27.58\\
DeBox  &33.62&33.80&33.53
        &29.75&30.20&29.59
        &27.90&27.96&27.57\\
\hline
	\end{tabular}}
\end{table}

For the image deblurring task, we test four classical datasets, i.e., Set3C, Set14, Kodak24, and Set17, with different blur kernels and noise levels. {\h For the sake of brevity, we present the image deblur results of Ker1 with various noise levels on Set3C in \cref{table: deblur1}, and more results can be found in \cref{app_E}. Our proposed methods demonstrate competitive performance in the task of image deblurring across different noise levels.
}
On the other hand, the visual results of image `powerpoint2002' in Set14 degraded by the blur Ker6 and noise level $12.75$ can be found in \cref{fig:deblur3}. To assess the convergence of the proposed algorithms in the experimental aspect, the evolution and energy curves are plotted and presented alongside the corresponding recovered images.

For the image super-resolution task, we set the scale factor as $\times 2$ and $\times 3$. Meanwhile, the blur and noise (mentioned in the deblurring task) are also considered in the experiments.  
The image super-resolution results on datasets Set5, CBSD68, and Urban100 are reported in \cref{app_E}. {\h More specifically, we report the numerical results on Set5 in \cref{table: sr}.}
Furthermore, the visual results for noise level $7.65$ with blur Ker8 and scale factor $\times 2$ are shown in \cref{fig:sr2}. The evolution and energy curves demonstrate the convergence of the proposed approaches in the experiment, which aligns with our theoretical results.

\begin{table}[h]
\caption{\h Numerical results (PSNR(dB)) of our DeTik and DeBox for image super-resolution with Ker1 and 3 noise levels and scales $\times 2$ and $\times 3$  on Dataset Set5.}\vspace{-0.1in}\label{table: sr}\centering
	\resizebox{1\hsize}{!}{
		\begin{tabular}{cc|ccccc|ccccc|ccccc}
			\hline
                Scales&Noise Level&  \multicolumn{5}{c|}{2.55} & \multicolumn{5}{c|}{7.65} &\multicolumn{5}{c}{12.75}   \\
			\hline
&Images&Baby&Bird&Butterfly&Head&Woman
&Baby&Bird&Butterfly&Head&Woman
&Baby&Bird&Butterfly&Head&Woman\\ \hline
\multirow{3}*{$\times 2$}
&Degraded&28.82&24.73&17.75&25.52&22.73
         &27.61&24.23&17.64&24.93&22.41
         &25.90&23.36&17.43&23.94&21.82\\
&DeTik   &33.93&31.90&27.88&29.17&30.63
         &32.49&29.75&26.42&28.53&29.02
         &31.51&27.87&24.67&27.74&26.81\\
&DeBox  &34.26&31.85&27.19&29.19&30.51
         &32.45&29.53&26.16&28.32&28.89
         &31.42&27.85&24.77&27.63&26.95\\
\hline
\multirow{3}*{$\times 3$}
&Degraded&28.75&24.72&17.75&25.43&22.66
         &27.20&24.05&17.61&24.65&22.23
         &25.15&22.94&17.34&23.40&21.47\\
&DeTik   &32.40&29.17&22.51&28.29&27.44
         &31.51&27.59&23.68&27.77&26.47
         &30.46&26.05&22.35&27.20&25.14\\
&DeBox  &32.53&29.09&22.91&27.67&27.06
         &31.54&27.44&23.37&27.69&26.45
         &30.63&26.15&22.44&27.15&25.20\\
\hline
	\end{tabular}}
\end{table} 

	\begin{figure}[h]
		\centering
\begin{minipage}{0.231\linewidth}
			\centering
			\centerline{\includegraphics[width=1.25in]{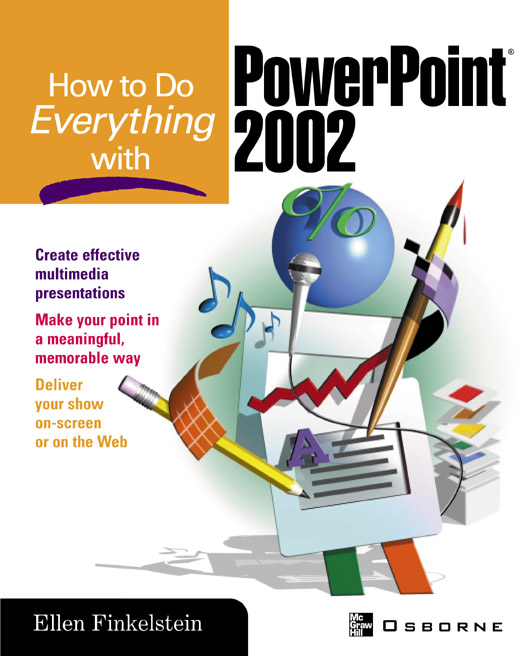}
			}\vspace{-0.03in}
			\centerline{\tiny{(a) Original}}
		\end{minipage}
                \centering
		\begin{minipage}{0.231\linewidth}
			\centering
			\centerline{\includegraphics[width=1.25in]{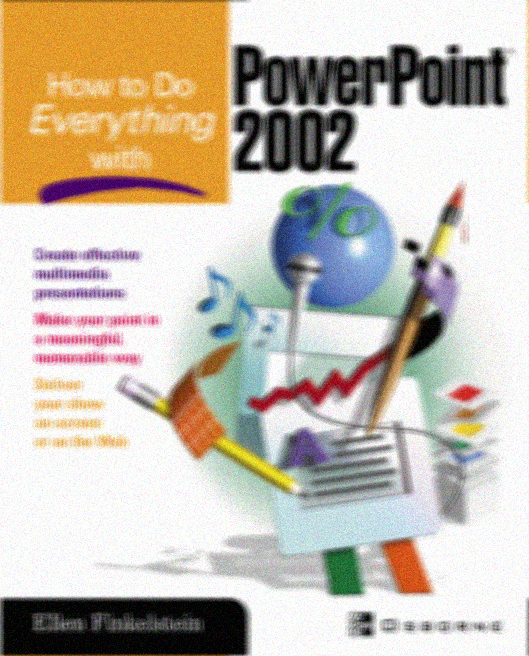}
			}\vspace{-0.03in}
			\centerline{\tiny{(b) Observed (18.14 dB)}}
		\end{minipage}
		\begin{minipage}{0.231\linewidth}
			\centering
			\centerline{\includegraphics[width=1.25in]{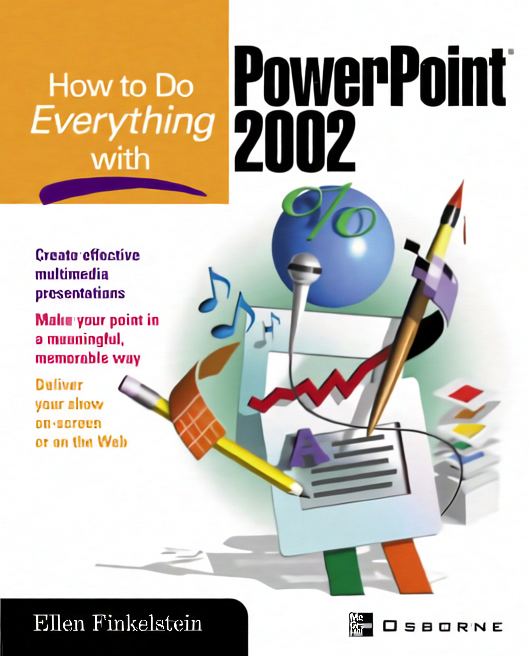}
			}\vspace{-0.03in}
			\centerline{\tiny{(c) DeTik (30.04 dB)}}
		\end{minipage}	
  		\begin{minipage}{0.231\linewidth}
			\centering
			\centerline{\includegraphics[width=1.25in]{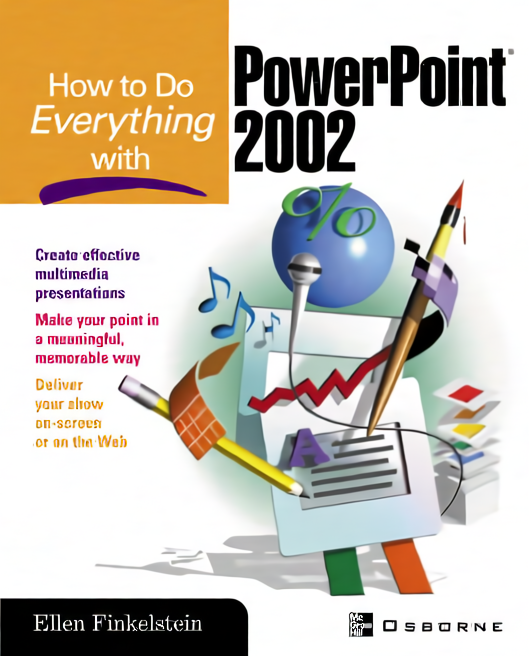}
			}\vspace{-0.03in}
			\centerline{\tiny{\h (d) DeBox (30.48 dB)}}
		\end{minipage}	
      		\begin{minipage}{0.231\linewidth}
			\centering
			\centerline{\includegraphics[width=1.55in]{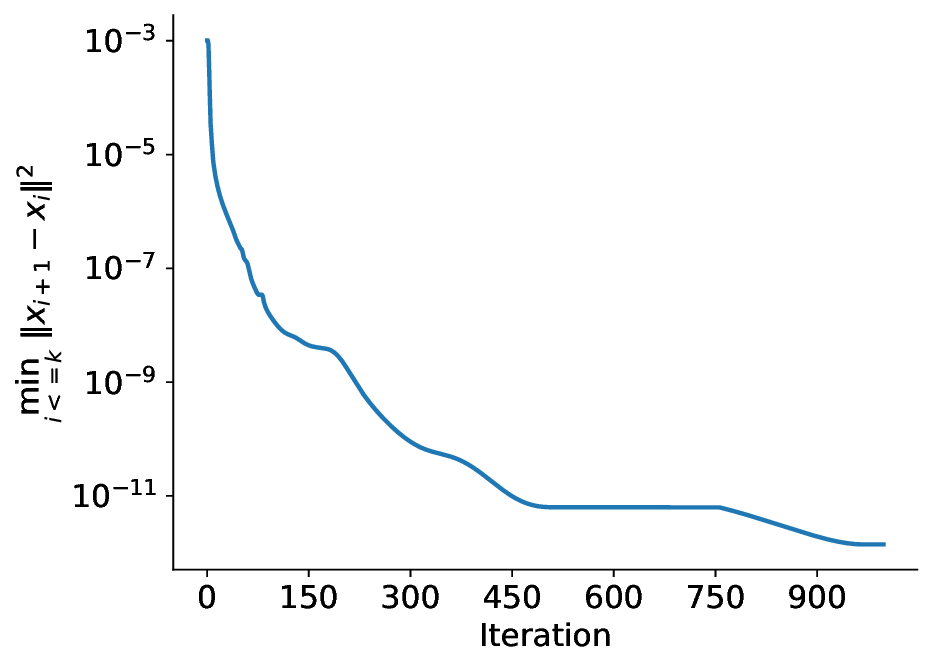}
			}\vspace{-0.03in}
			\centerline{\tiny{\h (e) Evolution of (c) }}
		\end{minipage}	
    		\begin{minipage}{0.231\linewidth}
			\centering
			\centerline{\includegraphics[width=1.55in]{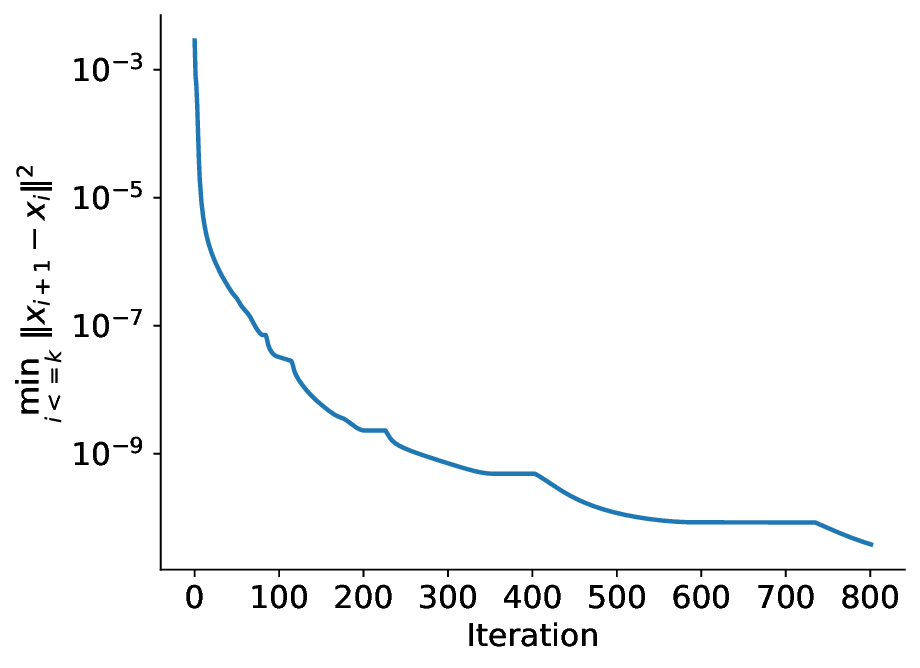}
			}\vspace{-0.03in}
			\centerline{\tiny{(f) \h Evolution of (d) }}
		\end{minipage}	
    		\begin{minipage}{0.231\linewidth}
			\centering
			\centerline{\includegraphics[width=1.55in]{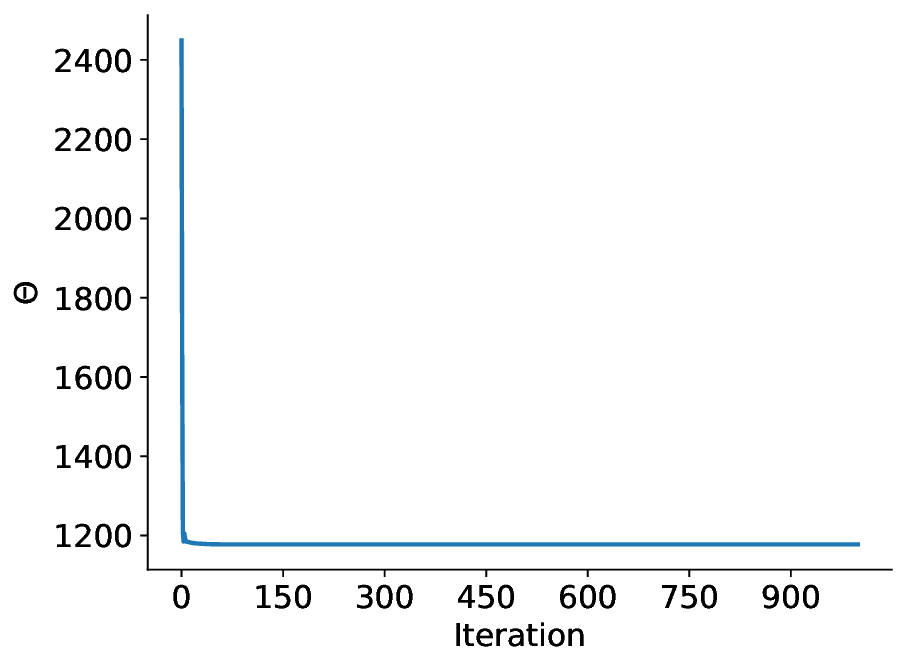}
			}\vspace{-0.03in}
			\centerline{\tiny{\h (g)   $\Theta_{\alpha,\gamma}$ value of (c)}}
		\end{minipage}
  		\begin{minipage}{0.231\linewidth}
			\centering
			\centerline{\includegraphics[width=1.55in]{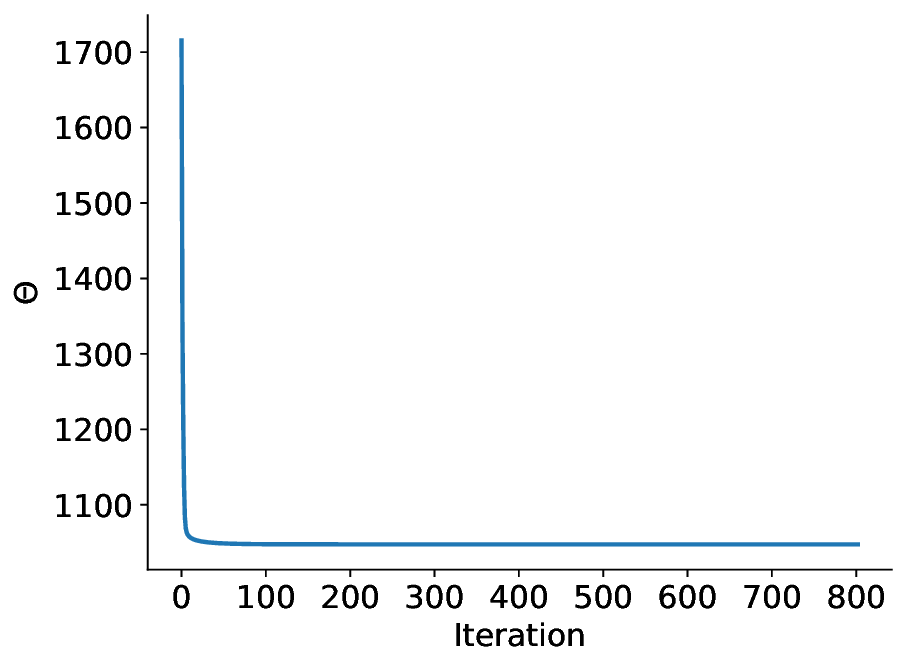}
			}\vspace{-0.03in}
			\centerline{\tiny{\h (h)  $\Theta_{\alpha,\gamma}$ value of (d) }}
		\end{minipage}	

		\caption{\y The deblurring results of DeTik and DeBox on image degradation with Ker6 and noise level $12.75$. The evolution and $\Theta_{\alpha,\gamma}$ value along with the number of iterations. }
		\label{fig:deblur3}
	\end{figure}

\begin{figure}[h]
		\centering
\begin{minipage}{0.231\linewidth}
			\centering
			\centerline{\includegraphics[width=1.25in]{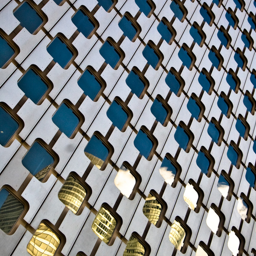}
			}\vspace{-0.03in}
			\centerline{\tiny{(a) Original}}
		\end{minipage}
                \centering
		\begin{minipage}{0.231\linewidth}
			\centering
			\centerline{\includegraphics[width=1.25in]{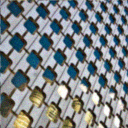}
			}\vspace{-0.03in}
			\centerline{\tiny{(b) Observed (18.03 dB)}}
		\end{minipage}
		\begin{minipage}{0.231\linewidth}
			\centering
			\centerline{\includegraphics[width=1.25in]{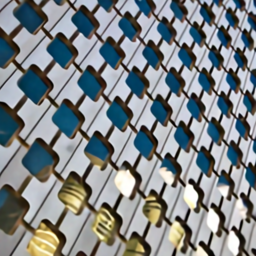}
			}\vspace{-0.03in}
			\centerline{\tiny{(c) DeTik (24.82 dB)}}
		\end{minipage}	
  		\begin{minipage}{0.231\linewidth}
			\centering
			\centerline{\includegraphics[width=1.25in]{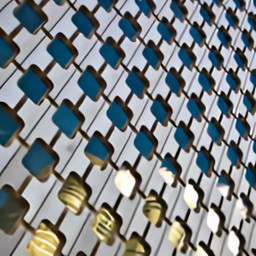}
			}\vspace{-0.03in}
			\centerline{\tiny{\h (d) DeBox (24.79 dB)}}
		\end{minipage}
      		\begin{minipage}{0.231\linewidth}
			\centering
			\centerline{\includegraphics[width=1.55in]{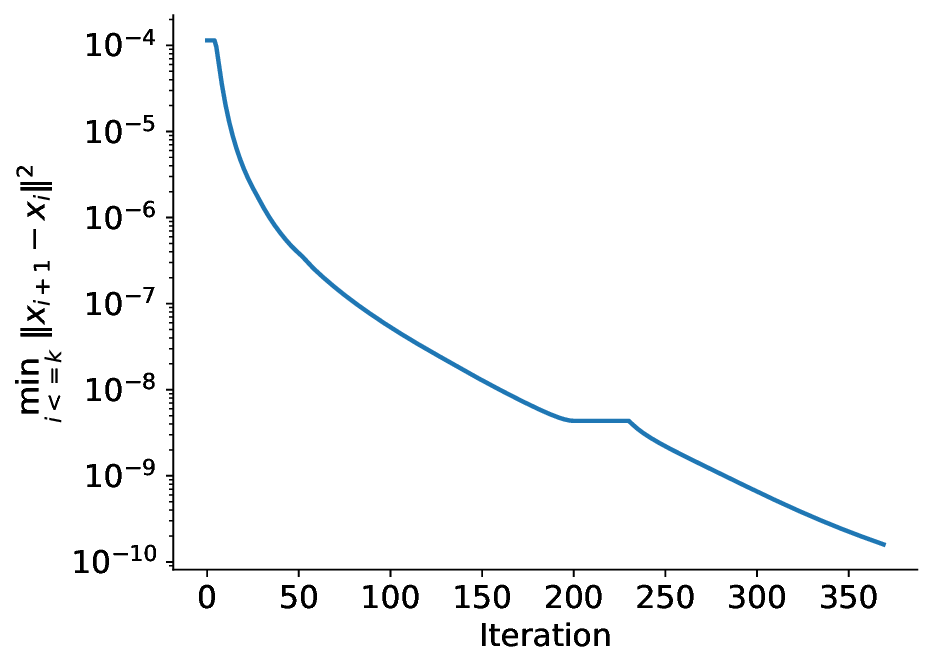}
			}\vspace{-0.03in}
			\centerline{\tiny{\h (e) Evolution of (c) }}
		\end{minipage}	
    		\begin{minipage}{0.231\linewidth}
			\centering
			\centerline{\includegraphics[width=1.55in]{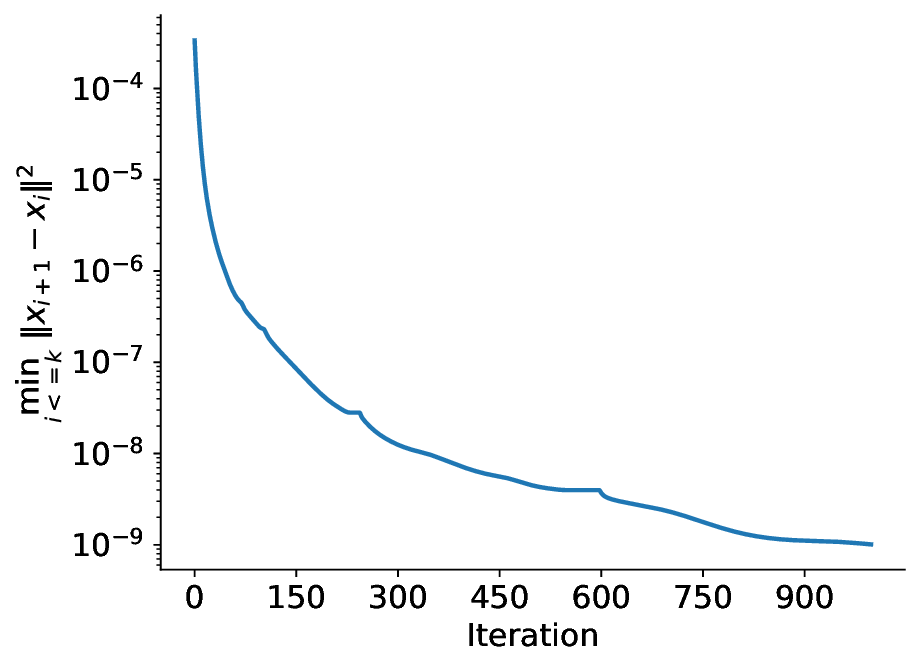}
			}\vspace{-0.03in}
			\centerline{\tiny{\h (f) Evolution of (d) }}
		\end{minipage}	
    		\begin{minipage}{0.231\linewidth}
			\centering
			\centerline{\includegraphics[width=1.55in]{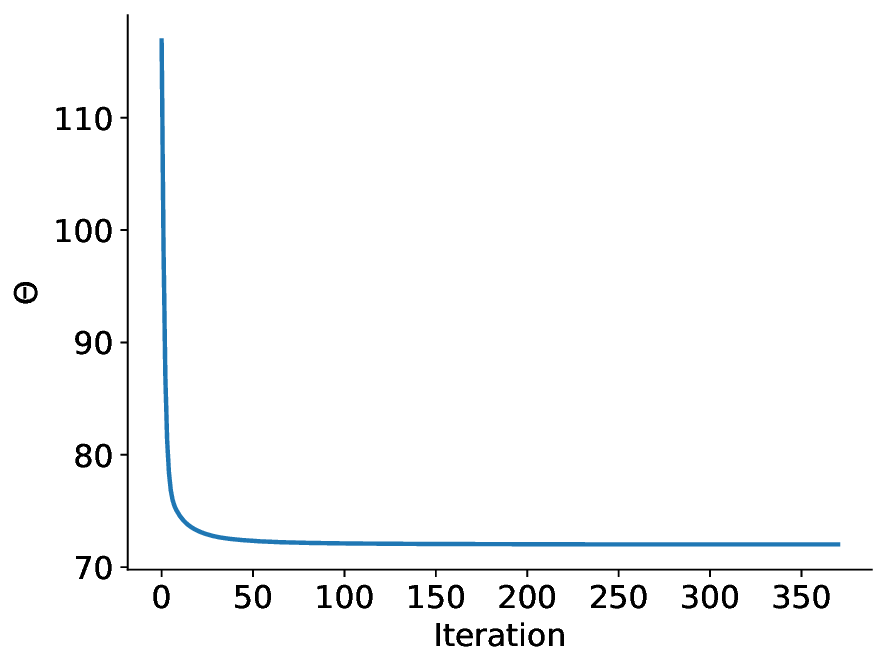}
			}\vspace{-0.03in}
			\centerline{\tiny{\h (g) $\Theta_{\alpha,\gamma}$ value of (c)}}
		\end{minipage}
  		\begin{minipage}{0.231\linewidth}
			\centering
			\centerline{\includegraphics[width=1.55in]{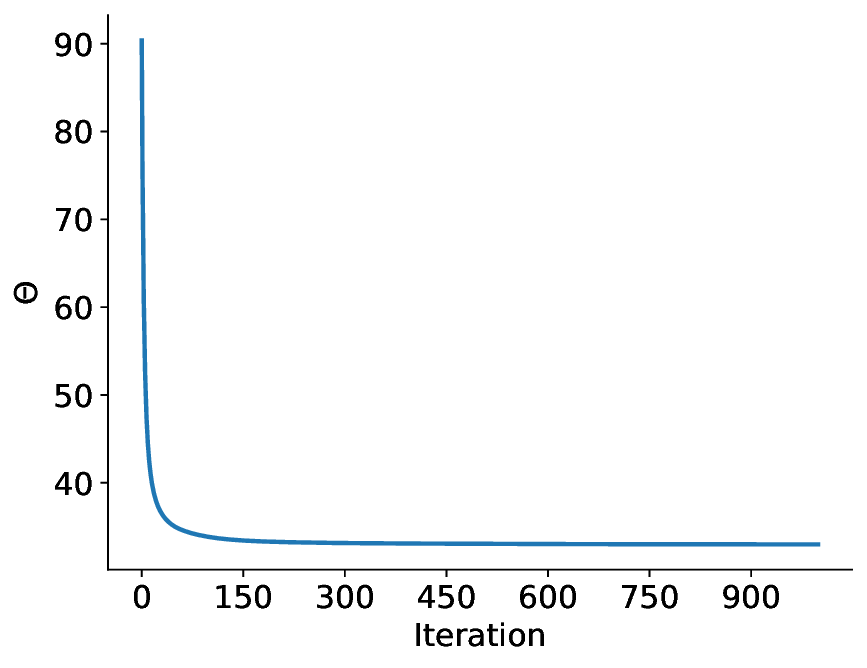}
			}\vspace{-0.03in}
			\centerline{\tiny{\h (h) $\Theta_{\alpha,\gamma}$ value  of (d) }}
		\end{minipage}

		\caption{\y The super-resolution results of DeTik and DeBox on image degradation with scale ($\times 2$) Ker8 and noise level $7.65$. The evolution and $\Theta_{\alpha,\gamma}$ value along with the number of iterations. }
		\label{fig:sr2}
	\end{figure}

\begin{table}[h]
\caption{Comparison on average image deblurring results (PSNR(dB)) of the state-of-the-art methods with our methods on Set3C, Set14, and Set17 datasets.}\vspace{-0.1in}\label{table: comp1}
\centering
	\resizebox{0.92\hsize}{!}{
		\begin{tabular}{c|c|ccccc|cc|cc}
                \hline 
                \multirow{2}*{Datasets}&\multirow{2}*{Noise Level}&\multirow{2}*{Degraded}&\multirow{2}*{DWDN}&\multirow{2}*
                {DP-IRCNN}&\multirow{2}*
                {DPIR}&\multirow{2}*
                {DREDDUN}&  Alg. \ref{alg3.1} & Alg. \ref{alg:buildtree}&  ADMM & Alg. \ref{alg:3}\\
                \cline{8-11}&&&
                &
                &
                
                &
                &TVTik&DeTik&TVBox&DeBox\\
			\hline 
			\multirow{3}*{Set3C}&{2.55}
                &19.93&30.92&30.92&{32.55}&30.71&29.46&30.98&28.84&{31.24}\\
                &{7.65}&19.52&28.62&27.60&28.60&28.62&25.10&{28.78}&25.18&28.62\\
                &{12.75}&18.84&26.92&25.93&26.80&26.97&23.34&{27.08}&23.39&{27.08}\\
                \hline
			\multirow{3}*{Set14}&{2.55}
                &22.82&31.08&30.64&{31.76}&{31.16}&28.47&30.17&27.68&30.08\\ 
                &{7.65}&22.10&28.41&28.13&{28.79}&{28.57}&26.68&28.47&26.06&28.33\\
                &{12.75}&21.03&27.20&27.03&{27.32}&27.38&25.30&27.30&25.10&{27.32}\\
                \hline 
			\multirow{3}*{Set17}&{2.55}
                &25.28&33.14&32.35&{33.98}&{33.41}&30.56&32.60&30.67&32.43\\
                &{7.65}
                &24.07&30.39&29.83&{30.64}&30.62&27.73&{30.64}&27.85&30.55\\
                &{12.75}
                &22.55&28.93&28.74&29.40&29.24&26.33&{29.25}&26.54&{29.29}\\
                \hline
\end{tabular}} 
\end{table}

\subsection{Comparison with state-of-the-art methods}
In the preceding subsections, we have substantiated the validity of the proposed algorithm in handling both smooth and non-smooth objective functions.  
However, these evaluations alone do not entirely showcase the advantage of our method. Hence, in this subsection, we conduct a comparative analysis with state-of-the-art methods to provide further evidence of the exceptional effectiveness of our approach.

	\begin{figure}[b!]
 \subfigure[\footnotesize{Original}]{
	\zoomincludgraphic{0.218\textwidth}{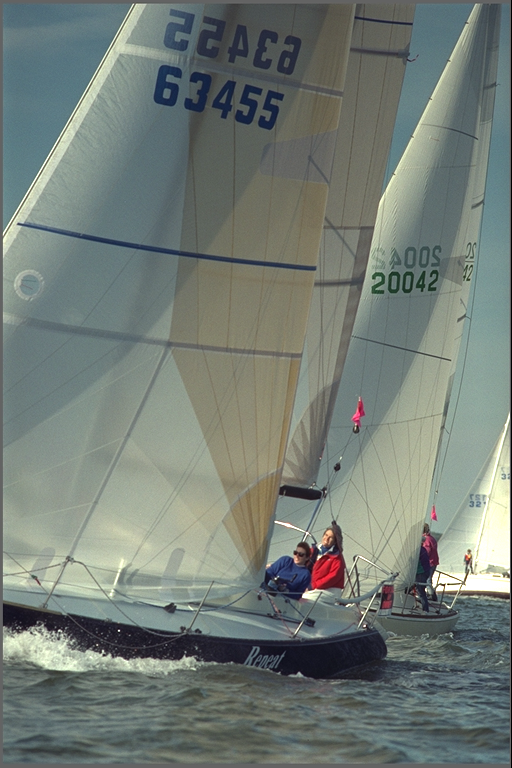}{0.7}{0.6}{0.9}{0.7}{3}{help_grid_off}{bottom_right}{line_connection_off}{2}{blue}{1}{red} 
	} \hspace{-0.28in}
    \subfigure[\footnotesize{Observed (23.02 dB)}]{
	\zoomincludgraphic{0.218\textwidth}{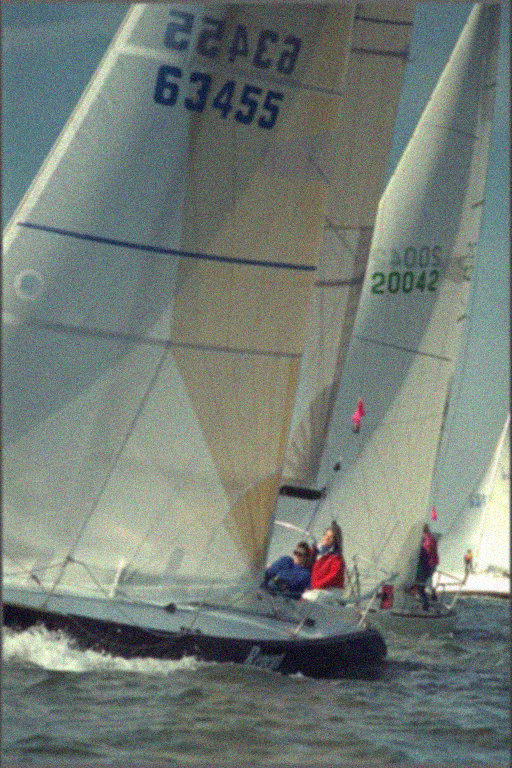}{0.7}{0.6}{0.9}{0.7}{3}{help_grid_off}{bottom_right}{line_connection_off}{2}{blue}{1}{red} }\hspace{-0.23in}
 \subfigure[\footnotesize{DWDN (31.05 dB)}]{	\zoomincludgraphic{0.218\textwidth}{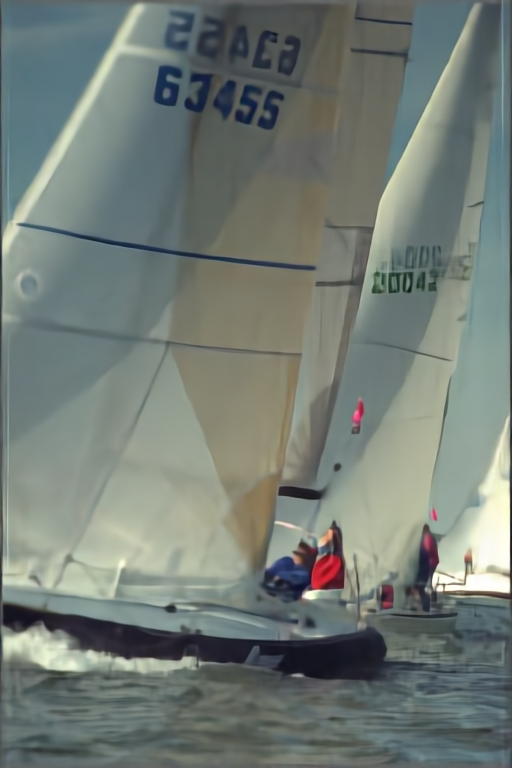}{0.7}{0.6}{0.9}{0.7}{3}{help_grid_off}{bottom_right}{line_connection_off}{2}{blue}{1}{red} 	}
\hspace{-0.28in}
      \subfigure[\footnotesize{DP-IRCNN (31.22 dB)}]{
	\zoomincludgraphic{0.218\textwidth}{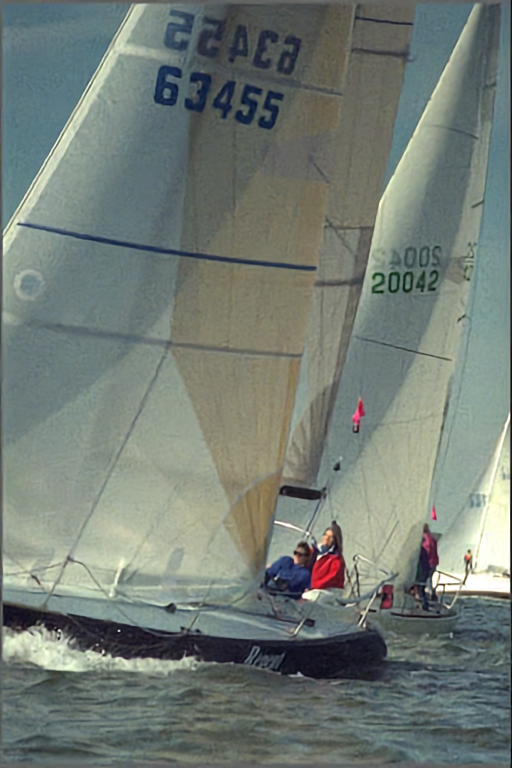}{0.7}{0.6}{0.9}{0.7}{3}{help_grid_off}{bottom_right}{line_connection_off}{2}{blue}{1}{red} }
 
    \subfigure[\footnotesize{DPIR (31.97 dB)}]{
	\zoomincludgraphic{0.218\textwidth}{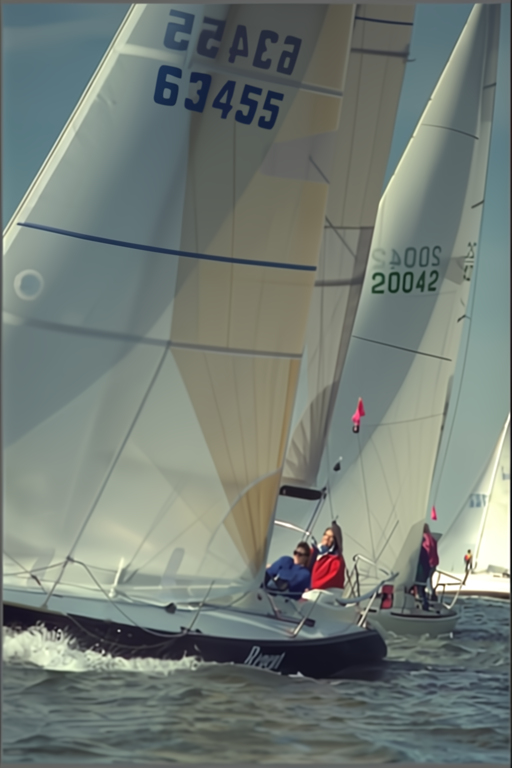}{0.7}{0.6}{0.9}{0.7}{3}{help_grid_off}{bottom_right}{line_connection_off}{2}{blue}{1}{red} 
	}\hspace{-0.23in}
    \subfigure[\footnotesize{DREDDUN (31.19 dB)}]{
	\zoomincludgraphic{0.218\textwidth}{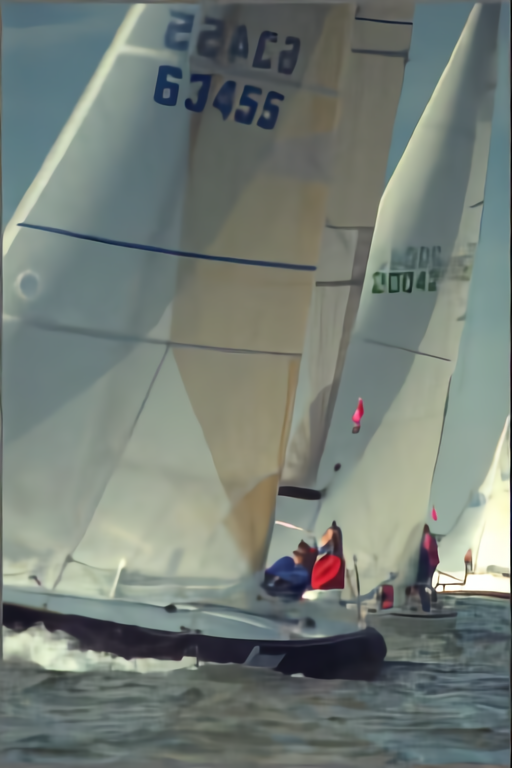}{0.7}{0.6}{0.9}{0.7}{3}{help_grid_off}{bottom_right}{line_connection_off}{2}{blue}{1}{red} 
	}\hspace{-0.23in}
    \subfigure[\footnotesize{DeTik (32.05 dB)}]{
	\zoomincludgraphic{0.218\textwidth}{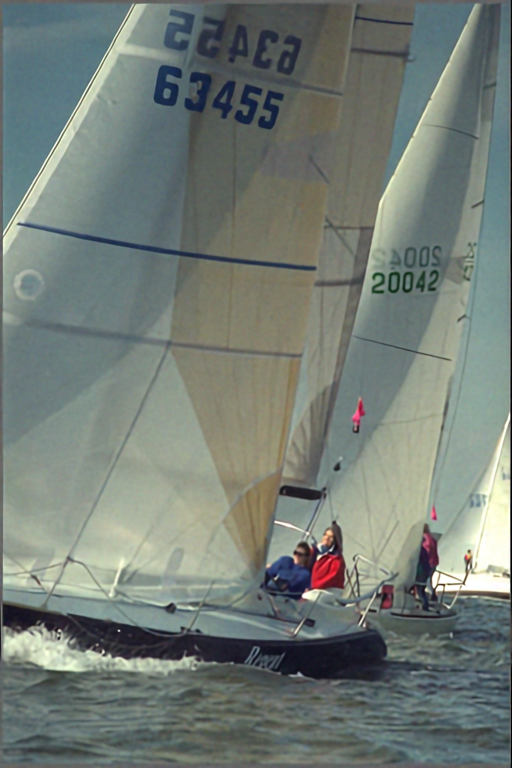}{0.7}{0.6}{0.9}{0.7}{3}{help_grid_off}{bottom_right}{line_connection_off}{2}{blue}{1}{red} 
	}\hspace{-0.23in}
  \subfigure[\footnotesize{DeBox (31.61 dB)}]{
	\zoomincludgraphic{0.218\textwidth}{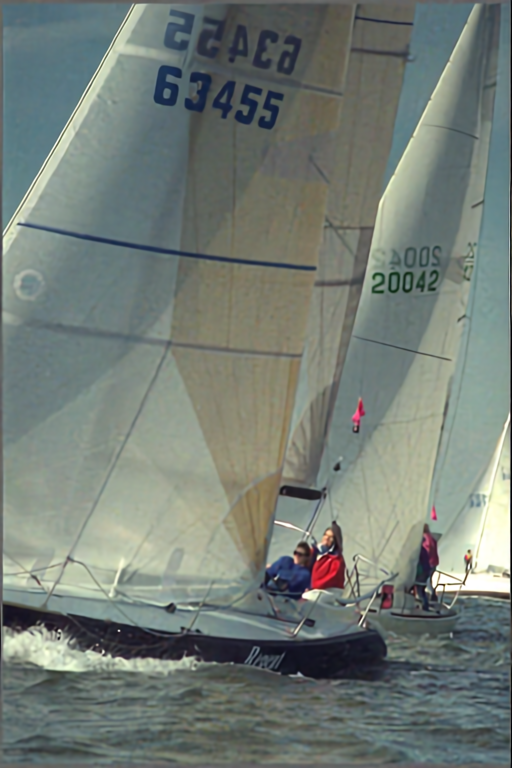}{0.7}{0.6}{0.9}{0.7}{3}{help_grid_off}{bottom_right}{line_connection_off}{2}{blue}{1}{red} 
	}
	\vspace{-0.3cm}	\caption{\y Image deblurring results with Ker9 and noise level 12.75. (g) is the result of proposed \cref{alg:buildtree} for DeTik; (h) is the result of the proposed \cref{alg:3} for DeBox. }
		\label{fig:compblur1}
	\end{figure}

\subsubsection{Comparisons with advanced deblurring models}
Following the implementation of the plug-and-play strategy, our proposed method integrates a denoiser into the objective function. 
Consequently, several methods that employ the same strategy are compared. 
While these methods yield competitive results, it is important to note that our proposed method holds a distinct advantage in terms of theoretical analysis. Specifically, our method guarantees convergence, whereas not all of the compared methods provide such a guarantee. In this paper, some plug-and-play methods and unrolling models 
DWDN \cite{dong2020deep}, 
DPIR \cite{zhang2021plug} with IRCNN \cite{zhang2017learning} (DP-IRCNN), DPIR \cite{zhang2021plug} with DRUNet (DPIR),
and DREDDUN \cite{9662672} are compared. All the compared codes were obtained either from the official published versions or were graciously provided by the authors themselves.

To provide more comprehensive results of the image deblurring, we compiled the average results for 10 blur kernels and 3 noise levels in
\cref{table: comp1}. We list the results of the proposed two algorithms with two cases, respectively. From the numerical results, it becomes evident that our DeTik and DeBox yield competitive performance compared to deep learning-based plug-and-play and unrolling methods. Nevertheless, it is important to note that the traditional TVTik and TVBox cases may exhibit less satisfactory results, which is understandable considering that deep learning-based models have the advantage of leveraging more prior information compared to traditional priors.
{\y Furthermore, the visual results are depicted in \cref{fig:compblur1} for a more comprehensive illustration. 
Note that we only present {\h our PnP-based results (DeTik and DeBox)} for visual comparison. We can see that although the PnP-based methods usually cause over-smoothing, the proposed algorithms (DeTik and DeBox) exhibit superior performance in detail restoration compared to the other methods.}

\begin{table}[t!]
\caption{Comparison on average  image super-resolution results (PSNR(dB)) of the state-of-the-art methods with our methods on  Set5 and Urban100 datasets.}\vspace{-0.1in}\label{table: comp2}
\centering
	\resizebox{0.92\hsize}{!}{
		\begin{tabular}{c|c|c|ccccc|cc|cc}
                \hline 
                \multirow{2}*{Scales}&\multirow{2}*{Datasets}&\multirow{2}*{Noise Level}&\multirow{2}*{Bicubic}&\multirow{2}*{USRNet}&\multirow{2}*{DP-IRCNN}&\multirow{2}*{DPIR}&\multirow{2}*{DREDDUN}& Alg. \ref{alg3.1} & Alg. \ref{alg:buildtree}&  ADMM & Alg. \ref{alg:3}\\
                \cline{9-12}&&&&&&&&TVTik&DeTik&TVBox&DeBox\\
			\hline 
			\multirow{6}*{$\times 2$}&\multirow{3}*{Set5}&{2.55}
                &24.21&30.75&29.33&31.07&30.49&28.16&30.29&27.70&{30.51}\\
                &&{7.65}&23.48&29.38&27.76&28.81&28.46&26.59&{29.16}&26.04&{29.12}\\
                &&{12.75}&22.45&27.98&26.96&27.60&27.34&23.26&27.91&24.40&{27.99}\\
                \cline{2-12}
			&\multirow{3}*{Urban100}&{2.55}
                &19.15&{25.67}&25.34&{25.40}&{25.43}&21.23&24.10&21.51&23.86\\
                &&{7.65}&18.93&24.49&23.69&{24.52}&23.81&19.86&{24.34}&20.80&23.18\\
                &&{12.75}&18.53&22.92&22.68&23.18&{22.89}&19.24&{23.29}&19.81&{22.34}\\
                \hline
                \multirow{6}*{$\times 3$}&\multirow{3}*{Set5}&{2.55}
                &23.29&30.11&27.99&28.95&{28.55}&25.79&28.20&26.09&28.42\\
                &&{7.65}&22.71&28.19&26.52&27.22&27.11&25.19&{27.65}&25.72&27.64\\
                &&{12.75}&21.84&27.04&25.68&26.18&26.14&24.57&{26.61}&25.03&26.67\\
                \cline{2-12}
			&\multirow{3}*{Urban100}&{2.55}
                &18.54&24.03&22.80&23.62&23.12&21.52&23.14&20.45&21.72\\
                &&{7.65}
                &18.35&22.12&21.90&22.36&21.67&20.05&21.65&19.92&21.45\\
                &&{12.75}&18.00&20.93&20.37&20.91&20.91&19.16&20.70&19.40&20.93\\
                \hline
\end{tabular}}
\end{table} 

\subsubsection{Comparisons with advanced super-resolution models}
For image super-resolution task, USRNet \cite{zhang2020deep}, IRCNN \cite{zhang2017learning} (DP-IRCNN), DPIR \cite{zhang2021plug} with DRUNet (DPIR),
and DREDDUN \cite{9662672} are compared. All the compared codes used in our study were obtained either from the official published versions or were graciously provided by the authors themselves.
Note that when addressing the image super-resolution task with sample scales $\times 2$ and $\times 3$, we simulated the degraded images by incorporating blur and noise during the sampling process. Specifically, we added 10 blur kernels and introduced the 3 Gaussian noises mentioned earlier. 

The average image super-resolution results of the proposed algorithms with other advanced super-resolution models are listed in \cref{table: comp2}. We can see that our methods achieve competitive results under different scaling factors. While it is true that some compared methods outperform the proposed algorithm in some degradation cases, it is important to note that most of these methods lack convergence guarantees. Furthermore, we conducted a visual comparison of the renderings in \cref{fig:compsr1}, in which the proposed methods exhibit distinct advantages. Our proposed method excels in detail recovery when compared to other methods. Hence, based on both theoretical guarantees and experimental evidence, the algorithms we proposed exhibit distinct advantages when applied to image super-resolution tasks.
	\begin{figure}[t!]
		\centering
 \subfigure[\footnotesize{Original}]{
	\zoomincludgraphic{0.218\textwidth}{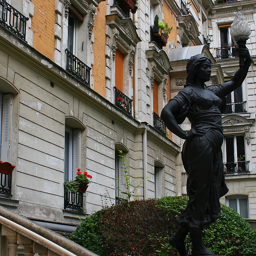}{0.01}{0.6}{0.3}{0.2}{1.5}{help_grid_off}{up_right}{line_connection_off}{2}{blue}{1}{red} 
	} \hspace{-0.23in}
    \subfigure[\footnotesize{Observed (18.56 dB)}]{
	\zoomincludgraphic{0.218\textwidth}{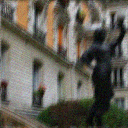}{0.01}{0.6}{0.3}{0.2}{1.5}{help_grid_off}{up_right}{line_connection_off}{2}{blue}{1}{red} 
	}\hspace{-0.23in}
      \subfigure[\footnotesize{USRNet (22.28 dB)}]{
	\zoomincludgraphic{0.218\textwidth}{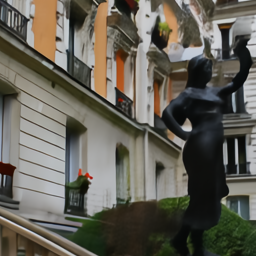}{0.01}{0.6}{0.3}{0.2}{1.5}{help_grid_off}{up_right}{line_connection_off}{2}{blue}{1}{red} 
	}\hspace{-0.23in}
    \subfigure[\footnotesize{DP-IRCNN (22.10 dB)}]{
	\zoomincludgraphic{0.218\textwidth}{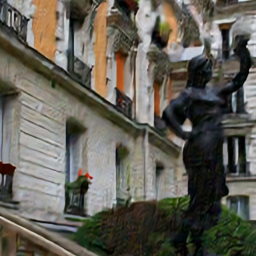}{0.01}{0.6}{0.3}{0.2}{1.5}{help_grid_off}{up_right}{line_connection_off}{2}{blue}{1}{red} 
	}\vspace{-0.1in}
 
    \subfigure[\footnotesize{DPIR (22.51 dB)}]{
	\zoomincludgraphic{0.218\textwidth}{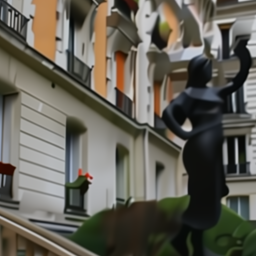}{0.01}{0.6}{0.3}{0.2}{1.5}{help_grid_off}{up_right}{line_connection_off}{2}{blue}{1}{red} 
	}\hspace{-0.23in}
    \subfigure[\footnotesize{DREDDUN (21.53 dB)}]{
	\zoomincludgraphic{0.218\textwidth}{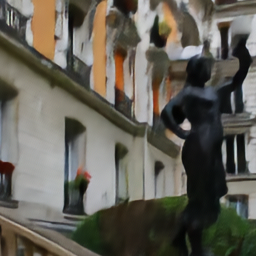}{0.01}{0.6}{0.3}{0.2}{1.5}{help_grid_off}{up_right}{line_connection_off}{2}{blue}{1}{red} 
	}\hspace{-0.23in}
    \subfigure[\footnotesize{DeTik (22.84 dB)}]{
	\zoomincludgraphic{0.218\textwidth}{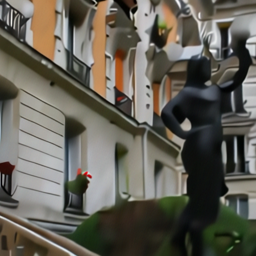}{0.01}{0.6}{0.3}{0.2}{1.5}{help_grid_off}{up_right}{line_connection_off}{2}{blue}{1}{red} 
	}\hspace{-0.23in}
  \subfigure[\footnotesize{DeBox (22.62 dB)}]{
	\zoomincludgraphic{0.218\textwidth}{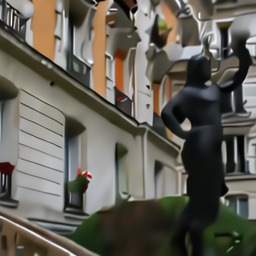}{0.01}{0.6}{0.3}{0.2}{1.5}{help_grid_off}{up_right}{line_connection_off}{2}{blue}{1}{red} 
	}\vspace{-0.1in}
	 	\caption{Image super-resolution results with scale $\times 2$, Ker1 and noise level 7.65. (g) is the result of proposed \cref{alg:buildtree} for DeTik; (h) is the results of the proposed \cref{alg:3} for DeBox. }
		\label{fig:compsr1}\vspace{-0.2in}
	\end{figure}
 
\section{Conclusions} \label{sec:conclusions}
This paper studied an extrapolated three-operator splitting method for solving a class of structural nonconvex optimization problems that minimize the sum of three functions. Our method extends the Davis-Yin splitting approach, which encompasses the widely-used forward-backward and Douglas-Rachford splitting methods, and introduces extrapolation techniques to handle nonconvex optimization problems. The convergence to a stationary point has been established by leveraging the Kurdyka-{\L}ojasiewicz property.
To further enhance the applicability, we applied the proposed splitting method within the Plug-and-Play (PnP) approach, incorporating a learned denoiser. The extrapolated PnP-based splitting methods replace the regularization step with a denoiser based on gradient step-based techniques, and we have provided theoretical guarantees for their convergence. This integration allows us to leverage the power of learning-based models.
Furthermore, we have conducted extensive numerical experiments to evaluate the performance of our proposed methods on image deblurring and super-resolution problems. The results of these experiments have demonstrated the advantages and efficiency of the extrapolation strategy employed in our algorithmic framework. Importantly, our experiments have highlighted the superiority of the learning-based model with the PnP denoiser in terms of image quality.

In future research, we will consider the variants of the proposed method, such as incorporating line search, inexact solving techniques, and dynamically adapting parameter choices, to extend the applicability of our framework to a broader range of practical problems. Further theoretical investigations are warranted to establish convergence guarantees for splitting methods combined with other efficient PnP denoisers, such as the Bregman-based denoiser proposed in \cite{hurault2023convergent2} for various Poisson inverse problems. {\h In addition, investigating the potential applications of the proposed methods in the field of medical image processing is a crucial aspect of our future work.}

\appendix
{
\section{Experimental results on effect of extrapolation}\label{app_D}
We report the average image deblurring results under 10 different blur kernels and 3 noise levels in \cref{table: alpha1}, which include iteration number (Iter.), computational time in seconds (Time(s)), recovered PSNR (dB), and SSIM for three tested images (butterfly, leaves, and starfish) in Sect3C with different levels of noise. From the presented results, we can see that \cref{alg:buildtree} exhibits improved performance as the extrapolation stepsize $\alpha$ increases, particularly in terms of computational cost. Increasing the extrapolation parameter $\alpha$ speeds-up the convergence of the algorithm. This increased convergence speed does not alter the quality of the proposed restoration. 

\begin{table}[h]
\caption{\y Parameter analysis of $\alpha$  in \cref{alg:buildtree} for image deblurring by DeTik model on the dataset Set3C with different noise levels.}\vspace{-0.1in}\label{table: alpha1}
\centering
	\resizebox{0.9\hsize}{!}{
		\begin{tabular}{c|c|ccccccccccc}
                \hline 
               \multirow{2}*{$\alpha$}&  {Image} &\multicolumn{3}{c}{butterfly}&&\multicolumn{3}{c}{leaves}&&\multicolumn{3}{c}{starfish}\\
			\cline{2-5}\cline{7-9}\cline{11-13} 
                &  {Noise Level} & 2.55 & 7.65 &12.75 &&2.55 & 7.65 &12.75&&2.55 & 7.65 &12.75\\
                \hline
                 \multirow{4}*{0}&Iter.  &681 &1001  &512  & &436&596 &428   &&388 &396  &513  \\
                                  &Time(s) &25.97&36.79 &18.87 &&15.54&20.92 &15.61 &&13.56&13.75 &18.30    \\
                &PSNR &33.18 &29.91 &27.94  && 33.97 & 30.33  &28.05  && 33.11 &29.78  &27.57  \\
                  & SSIM & 0.9760 & 0.9569 &0.9367 &&  0.9890 &  0.9760 & 0.9617 && 0.9551& 0.9233 & 0.8866 \\
                 \hline                 
                 \multirow{4}*{$  0.25*\Lambda(\gamma)$}&Iter.& 617&972&457&&383&532 &381  &&340 &351  &417  \\
                                                 &Time(s)&22.54  &37.87 &17.03&&13.13&19.00 &13.72  &&11.42&12.40  &14.74  \\
                  &PSNR  &33.18 &29.91 &27.94 &&33.97  &30.33  &28.05 &&33.11  &29.78   &27.57  \\
                   & SSIM &0.9760&0.9569&0.9367&&0.9890 &0.9760 &0.9617&&0.9551 &0.9233  &0.8866  \\
                  \hline                  
                 \multirow{4}*{$ 0.50*\Lambda(\gamma)$}&Iter.& 550&901 &403 &&332&467&226   &&297 &308  &396   \\
                  &Time(s)&20.07&36.18 &15.07    &&11.79 &16.32 &12.07    &&10.11 &10.51 &13.88  \\
                   &PSNR &33.18 &29.91 &27.94 &&33.97 &30.33 &28.05  &&33.11  &29.78 &27.57  \\
                    &SSIM&0.9760 &0.9569 &0.9367&&0.9890&0.9760&0.9617 &&0.9551 &0.9233&0.8866 \\
                    \hline                    
                 \multirow{4}*{$  0.75*\Lambda(\gamma)$}&Iter.&463 &873 &348 &&279&405&317     &&251 &265  &507  \\
                  &Time(s)&16.77 &37.86&12.46    &&10.06&14.28 &10.99 &&8.71&9.35 &17.20   \\
                   &PSNR &33.18  &29.91 &27.94   &&33.9 &30.33  &28.05  &&33.11 &29.78  &27.58  \\
                     &SSIM &0.9760 &0.9569 &0.9367   &&0.9890&0.9760 &0.9617 &&0.9551 &0.9233 &0.8866 \\
                    \hline                    
                 \multirow{4}*{$  0.99*\Lambda(\gamma)$}&Iter.&375&861 &491   &&225&344&26 &&204&223 &276 \\
                  &Time(s)&13.63&32.35 &18.41     &&7.66&12.10&9.21  &&6.69 &7.69  &9.31   \\
                   &PSNR  &33.18&29.91&27.95   &&33.97&30.33 &28.05  &&33.14&29.78&27.57\\
                    &SSIM  &0.9760&0.9569&0.9367   &&0.9890&0.9760 &0.9617  &&0.9551&0.9233&0.8866\\
            \hline
\end{tabular}}\vspace{-0.05in}
\end{table}

\begin{figure}[t!]
\centering
\begin{minipage}{0.38\linewidth}
		\centering
		\centerline{\includegraphics[height=1.51in,width=2.5in]{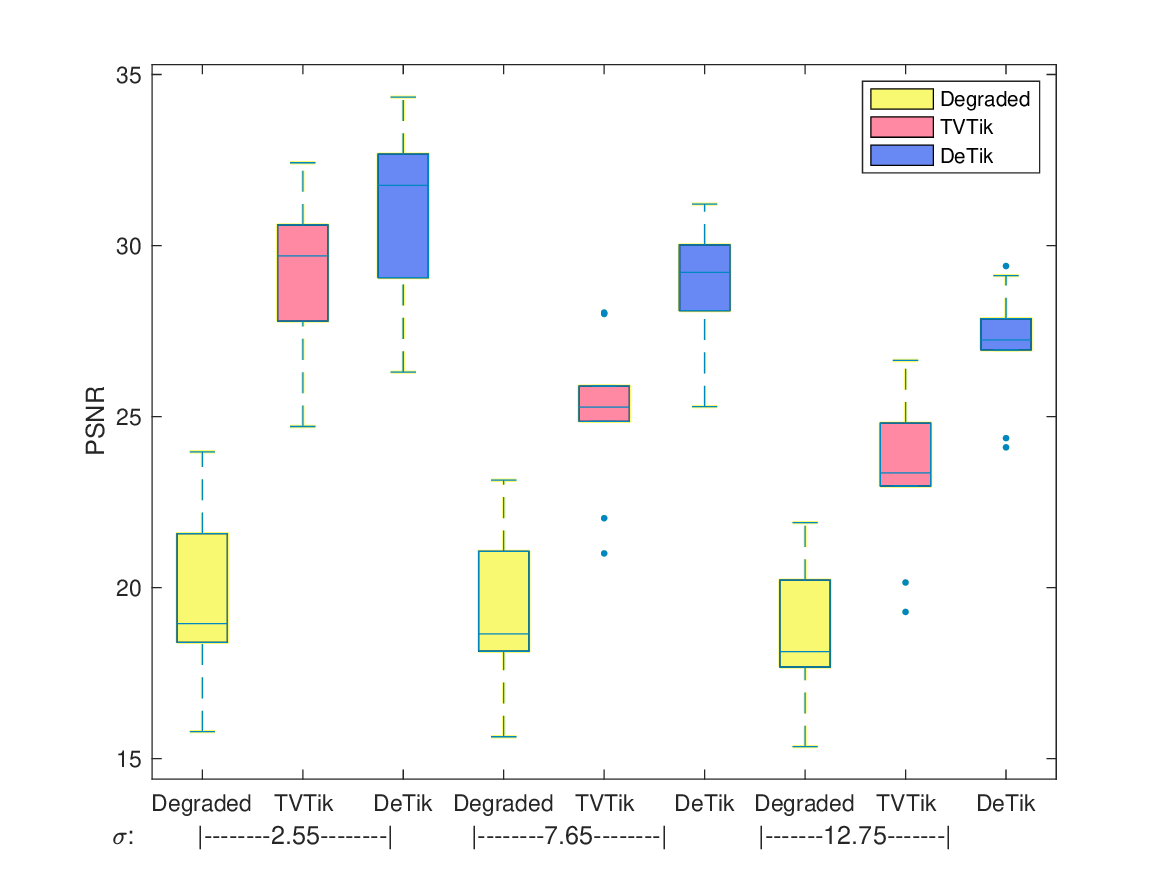}
		}\vspace{-0.03in}
		\centerline{\footnotesize{\y Set3C}}
	\end{minipage} 
 \begin{minipage}{0.38\linewidth}
		\centering
		\centerline{\includegraphics[height=1.51in,width=2.5in]{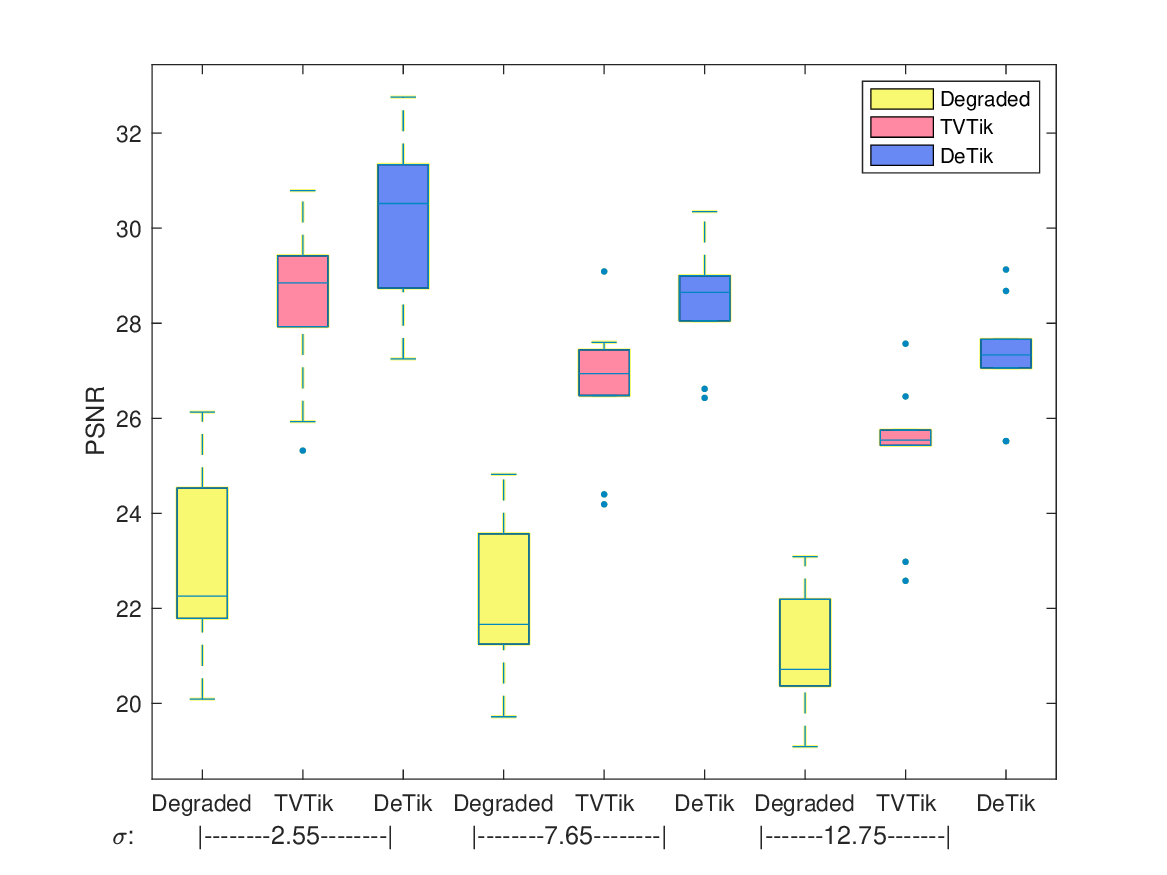}
		}\vspace{-0.03in}
		\centerline{\footnotesize{\y Set14}} 
	\end{minipage}

 \hspace{0.05in}
  \begin{minipage}{0.38\linewidth}
		\centering
		\centerline{\includegraphics[height=1.51in,width=2.5in]{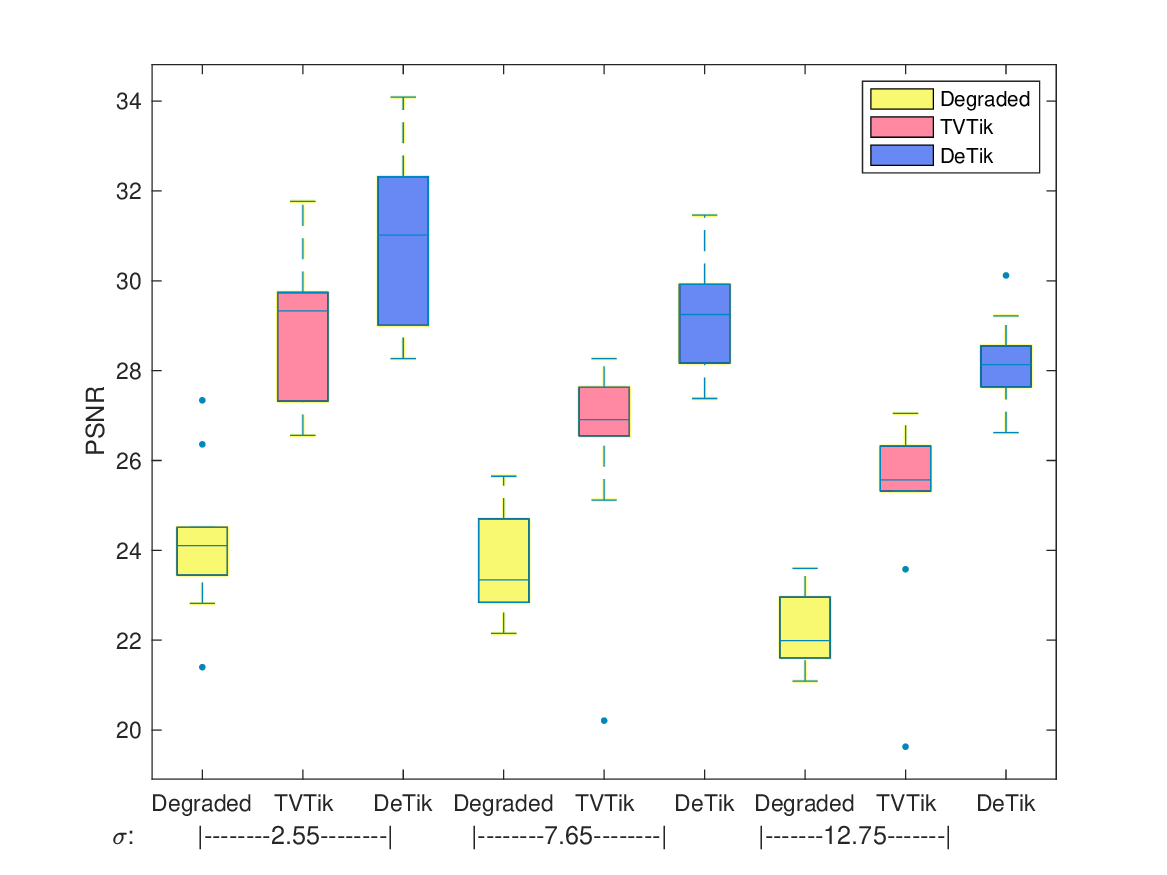}
		}\vspace{-0.03in}
		\centerline{\footnotesize{\y Kodak24}} 
	\end{minipage}
  \begin{minipage}{0.38\linewidth}
		\centering
		\centerline{\includegraphics[height=1.51in,width=2.5in]{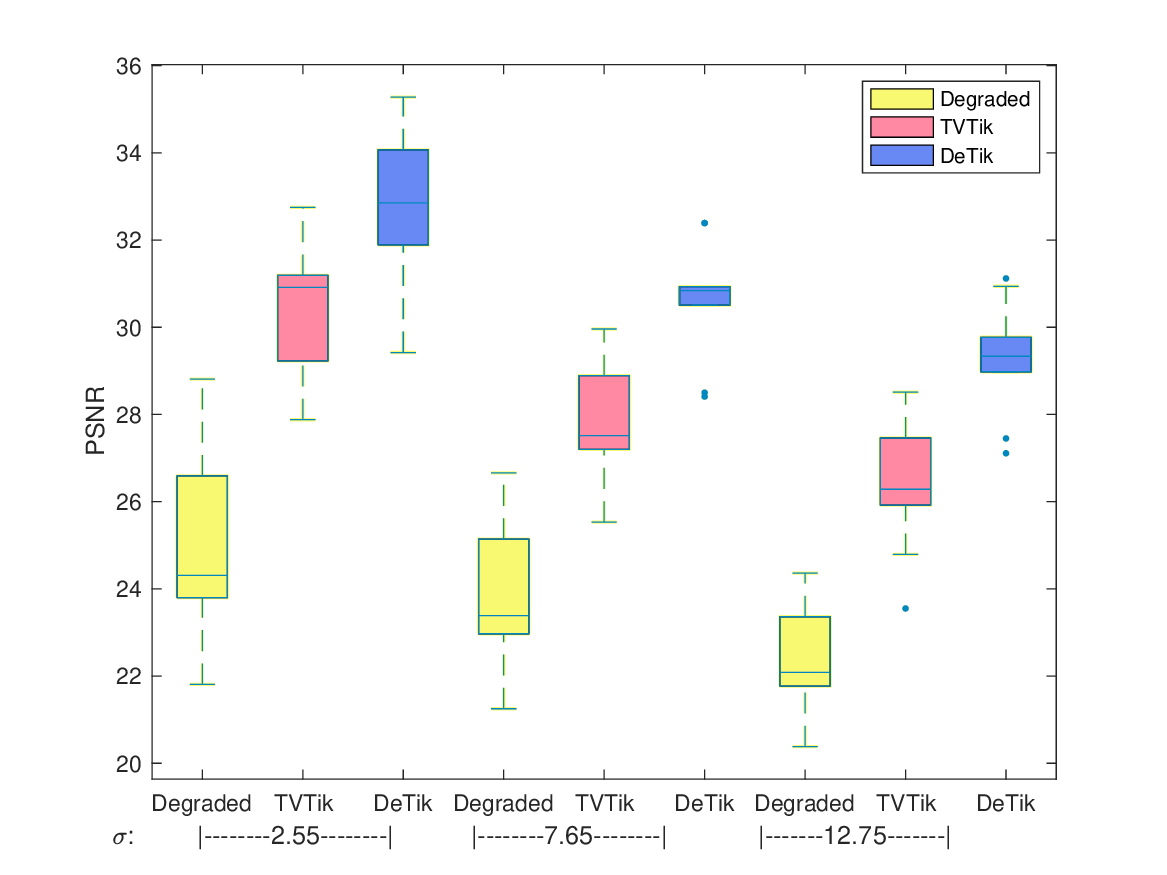}
		}\vspace{-0.03in}
		\centerline{\footnotesize{\y Set17}} 
	\end{minipage}
		\label{table: deblur}
  \caption{\y Average results (PSNR(dB)) of TVTik and DeTik for image deblurring with 10 different blur kernels and 3 noise levels on   Set3C, Set14, Kodak24, and Set17 datasets. }
\end{figure}

\begin{figure}[b!]
\begin{minipage}{0.314\linewidth}
		\centering
		\centerline{\includegraphics[height=1.51in,width=2.3in]{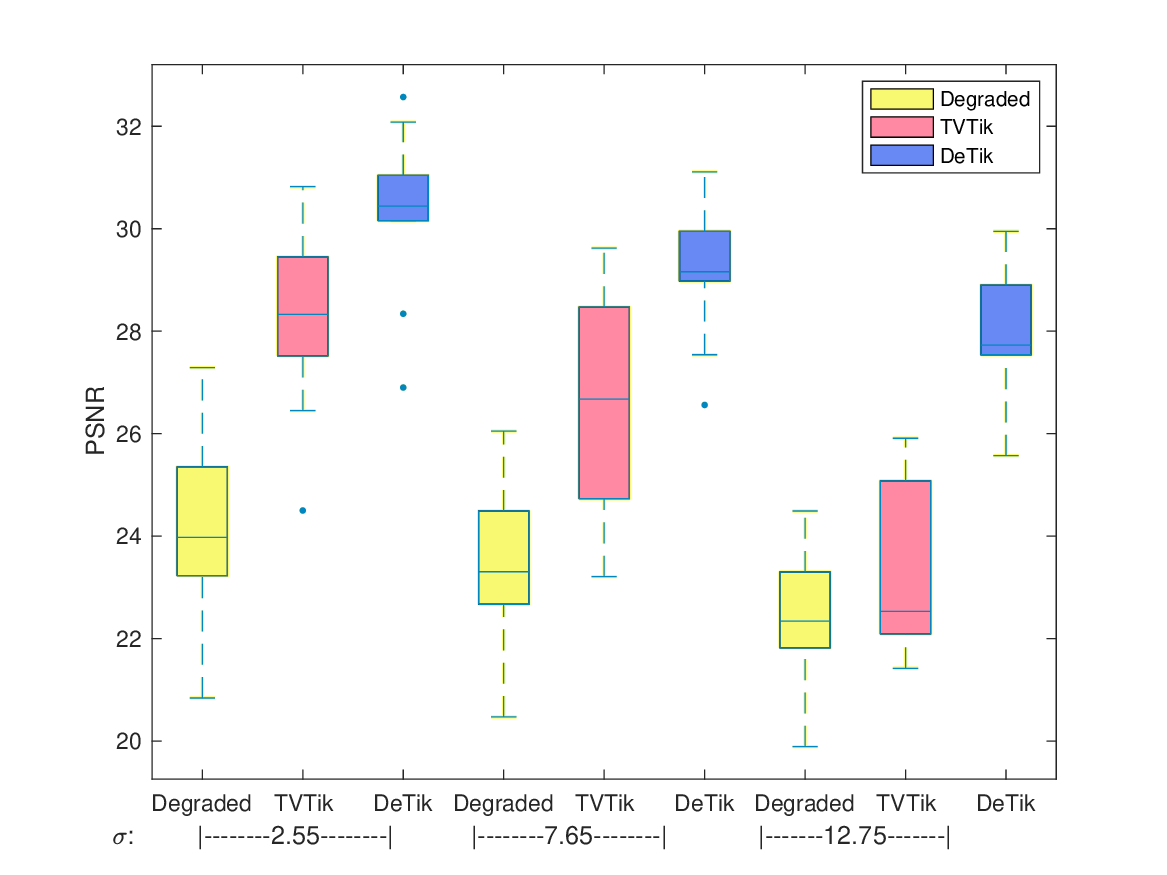}
		}\vspace{-0.03in}
		\centerline{\footnotesize{\y Set5 ($\times 2$)}}
	\end{minipage}\hspace{0.1in}
  \begin{minipage}{0.314\linewidth}
		\centering
		\centerline{\includegraphics[height=1.51in,width=2.3in]{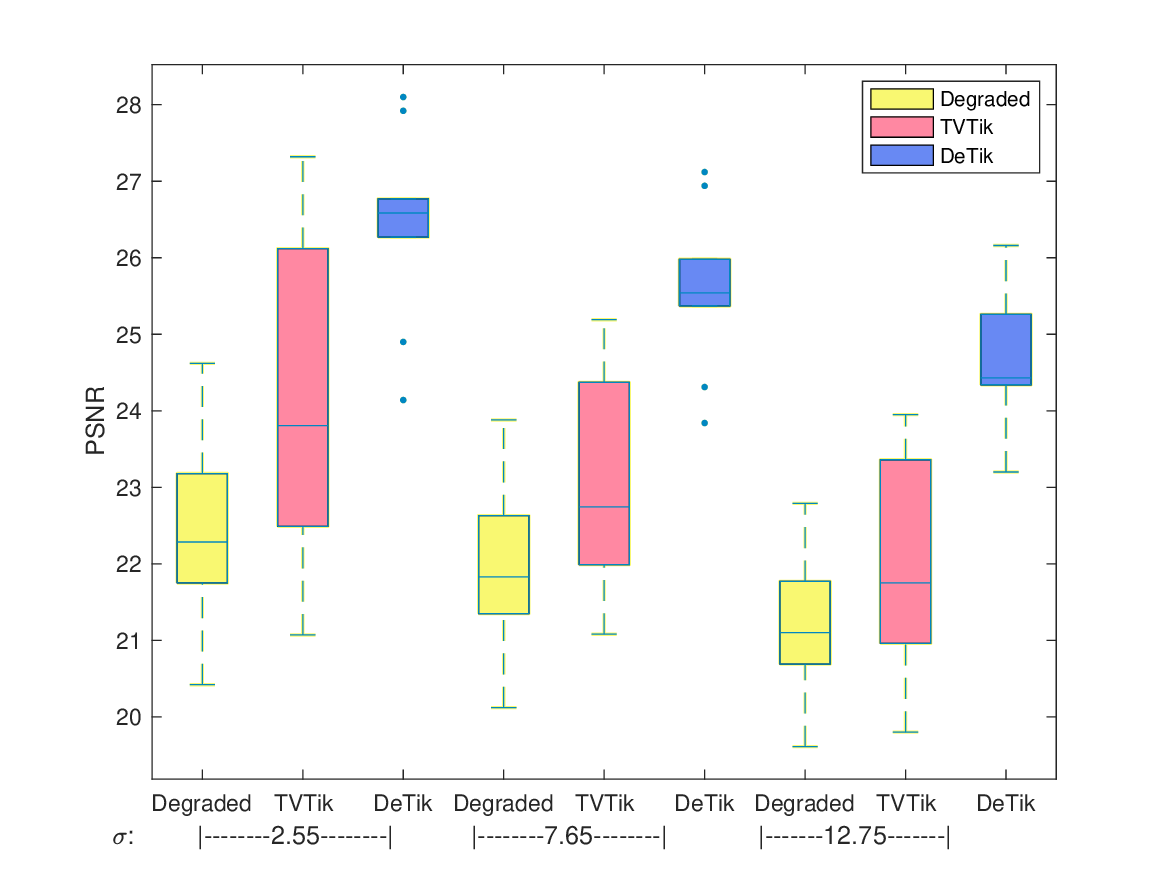}
		}\vspace{-0.03in}
		\centerline{\footnotesize{\y CBSD68 ($\times 2$)}} 
	\end{minipage}\hspace{0.1in}
   \begin{minipage}{0.314\linewidth}
		\centering
		\centerline{\includegraphics[height=1.51in,width=2.3in]{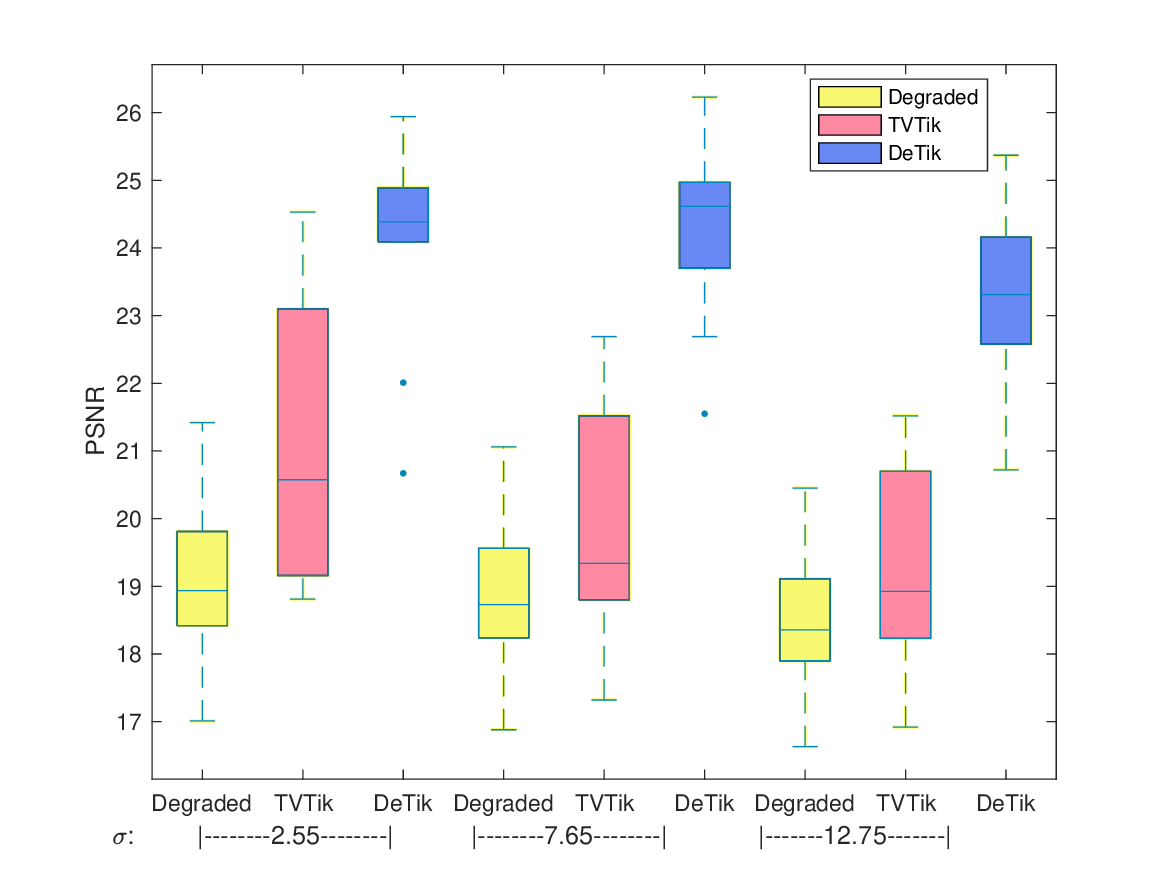}
		}\vspace{-0.03in}
		\centerline{\footnotesize{\y Urban100 ($\times 2$)}} 
	\end{minipage}

   \begin{minipage}{0.314\linewidth}
		\centering
		\centerline{\includegraphics[height=1.51in,width=2.3in]{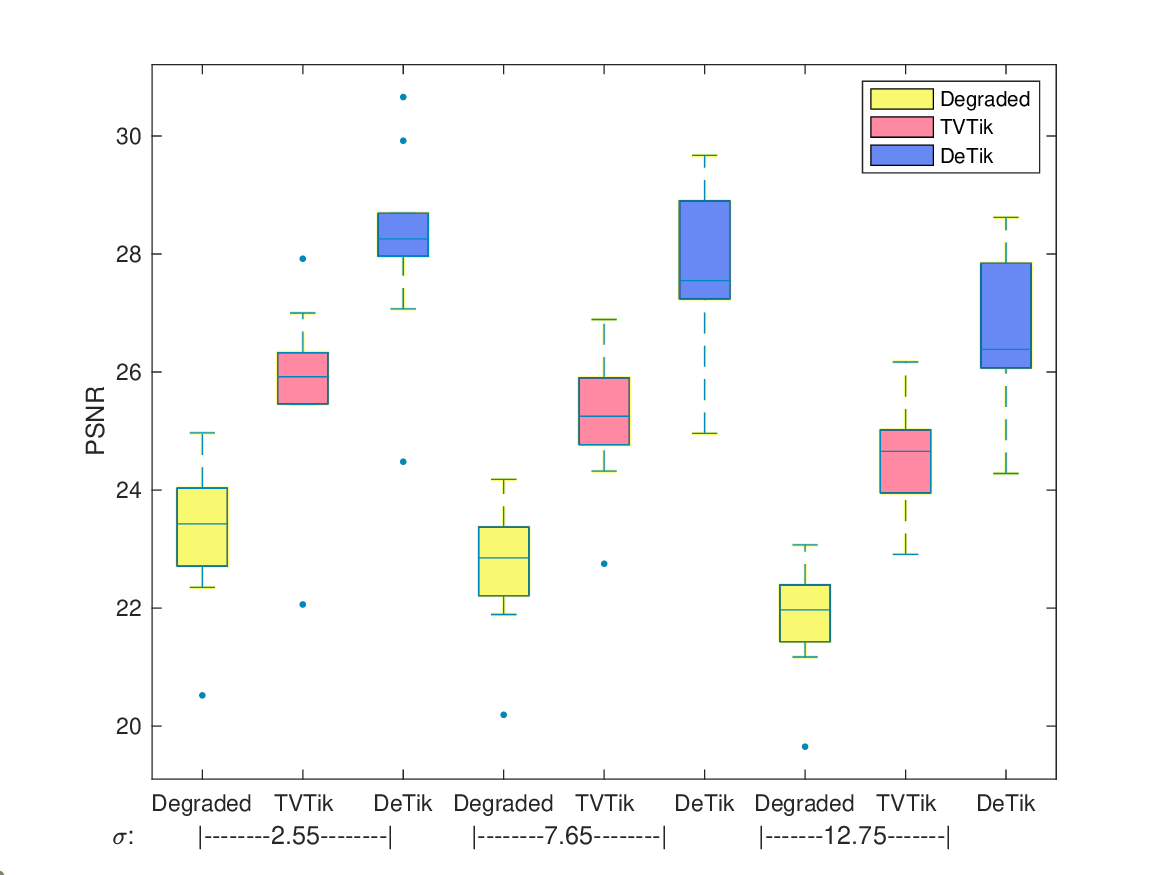}
		}\vspace{-0.03in}
		\centerline{\footnotesize{\y Set5 ($\times 3$)}}
	\end{minipage} \hspace{0.1in}
   \begin{minipage}{0.314\linewidth}
		\centering
		\centerline{\includegraphics[height=1.51in,width=2.3in]{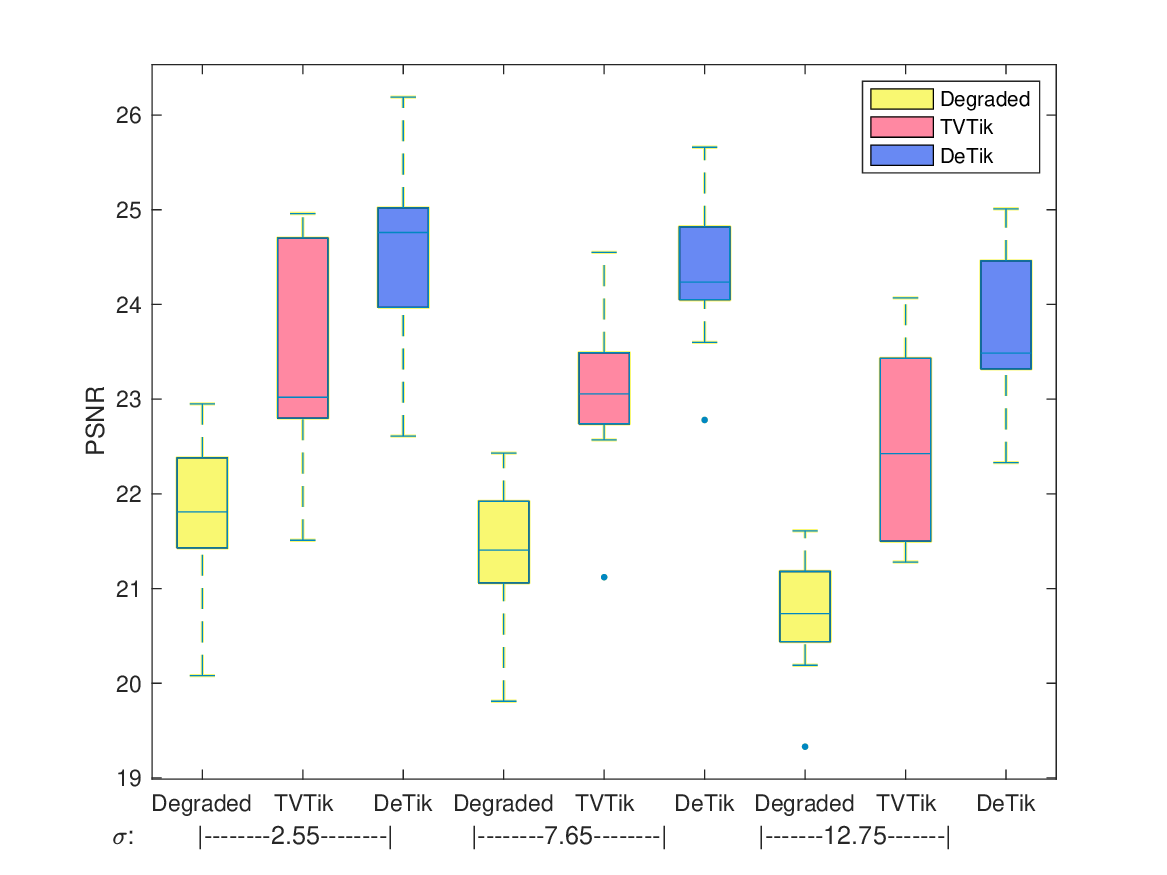}
		}\vspace{-0.03in}
		\centerline{\footnotesize{\y CBSD68 ($\times 3$)}} 
	\end{minipage}\hspace{0.1in}
  \begin{minipage}{0.314\linewidth}
		\centering
		\centerline{\includegraphics[height=1.51in,width=2.3in]{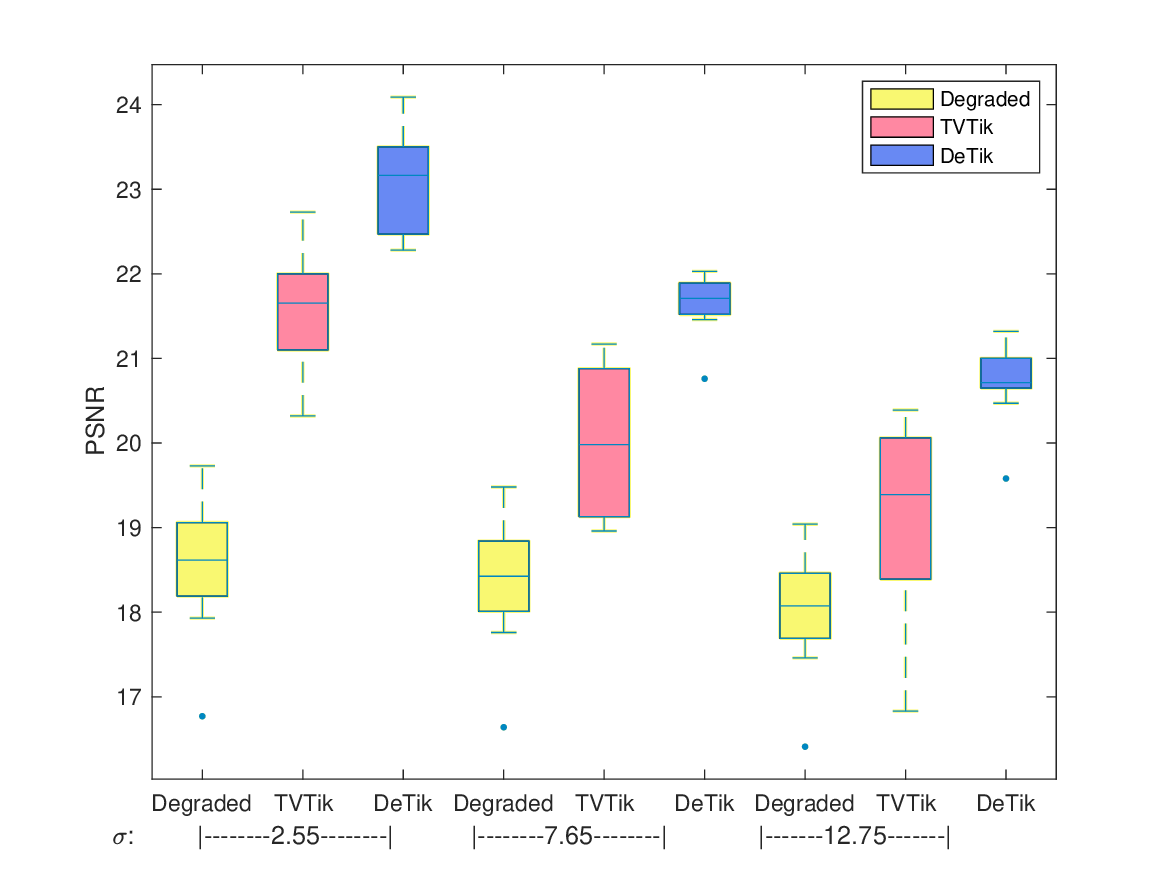}
		}\vspace{-0.03in}
		\centerline{\footnotesize{\y Urban100 ($\times 3$)}} 
	\end{minipage}
		\label{table: SR}
  \caption{\y Average results (PSNR(dB)) of TVTik and DeTik for image super-resolution with 2 scales ($\times 2$ and $\times 3$), 10 different blur kernels and 3 noise levels on   Set5, CBSD68, and Urban100 datasets. }
\end{figure}

\begin{figure}[t!]
\centering
\begin{minipage}{0.38\linewidth}
		\centering
		\centerline{\includegraphics[height=1.51in,width=2.5in]{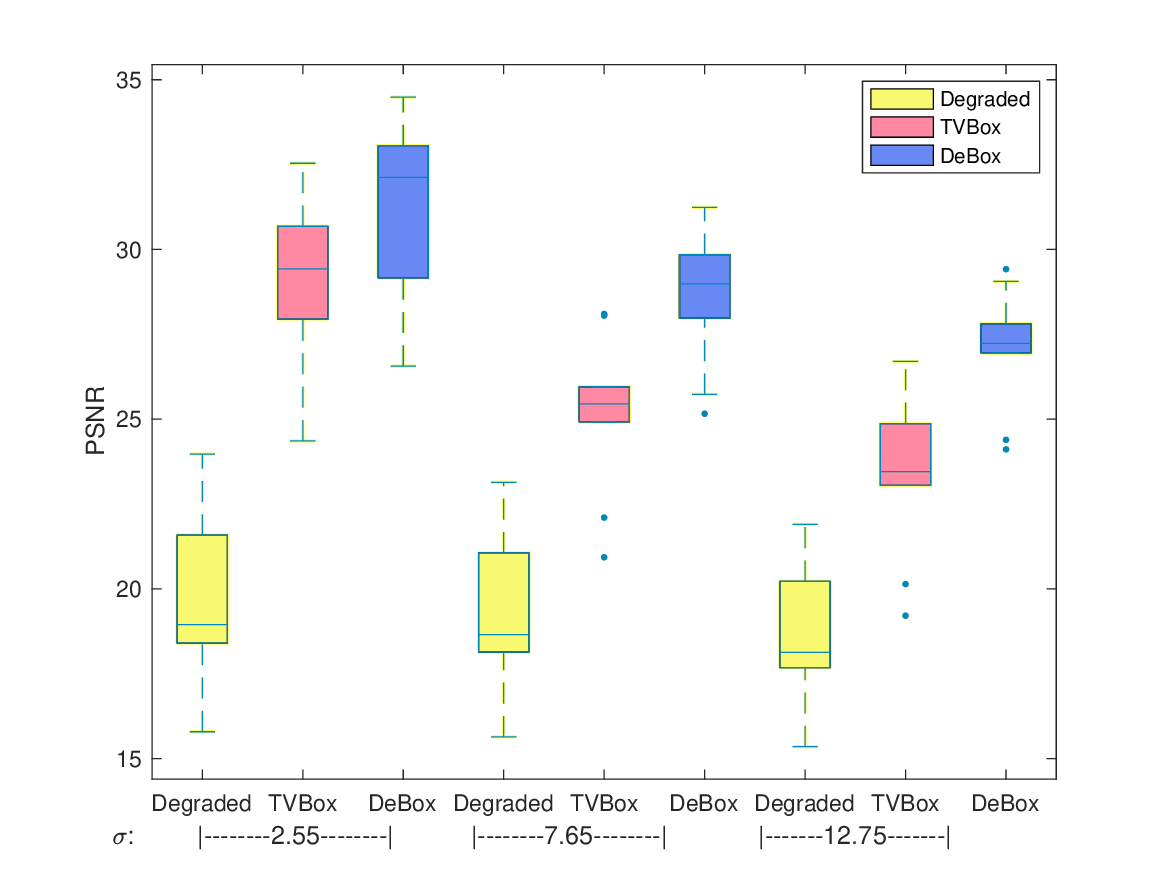}
		}\vspace{-0.03in}
		\centerline{\footnotesize{\y Set3C}}
	\end{minipage} 
 \begin{minipage}{0.38\linewidth}
		\centering
		\centerline{\includegraphics[height=1.51in,width=2.5in]{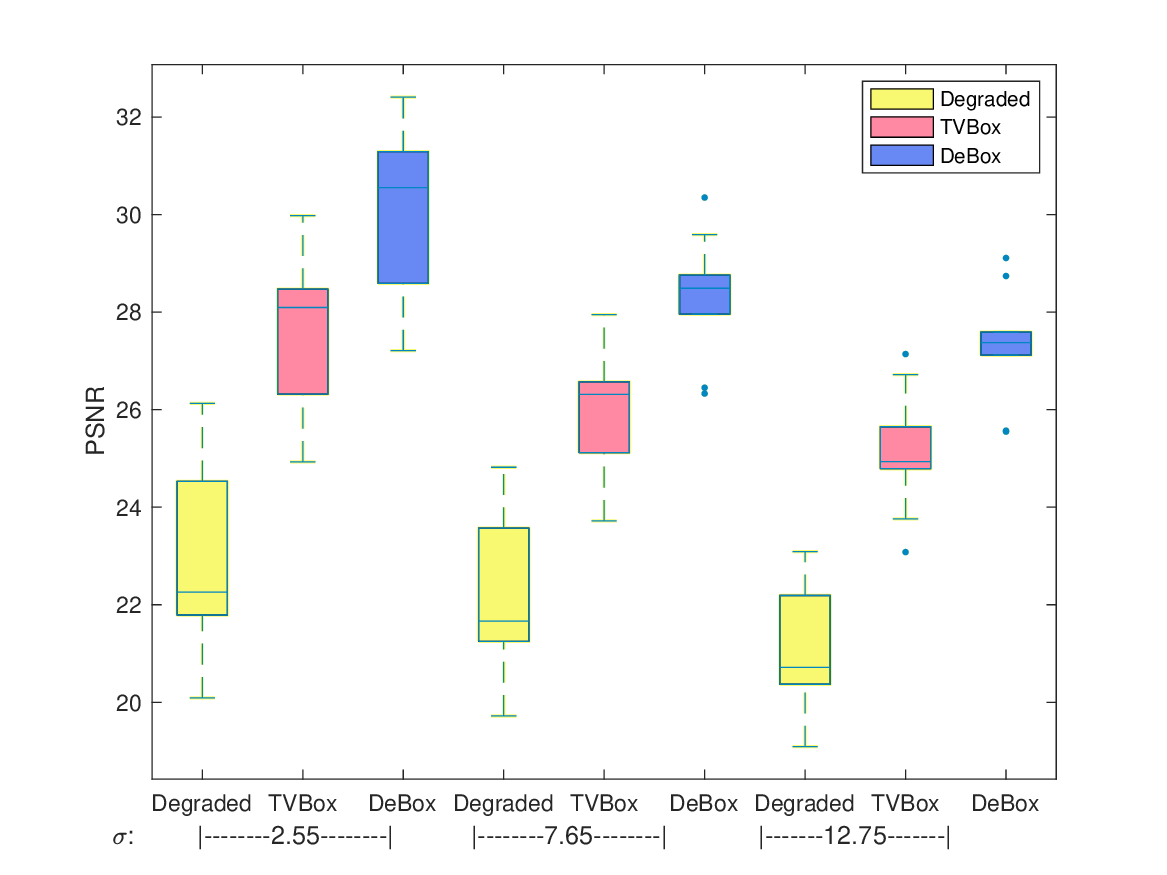}
		}\vspace{-0.03in}
		\centerline{\footnotesize{\y Set14}} 
	\end{minipage}

 \hspace{0.05in}
  \begin{minipage}{0.38\linewidth}
		\centering
		\centerline{\includegraphics[height=1.51in,width=2.5in]{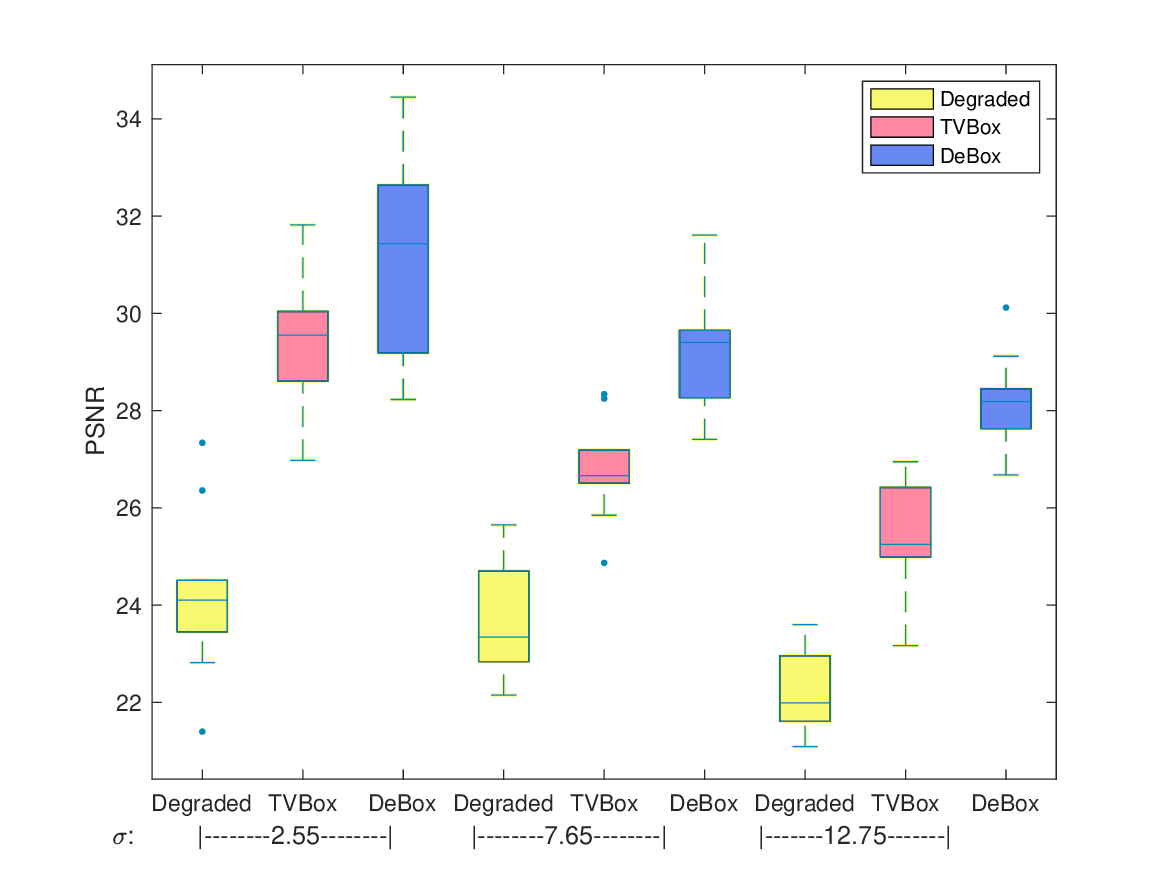}
		}\vspace{-0.03in}
		\centerline{\footnotesize{\y Kodak24}} 
	\end{minipage}
  \begin{minipage}{0.38\linewidth}
		\centering
		\centerline{\includegraphics[height=1.51in,width=2.5in]{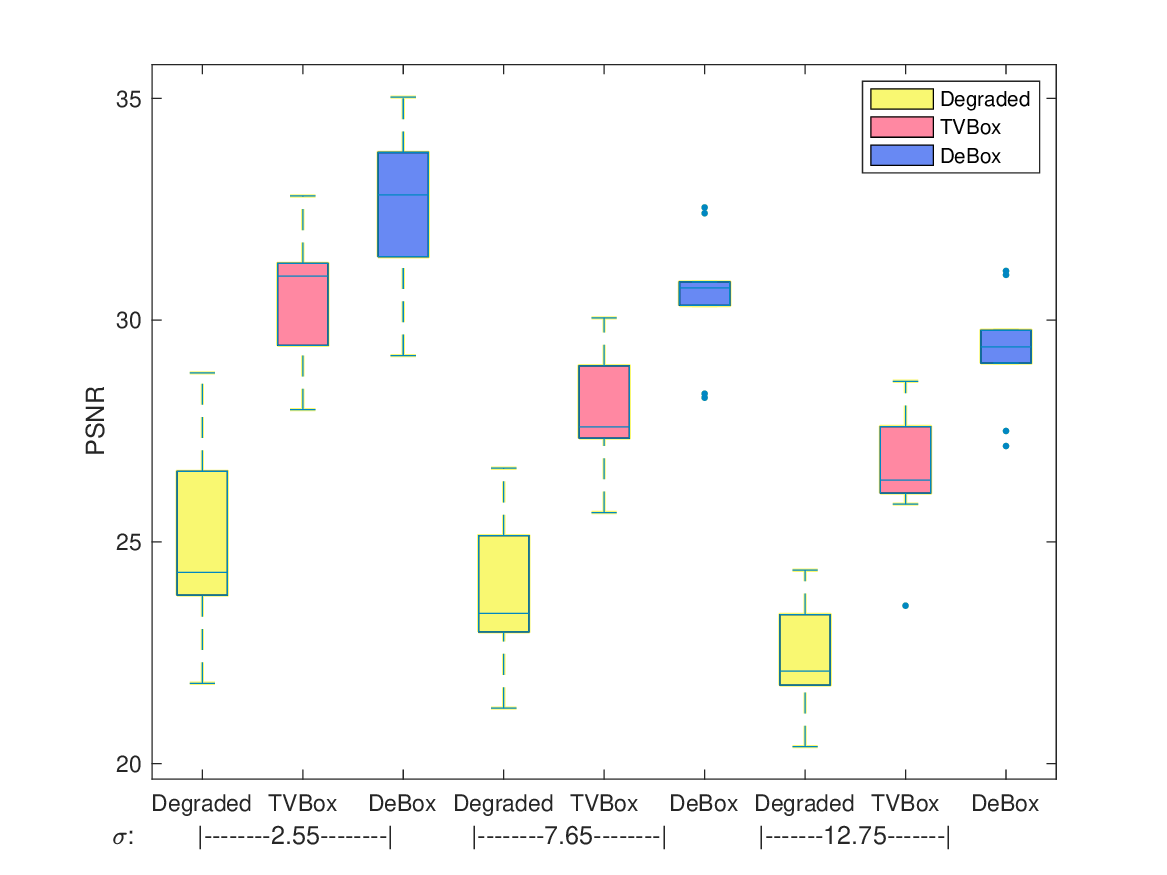}
		}\vspace{-0.03in}
		\centerline{\footnotesize{\y Set17}} 
	\end{minipage}
		\label{table: deblur_2}
  \caption{\y Average results (PSNR(dB)) of TVBox and DeBox for image deblurring with 10 different blur kernels and 3 noise levels on   Set3C, Set14, Kodak24, and Set17 datasets. }
\end{figure}

\begin{figure}[b!]
\begin{minipage}{0.314\linewidth}
		\centering
		\centerline{\includegraphics[height=1.51in,width=2.2in]{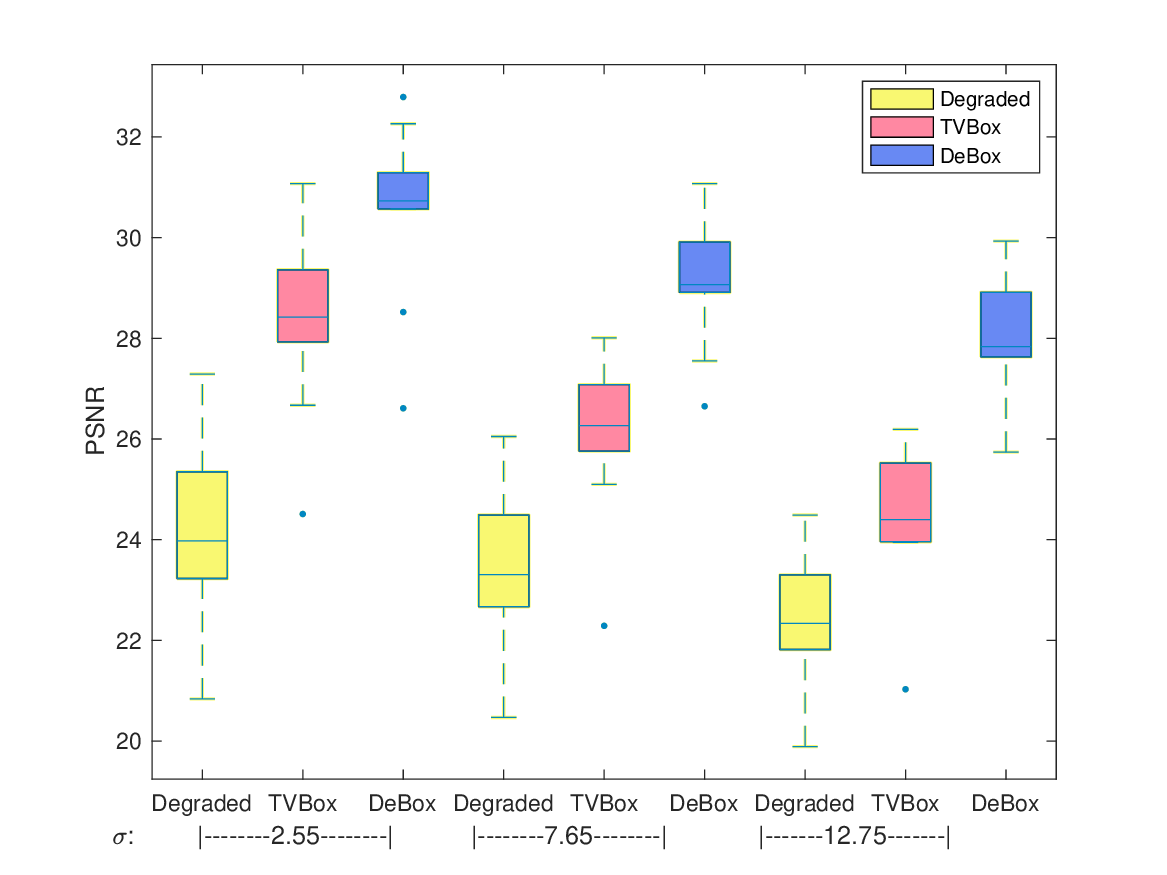}
		}\vspace{-0.03in}
		\centerline{\footnotesize{\y Set5 ($\times 2$)}}
	\end{minipage} \hspace{0.1in}
  \begin{minipage}{0.314\linewidth}
		\centering
		\centerline{\includegraphics[height=1.51in,width=2.2in]{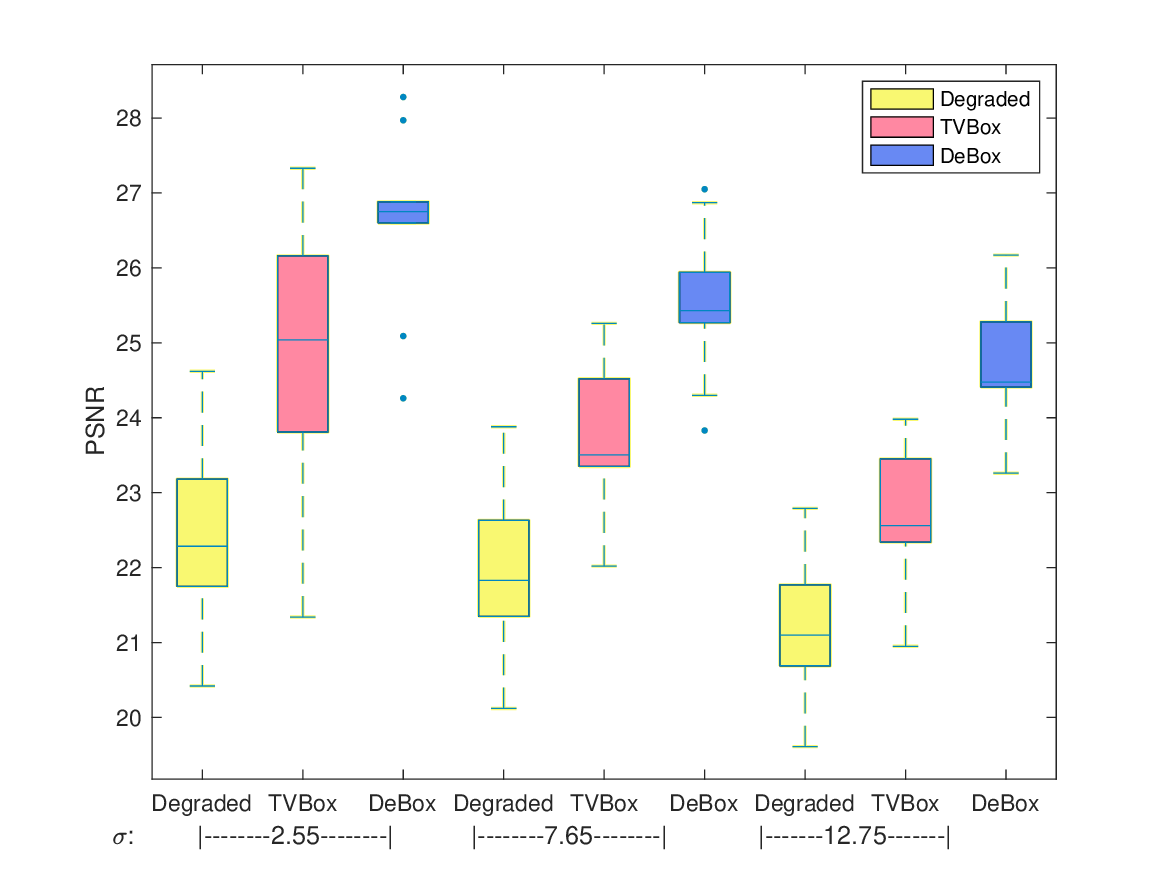}
		}\vspace{-0.03in}
		\centerline{\footnotesize{\y CBSD68 ($\times 2$)}} 
	\end{minipage}\hspace{0.1in}
   \begin{minipage}{0.314\linewidth}
		\centering
		\centerline{\includegraphics[height=1.51in,width=2.2in]{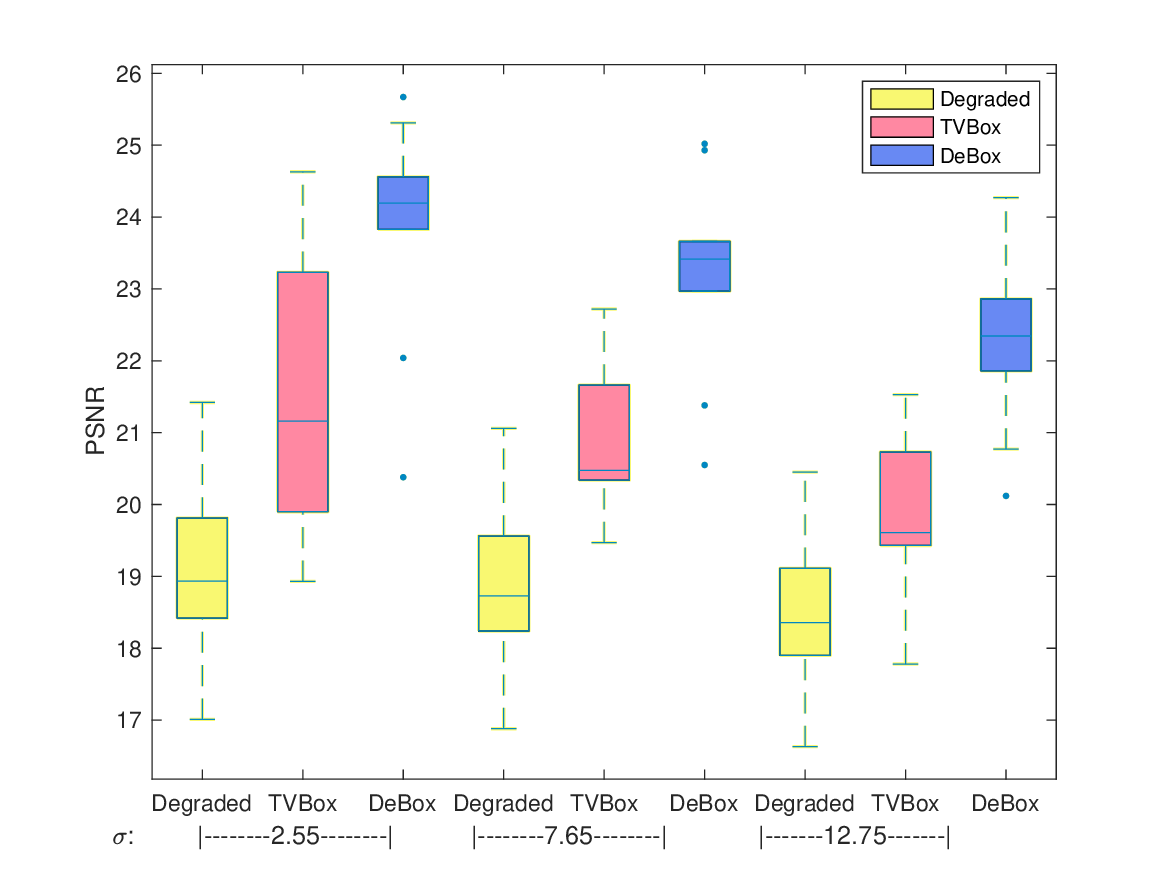}
		}\vspace{-0.03in}
		\centerline{\footnotesize{\y Urban100 ($\times 2$)}} 
	\end{minipage}
 
  \begin{minipage}{0.314\linewidth}
		\centering
		\centerline{\includegraphics[height=1.51in,width=2.2in]{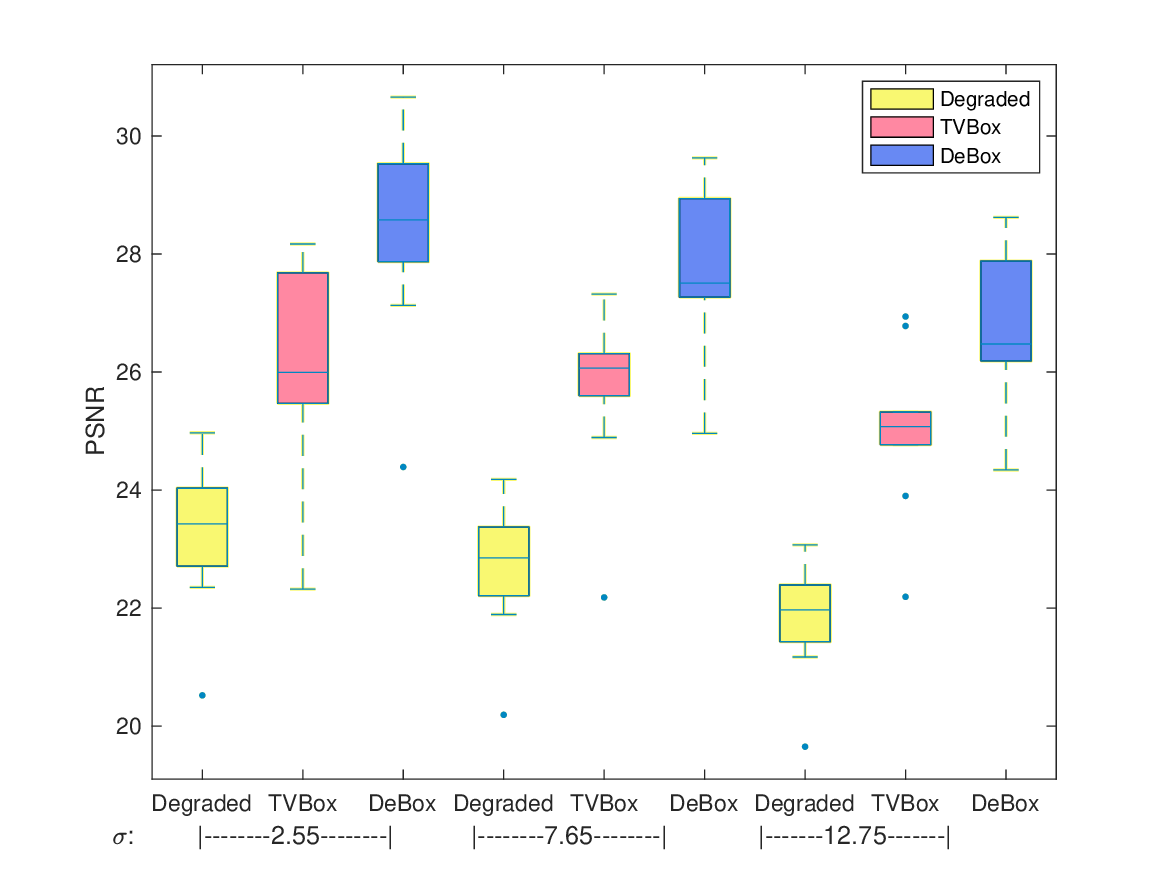}
		}\vspace{-0.03in}
		\centerline{\footnotesize{\y Set5 ($\times 3$)}}
	\end{minipage} \hspace{0.1in}
  \begin{minipage}{0.314\linewidth}
		\centering
		\centerline{\includegraphics[height=1.51in,width=2.2in]{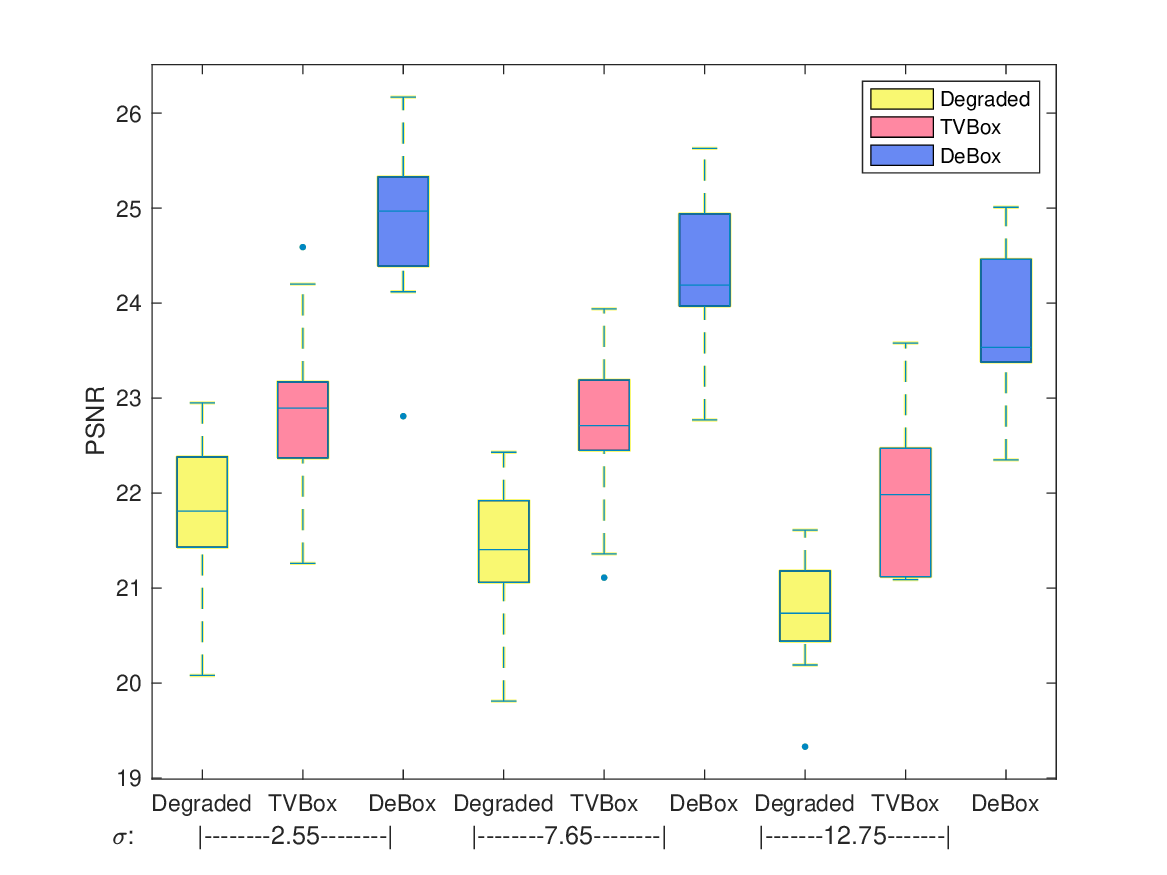}
		}\vspace{-0.03in}
		\centerline{\footnotesize{\y CBSD68 ($\times 3$)}} 
	\end{minipage}\hspace{0.1in}
  \begin{minipage}{0.314\linewidth}
		\centering
		\centerline{\includegraphics[height=1.51in,width=2.2in]{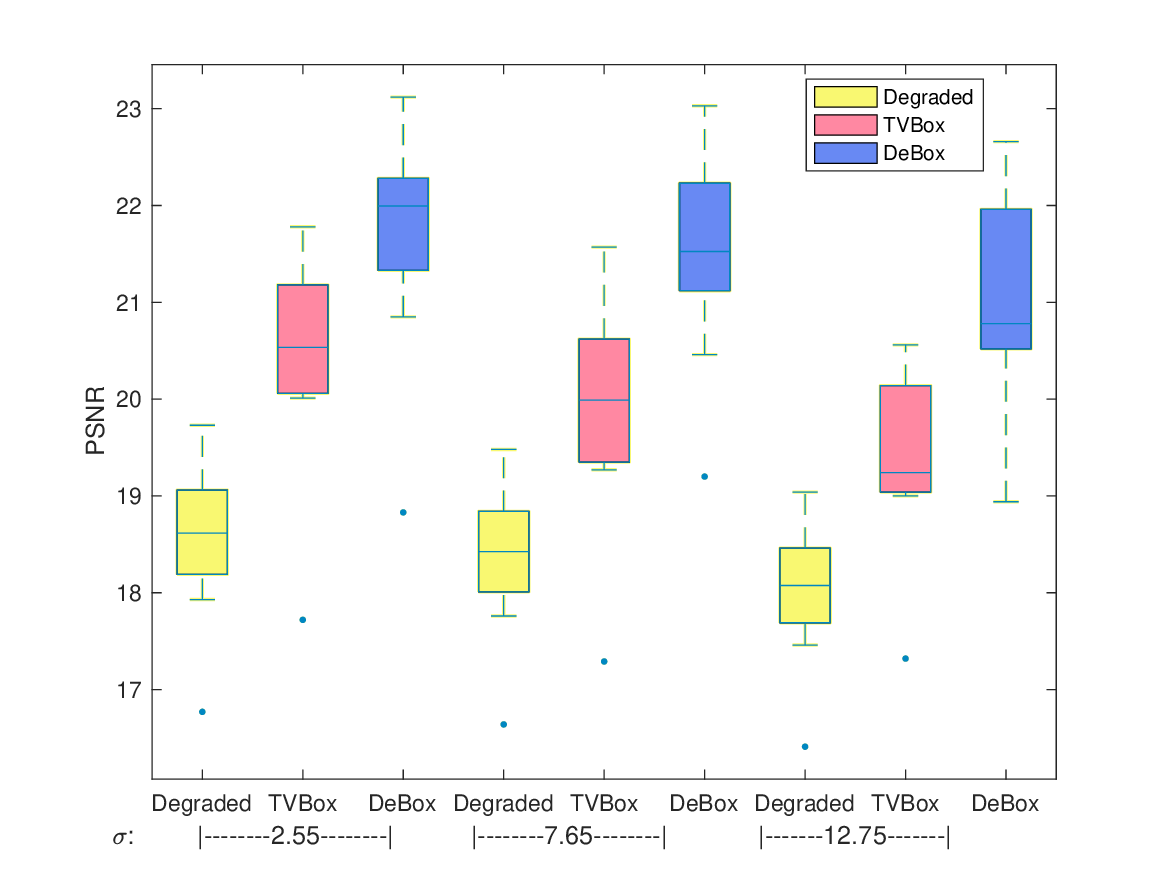}
		}\vspace{-0.03in}
		\centerline{\footnotesize{\y Urban100 ($\times 3$)}} 
	\end{minipage}
		\label{table: SR_2}
\vspace{-0.3cm}  \caption{\y Average numerical results (PSNR(dB)) of TVBox and DeBox for image super-resolution with 2 scales ($\times 2$ and $\times 3$), 10 different blur kernels and 3 noise levels on Set5, CBSD68, and Urban100 datasets. }
\end{figure}

\section{Experimental results on robustness of \cref{alg:buildtree} and \cref{alg:3}}\label{app_E}
To further demonstrate the effectiveness of the proposed methods, we compare the results recovered by the model TVTik and DeTik in \cref{table: deblur} and \cref{table: SR}, and the model DeBox and TVBox in \cref{table: deblur_2} and \cref{table: SR_2}, for image deblurring and super-resolution, respectively. 

We use the Matlab built-in function `boxplot' to create a box plot. 
As shown in \cref{table: deblur}, each picture contains 9 boxes. The yellow, pink, and blue boxes represent the average PSNR values of the degraded images, the images restored by TVTik and DeTik, and the first, second, and third sets of yellow, pink, and blue boxes correspond to the noise levels of $2.55$, $7.65$, and $12.75$, respectively.
On each box, the central mark indicates the median, and the bottom and top edges of the box indicate the $25$th and $75$th percentiles. 
The whiskers extend to the most extreme data points not considered outliers, and the outliers are plotted individually using the dot symbol.  From the box plot, we can see that the median of DeTik is higher than that of TVTik. Note that the TVTik model also enhances the quality of the degraded image when compared to the yellow boxes. 
These results demonstrate that \cref{alg:buildtree} is efficient in image restoration, as it successfully restores images affected by 10 different kernels and 3 different noise levels.
Similarly to the deblurring results, the box plot is presented to show the super-resolution outcomes. The first and second rows of the box plot display the results of super-resolution under degradation with scale factors $\times 2$ and $\times 3$, respectively.
The results presented in  \cref{table: SR} also demonstrate that the proposed algorithm effectively solves the tested models, and DeTik outperforms TVTik in terms of recovery quality for image super-resolution.

For different noise levels and blur kernels, the average image restoration results of Set3C, Set14, Kodak24, and Set17 with box plot are demonstrated in \cref{table: deblur_2}. The yellow, pink, and blue boxes denote the average PSNR of the degraded images, the image restored by TVBox and DeBox. 
The first, second, and third sets of yellow, pink, and blue boxes correspond to the noise levels of $2.55$, $7.65$, and $12.75$, respectively. 
Similarly, the super-resolution results for two scale factors, $\times 2$ and $\times 3$, are presented in \cref{table: SR_2}.
The result demonstrates that the proposed method exhibits consistent and stable image restoration performance. 
From \cref{table: deblur_2} and \cref{table: SR_2},
we can see that \cref{alg:3} effectively solves the DeBox model, and DeBox
outperforms TVBox in terms of recovery quality for both image deblurring and super-resolution tasks. 
The experiment results also demonstrate that \cref{alg:3} can handle the minimization with the non-differentiable term.

\section*{Acknowledgement}
The authors are grateful to the anonymous referees for their valuable comments, which largely improve the quality of this paper.

}

\bibliographystyle{siamplain}
\bibliography{references}

\begin{thebibliography}{10}

\bibitem{ahookhosh2021bregman}
{\sc M.~Ahookhosh, A.~Themelis, and P.~Patrinos}, {\em A {B}regman
  forward-backward linesearch algorithm for nonconvex composite optimization:
  superlinear convergence to nonisolated local minima}, SIAM Journal on
  Optimization, 31 (2021), pp.~653--685.

\bibitem{attouch2010proximal}
{\sc H.~Attouch, J.~Bolte, P.~Redont, and A.~Soubeyran}, {\em Proximal
  alternating minimization and projection methods for nonconvex problems: An
  approach based on the {K}urdyka-{{\L}}ojasiewicz inequality}, Mathematics of
  Operations Research, 35 (2010), pp.~438--457.

\bibitem{attouch2013convergence}
{\sc H.~Attouch, J.~Bolte, and B.~F. Svaiter}, {\em Convergence of descent
  methods for semi-algebraic and tame problems: proximal algorithms,
  forward--backward splitting, and regularized {G}auss--{S}eidel methods},
  Mathematical Programming, 137 (2013), pp.~91--129.

\bibitem{attouch2014dynamical}
{\sc H.~Attouch, J.~Peypouquet, and P.~Redont}, {\em A dynamical approach to an
  inertial forward-backward algorithm for convex minimization}, SIAM Journal on
  Optimization, 24 (2014), pp.~232--256.

\bibitem{barakat2020convergence}
{\sc A.~Barakat and P.~Bianchi}, {\em Convergence rates of a momentum algorithm
  with bounded adaptive step size for nonconvex optimization}, in Asian
  Conference on Machine Learning, PMLR, 2020, pp.~225--240.

\bibitem{beck2009fast}
{\sc A.~Beck and M.~Teboulle}, {\em A fast iterative shrinkage-thresholding
  algorithm for linear inverse problems}, SIAM Journal on Imaging Sciences, 2
  (2009), pp.~183--202.

\bibitem{bian2021three}
{\sc F.~Bian and X.~Zhang}, {\em A three-operator splitting algorithm for
  nonconvex sparsity regularization}, SIAM Journal on Scientific Computing, 43
  (2021), pp.~2809--2839.

\bibitem{bolte2014alternating}
{\sc J.~Bolte, S.~Sabach, and M.~Teboulle}, {\em Proximal alternating
  linearized minimization for nonconvex and nonsmooth problems}, Mathematical
  Programming, 146 (2014), pp.~459--494.

\bibitem{boct2016inertial}
{\sc R.~I. Bo{\c{t}}, E.~R. Csetnek, and S.~C. L{\'a}szl{\'o}}, {\em An
  inertial forward--backward algorithm for the minimization of the sum of two
  nonconvex functions}, EURO Journal on Computational Optimization, 4 (2016),
  pp.~3--25.

\bibitem{buzzard2018plug}
{\sc G.~T. Buzzard, S.~H. Chan, S.~Sreehari, and C.~A. Bouman}, {\em
  Plug-and-play unplugged: Optimization-free reconstruction using consensus
  equilibrium}, SIAM Journal on Imaging Sciences, 11 (2018), pp.~2001--2020.

\bibitem{castera2021inertial}
{\sc C.~Castera, J.~Bolte, C.~F{\'e}votte, and E.~Pauwels}, {\em An inertial
  newton algorithm for deep learning}, The Journal of Machine Learning
  Research, 22 (2021), pp.~5977--6007.

\bibitem{chen2015general}
{\sc C.~Chen, S.~Ma, and J.~Yang}, {\em A general inertial proximal point
  algorithm for mixed variational inequality problem}, SIAM Journal on
  Optimization, 25 (2015), pp.~2120--2142.

\bibitem{cohen2021has}
{\sc R.~Cohen, Y.~Blau, D.~Freedman, and E.~Rivlin}, {\em It has potential:
  Gradient-driven denoisers for convergent solutions to inverse problems},
  Advances in Neural Information Processing Systems, 34 (2021),
  pp.~18152--18164.

\bibitem{combettes2021fixed}
{\sc P.~L. Combettes and J.-C. Pesquet}, {\em Fixed point strategies in data
  science}, IEEE Transactions on Signal Processing, 69 (2021), pp.~3878--3905.

\bibitem{condat2023proximal}
{\sc L.~Condat, D.~Kitahara, A.~Contreras, and A.~Hirabayashi}, {\em Proximal
  splitting algorithms for convex optimization: A tour of recent advances, with
  new twists}, SIAM Review, 65 (2023), pp.~375--435.

\bibitem{davis2017three}
{\sc D.~Davis and W.~Yin}, {\em A three-operator splitting scheme and its
  optimization applications}, Set-Valued and Variational Analysis, 25 (2017),
  pp.~829--858.

\bibitem{deng2019new}
{\sc L.-J. Deng, R.~Glowinski, and X.-C. Tai}, {\em A new operator splitting
  method for the {E}uler elastica model for image smoothing}, SIAM Journal on
  Imaging Sciences, 12 (2019), pp.~1190--1230.

\bibitem{dong2020deep}
{\sc J.~Dong, S.~Roth, and B.~Schiele}, {\em Deep wiener deconvolution: Wiener
  meets deep learning for image deblurring}, Advances in Neural Information
  Processing Systems, 33 (2020), pp.~1048--1059.

\bibitem{gavaskar2021plug}
{\sc R.~G. Gavaskar, C.~D. Athalye, and K.~N. Chaudhury}, {\em On plug-and-play
  regularization using linear denoisers}, IEEE Transactions on Image
  Processing, 30 (2021), pp.~4802--4813.

\bibitem{goldstein2009split}
{\sc T.~Goldstein and S.~Osher}, {\em The split bregman method for
  l1-regularized problems}, SIAM Journal on Imaging Sciences, 2 (2009),
  pp.~323--343.

\bibitem{guo2018note}
{\sc K.~Guo and D.~Han}, {\em A note on the {D}ouglas--{R}achford splitting
  method for optimization problems involving hypoconvex functions}, Journal of
  Global Optimization, 72 (2018), pp.~431--441.

\bibitem{guo2017convergence}
{\sc K.~Guo, D.~Han, and X.~Yuan}, {\em Convergence analysis of
  {D}ouglas--{R}achford splitting method for ``strongly+ weakly” convex
  programming}, SIAM Journal on Numerical Analysis, 55 (2017), pp.~1549--1577.

\bibitem{han2022survey}
{\sc D.~Han}, {\em A survey on some recent developments of alternating
  direction method of multipliers}, Journal of the Operations Research Society
  of China,  (2022), pp.~1--52.

\bibitem{hertrich2021convolutional}
{\sc J.~Hertrich, S.~Neumayer, and G.~Steidl}, {\em Convolutional proximal
  neural networks and plug-and-play algorithms}, Linear Algebra and its
  Applications, 631 (2021), pp.~203--234.

\bibitem{hurault2023convergent}
{\sc S.~Hurault, A.~Chambolle, A.~Leclaire, and N.~Papadakis}, {\em Convergent
  {P}lug-and-{P}lay with proximal denoiser and unconstrained regularization
  parameter}, arXiv preprint arXiv:2311.01216,  (2023).

\bibitem{hurault2023convergent2}
{\sc S.~Hurault, U.~Kamilov, A.~Leclaire, and N.~Papadakis}, {\em Convergent
  {B}regman plug-and-play image restoration for {P}oisson inverse problems},
  arXiv preprint arXiv:2306.03466,  (2023).

\bibitem{hurault2022gradient}
{\sc S.~Hurault, A.~Leclaire, and N.~Papadakis}, {\em Gradient step denoiser
  for convergent plug-and-play}, in International Conference on Learning
  Representations (ICLR'22), 2022.

\bibitem{hurault2022proximal}
{\sc S.~Hurault, A.~Leclaire, and N.~Papadakis}, {\em Proximal denoiser for
  convergent plug-and-play optimization with nonconvex regularization}, in
  International Conference on Machine Learning, PMLR, 2022, pp.~9483--9505.

\bibitem{jain2017non}
{\sc P.~Jain, P.~Kar, et~al.}, {\em Non-convex optimization for machine
  learning}, Foundations and Trends{\textregistered} in Machine Learning, 10
  (2017), pp.~142--363.

\bibitem{9662672}
{\sc S.~Kong, W.~Wang, X.~Feng, and X.~Jia}, {\em Deep red unfolding network
  for image restoration}, IEEE Transactions on Image Processing, 31 (2022),
  pp.~852--867.

\bibitem{krantz2002primer}
{\sc S.~G. Krantz and H.~R. Parks}, {\em A primer of real analytic functions},
  Springer Science \& Business Media, 2002.

\bibitem{latafat2017asymmetric}
{\sc P.~Latafat and P.~Patrinos}, {\em Asymmetric forward--backward--adjoint
  splitting for solving monotone inclusions involving three operators},
  Computational Optimization and Applications, 68 (2017), pp.~57--93.

\bibitem{le2020inertial}
{\sc H.~Le, N.~Gillis, and P.~Patrinos}, {\em Inertial block proximal methods
  for non-convex non-smooth optimization}, in International Conference on
  Machine Learning, PMLR, 2020, pp.~5671--5681.

\bibitem{li2016douglas}
{\sc G.~Li and T.~K. Pong}, {\em Douglas--{R}achford splitting for nonconvex
  optimization with application to nonconvex feasibility problems},
  Mathematical Programming, 159 (2016), pp.~371--401.

\bibitem{li2023spherical}
{\sc J.~Li, C.~Huang, R.~Chan, H.~Feng, M.~K. Ng, and T.~Zeng}, {\em Spherical
  image inpainting with frame transformation and data-driven prior deep
  networks}, SIAM Journal on Imaging Sciences, 16 (2023), pp.~1179--1196.

\bibitem{li2019convergence}
{\sc M.~Li and Z.~Wu}, {\em Convergence analysis of the generalized splitting
  methods for a class of nonconvex optimization problems}, Journal of
  Optimization Theory and Applications, 183 (2019), pp.~535--565.

\bibitem{liang2016multi}
{\sc J.~Liang, J.~Fadili, and G.~Peyr{\'e}}, {\em A multi-step inertial
  forward-backward splitting method for non-convex optimization}, Advances in
  Neural Information Processing Systems, 29 (2016).

\bibitem{liang2017activity}
{\sc J.~Liang, J.~Fadili, and G.~Peyr{\'e}}, {\em Activity identification and
  local linear convergence of forward--backward-type methods}, SIAM Journal on
  Optimization, 27 (2017), pp.~408--437.

\bibitem{lindstrom2021survey}
{\sc S.~B. Lindstrom and B.~Sims}, {\em Survey: sixty years of
  {D}ouglas--{R}achford}, Journal of the Australian Mathematical Society, 110
  (2021), pp.~333--370.

\bibitem{liu2022operator}
{\sc H.~Liu, X.-C. Tai, and R.~Glowinski}, {\em An operator-splitting method
  for the gaussian curvature regularization model with applications to surface
  smoothing and imaging}, SIAM Journal on Scientific Computing, 44 (2022),
  pp.~A935--A963.

\bibitem{liu2021recovery}
{\sc J.~Liu, S.~Asif, B.~Wohlberg, and U.~Kamilov}, {\em Recovery analysis for
  plug-and-play priors using the restricted eigenvalue condition}, Advances in
  Neural Information Processing Systems, 34 (2021), pp.~5921--5933.

\bibitem{liu2019envelope}
{\sc Y.~Liu and W.~Yin}, {\em An envelope for {D}avis-{Y}in splitting and
  strict saddle-point avoidance}, Journal of Optimization Theory and
  Applications, 181 (2019), pp.~567--587.

\bibitem{lorenz2015inertial}
{\sc D.~A. Lorenz and T.~Pock}, {\em An inertial forward-backward algorithm for
  monotone inclusions}, Journal of Mathematical Imaging and Vision, 51 (2015),
  pp.~311--325.

\bibitem{nesterov2004}
{\sc Y.~Nesterov}, {\em Introductory lectures on convex optimization: A basic
  course}, vol.~87, Springer Science \& Business Media, 2003.

\bibitem{ochs2014ipiano}
{\sc P.~Ochs, Y.~Chen, T.~Brox, and T.~Pock}, {\em ipiano: Inertial proximal
  algorithm for nonconvex optimization}, SIAM Journal on Imaging Sciences, 7
  (2014), pp.~1388--1419.

\bibitem{ono2017primal}
{\sc S.~Ono}, {\em Primal-dual plug-and-play image restoration}, IEEE Signal
  Processing Letters, 24 (2017), pp.~1108--1112.

\bibitem{pesquet2021learning}
{\sc J.-C. Pesquet, A.~Repetti, M.~Terris, and Y.~Wiaux}, {\em Learning
  maximally monotone operators for image recovery}, SIAM Journal on Imaging
  Sciences, 14 (2021), pp.~1206--1237.

\bibitem{phan2023inertial}
{\sc D.~N. Phan and N.~Gillis}, {\em An inertial block majorization
  minimization framework for nonsmooth nonconvex optimization}, Journal of
  Machine Learning Research, 24 (2023), pp.~1--41.

\bibitem{pock2016inertial}
{\sc T.~Pock and S.~Sabach}, {\em Inertial proximal alternating linearized
  minimization (i{PALM}) for nonconvex and nonsmooth problems}, SIAM Journal on
  Imaging Sciences, 9 (2016), pp.~1756--1787.

\bibitem{polyak1964some}
{\sc B.~T. Polyak}, {\em Some methods of speeding up the convergence of
  iteration methods}, Ussr Computational Mathematics and Mathematical Physics,
  4 (1964), pp.~1--17.

\bibitem{raguet2013generalized}
{\sc H.~Raguet, J.~Fadili, and G.~Peyr{\'e}}, {\em A generalized
  forward-backward splitting}, SIAM Journal on Imaging Sciences, 6 (2013),
  pp.~1199--1226.

\bibitem{reehorst2018regularization}
{\sc E.~T. Reehorst and P.~Schniter}, {\em Regularization by denoising:
  Clarifications and new interpretations}, IEEE Transactions on Computational
  Imaging, 5 (2018), pp.~52--67.

\bibitem{rockafellar2009variational}
{\sc R.~T. Rockafellar and R.~J.-B. Wets}, {\em Variational analysis},
  vol.~317, Springer Science \& Business Media, 2009.

\bibitem{rudin1992nonlinear}
{\sc L.~I. Rudin, S.~Osher, and E.~Fatemi}, {\em Nonlinear total variation
  based noise removal algorithms}, Physica D: Nonlinear Phenomena, 60 (1992),
  pp.~259--268.

\bibitem{ryu2019plug}
{\sc E.~Ryu, J.~Liu, S.~Wang, X.~Chen, Z.~Wang, and W.~Yin}, {\em Plug-and-play
  methods provably converge with properly trained denoisers}, International
  Conference on Machine Learning,  (2019), pp.~5546--5557.

\bibitem{ryu2020operator}
{\sc E.~K. Ryu, A.~B. Taylor, C.~Bergeling, and P.~Giselsson}, {\em Operator
  splitting performance estimation: Tight contraction factors and optimal
  parameter selection}, SIAM Journal on Optimization, 30 (2020),
  pp.~2251--2271.

\bibitem{salim2022dualize}
{\sc A.~Salim, L.~Condat, K.~Mishchenko, and P.~Richt{\'a}rik}, {\em Dualize,
  split, randomize: Toward fast nonsmooth optimization algorithms}, Journal of
  Optimization Theory and Applications, 195 (2022), pp.~102--130.

\bibitem{setzer2011operator}
{\sc S.~Setzer}, {\em Operator splittings, bregman methods and frame shrinkage
  in image processing}, International Journal of Computer Vision, 92 (2011),
  pp.~265--280.

\bibitem{sreehari2016plug}
{\sc S.~Sreehari, S.~V. Venkatakrishnan, B.~Wohlberg, G.~T. Buzzard, L.~F.
  Drummy, J.~P. Simmons, and C.~A. Bouman}, {\em Plug-and-play priors for
  bright field electron tomography and sparse interpolation}, IEEE Transactions
  on Computational Imaging, 2 (2016), pp.~408--423.

\bibitem{sun2019online}
{\sc Y.~Sun, B.~Wohlberg, and U.~S. Kamilov}, {\em An online plug-and-play
  algorithm for regularized image reconstruction}, IEEE Transactions on
  Computational Imaging, 5 (2019), pp.~395--408.

\bibitem{sun2021scalable}
{\sc Y.~Sun, Z.~Wu, X.~Xu, B.~Wohlberg, and U.~S. Kamilov}, {\em Scalable
  plug-and-play {ADMM} with convergence guarantees}, IEEE Transactions on
  Computational Imaging, 7 (2021), pp.~849--863.

\bibitem{tang2022preconditioned}
{\sc Y.~Tang, M.~Wen, and T.~Zeng}, {\em Preconditioned three-operator
  splitting algorithm with applications to image restoration}, Journal of
  Scientific Computing, 92 (2022), pp.~1--26.

\bibitem{themelis2020douglas}
{\sc A.~Themelis and P.~Patrinos}, {\em Douglas--{R}achford splitting and
  {ADMM} for nonconvex optimization: Tight convergence results}, SIAM Journal
  on Optimization, 30 (2020), pp.~149--181.

\bibitem{themelis2018forward}
{\sc A.~Themelis, L.~Stella, and P.~Patrinos}, {\em Forward-backward envelope
  for the sum of two nonconvex functions: Further properties and nonmonotone
  linesearch algorithms}, SIAM Journal on Optimization, 28 (2018),
  pp.~2274--2303.

\bibitem{themelis2022douglas}
{\sc A.~Themelis, L.~Stella, and P.~Patrinos}, {\em Douglas--{R}achford
  splitting and {ADMM} for nonconvex optimization: accelerated and newton-type
  linesearch algorithms}, Computational Optimization and Applications, 82
  (2022), pp.~395--440.

\bibitem{tirer2018image}
{\sc T.~Tirer and R.~Giryes}, {\em Image restoration by iterative denoising and
  backward projections}, IEEE Transactions on Image Processing, 28 (2018),
  pp.~1220--1234.

\bibitem{venkatakrishnan2013plug}
{\sc S.~V. Venkatakrishnan, C.~A. Bouman, and B.~Wohlberg}, {\em Plug-and-play
  priors for model based reconstruction}, in 2013 IEEE Global Conference on
  Signal and Information Processing, IEEE, 2013, pp.~945--948.

\bibitem{villa2013accelerated}
{\sc S.~Villa, S.~Salzo, L.~Baldassarre, and A.~Verri}, {\em Accelerated and
  inexact forward-backward algorithms}, SIAM Journal on Optimization, 23
  (2013), pp.~1607--1633.

\bibitem{wang2023generalized}
{\sc Q.~Wang and D.~Han}, {\em A generalized inertial proximal alternating
  linearized minimization method for nonconvex nonsmooth problems}, Applied
  Numerical Mathematics, 189 (2023), pp.~66--87.

\bibitem{wei2022tfpnp}
{\sc K.~Wei, A.~Aviles-Rivero, J.~Liang, Y.~Fu, H.~Huang, and C.-B.
  Sch{\"o}nlieb}, {\em Tfpnp: Tuning-free plug-and-play proximal algorithms
  with applications to inverse imaging problems}, The Journal of Machine
  Learning Research, 23 (2022), pp.~699--746.

\bibitem{wu2023retinex}
{\sc T.~Wu, W.~Wu, Y.~Yang, F.-L. Fan, and T.~Zeng}, {\em Retinex image
  enhancement based on sequential decomposition with a plug-and-play
  framework}, IEEE Transactions on Neural Networks and Learning Systems,
  (2023), pp.~1--14.

\bibitem{wu2021inertial}
{\sc Z.~Wu, C.~Li, M.~Li, and A.~Lim}, {\em Inertial proximal gradient methods
  with {B}regman regularization for a class of nonconvex optimization
  problems}, Journal of Global Optimization, 79 (2021), pp.~617--644.

\bibitem{wu2019general}
{\sc Z.~Wu and M.~Li}, {\em General inertial proximal gradient method for a
  class of nonconvex nonsmooth optimization problems}, Computational
  Optimization and Applications, 73 (2019), pp.~129--158.

\bibitem{yang2011alternating}
{\sc J.~Yang and Y.~Zhang}, {\em Alternating direction algorithms for
  ${L}_1$-problems in compressive sensing}, SIAM Journal on Scientific
  Computing, 33 (2011), pp.~250--278.

\bibitem{yin2015minimization}
{\sc P.~Yin, Y.~Lou, Q.~He, and J.~Xin}, {\em Minimization of ${L}_1$-${L}_2$
  for compressed sensing}, SIAM Journal on Scientific Computing, 37 (2015),
  pp.~536--583.

\bibitem{yurtsever2021three}
{\sc A.~Yurtsever, V.~Mangalick, and S.~Sra}, {\em Three operator splitting
  with a nonconvex loss function}, in International Conference on Machine
  Learning, PMLR, 2021, pp.~12267--12277.

\bibitem{zeng2019global}
{\sc J.~Zeng, T.~T.-K. Lau, S.~Lin, and Y.~Yao}, {\em Global convergence of
  block coordinate descent in deep learning}, in International Conference on
  Machine Learning, PMLR, 2019, pp.~7313--7323.

\bibitem{zhang2021plug}
{\sc K.~Zhang, Y.~Li, W.~Zuo, L.~Zhang, L.~Van~Gool, and R.~Timofte}, {\em
  Plug-and-play image restoration with deep denoiser prior}, IEEE Transactions
  on Pattern Analysis and Machine Intelligence, 44 (2021), pp.~6360--6376.

\bibitem{zhang2020deep}
{\sc K.~Zhang, L.~Van~Gool, and R.~Timofte}, {\em Deep unfolding network for
  image super-resolution}, in IEEE Conference on Computer Vision and Pattern
  Recognition, 2020, pp.~3217--3226.

\bibitem{zhang2017learning}
{\sc K.~Zhang, W.~Zuo, S.~Gu, and L.~Zhang}, {\em Learning deep cnn denoiser
  prior for image restoration}, in IEEE Conference on Computer Vision and
  Pattern Recognition, 2017, pp.~3929--3938.

\end{thebibliography}

\end{document}